\documentclass[11pt, leqno]{amsart}
\usepackage{amsmath,amssymb,amsthm,amscd,eucal}

\usepackage{latexsym}

\usepackage{color}

\usepackage[all,knot]{xy}
\usepackage{pictexwd}

\theoremstyle{definition}
\newtheorem{theorem}[equation]{Theorem}
\newtheorem{corollary}[equation]{Corollary}
\newtheorem{lemma}[equation]{Lemma}
\newtheorem{proposition}[equation]{Proposition}

\newtheorem{definition}[equation]{Definition}

\newtheorem{remark}[equation]{Remark}

\newtheorem{example}[equation]{Example}
\numberwithin{equation}{section}

\newcommand{\eps}{\xi}

\newcommand{\vphi}{\varphi}

\newcommand{\mbc}{\mathbf{c}}

\newcommand{\mbe}{\mathbf{e}}

\newcommand{\mbs}{\mathbf{s}}

\newcommand{\mcM}{\mathcal{M}}

\newcommand{\mcR}{\mathcal{R}}

\newcommand{\mcX}{\mathcal{X}}
\newcommand{\mcY}{\mathcal{Y}}

\newcommand{\mfgl}{\mathfrak{g}\mathfrak{l}}

\newcommand{\mfq}{\mathfrak{q}}

\newcommand{\mfU}{\mathfrak{U}}

\newcommand{\msA}{\mathsf{A}}
\newcommand{\msB}{\mathsf{B}}

\newcommand{\msC}{\mathsf{C}}
\newcommand{\msc}{\mathsf{c}}

\newcommand{\mse}{\mathsf{e}}
\newcommand{\msE}{\mathsf{E}}

\newcommand{\msF}{\mathsf{F}}
\newcommand{\msG}{\mathsf{G}}

\newcommand{\msH}{\mathsf{H}}

\newcommand{\msK}{\mathsf{K}}
\newcommand{\msL}{\mathsf{L}}
\newcommand{\msm}{\mathsf{m}}
\newcommand{\msM}{\mathsf{M}}

\newcommand{\msn}{\mathsf{n}}

\newcommand{\msp}{\mathsf{p}}
\newcommand{\msQ}{\mathsf{Q}}
\newcommand{\msR}{\mathsf{R}}
\newcommand{\msS}{\mathsf{S}}

\newcommand{\mst}{\mathsf{t}}

\newcommand{\msu}{\mathsf{u}}
\newcommand{\msV}{\mathsf{V}}

\newcommand{\End}{\mathsf{End}}

\newcommand{\Hom}{\mathsf{Hom}}
\newcommand{\id}{\mathsf{id}}

\newcommand{\wt}{\widetilde}
\newcommand{\lra}{\longrightarrow}

\newcommand{\C}{\mathbb{C}}

\newcommand{\Z}{\mathbb{Z}}

\newcommand{\ot}{\otimes}
\newcommand{\ol}{\overline}

\newcommand{\w}{\omega}
\newcommand{\BD}{\mathsf{BD}}
\newcommand{\BT}{\mathsf{BT}}

\newcommand{\BC}{\mathsf{BC}_{r,s}(q)}
\newcommand{\Iphi}{\mathit{\Phi}}
\newcommand{\Heckl}{\mathsf {HC}_\ell(q)}
\newcommand{\Heckr}{\mathsf {HC}_r(q)}
\newcommand{\Hecks}{ \mathsf {HC}_s(q)}
\newcommand{\BCloc}{ \mathsf {BC}_{r,s}(\mathcal R)}

\newcommand{\qn}{\mathfrak{q}(n)}
\newcommand{\V}{\mathbf{V}}
\newcommand{\er}{e_{r,r+1}}

\title[Quantum walled Brauer-Clifford superalgebras]{Quantum walled Brauer-Clifford superalgebras}

\author[Benkart et al.]{Georgia Benkart$^{1}$, Nicolas Guay$^{2}$, Ji Hye Jung$^{3}$, Seok-Jin Kang$^{4}$, and Stewart Wilcox$^{5}$}

\address{ Department of Mathematics \\
 University of Wisconsin-Madison \\
 Madison, WI 53706-1325, USA }
\email{benkart@math.wisc.edu}

\address{ Department of Mathematical and Statistical Sciences\\
\noindent University of Alberta\\
\noindent CAB 632 \\
\noindent Edmonton, AB  T6G 2G1 \\
\noindent Canada}

\email{nguay@ualberta.ca}

\address{ Department of Mathematical Sciences and Research Institute of Mathematics \\
\noindent Seoul National University \\
\noindent Seoul 151-747 \\
\noindent Republic of Korea}

\email{jung.ji.hye@hotmail.com, sjkang@snu.ac.kr}

\address{San Francisco, CA, USA}

\email{stewbasic@gmail.com}

\thanks{$^{1}$ G.B. acknowledges with gratitude the hospitality of the Banff International Research Station (BIRS),
where her collaboration on this project with N.G. began.}
\thanks{$^{2}$ This work was partly supported by a Discovery Grant of the Natural Sciences
and Engineering Research Council of Canada.}
\thanks{$^{3}$ This work was supported by NRF Grant \#2013-035155, NRF-2010-0019516 and NRF-2013R1A1A2063671.}
\thanks{$^{4}$ This work was partially supported by NRF Grant
\#2013-035155 and NRF Grant \#2013-055408.}
\thanks{$^{5}$ This work was partly supported by a Postdoctoral Fellowship
of the Pacific Institute for the Mathematical Sciences.}


\begin{document}

\setlength{\pdfpagewidth}{8.5in} \setlength{\pdfpageheight}{11in}

\begin{abstract}
We introduce a new family of superalgebras, the quantum walled
Brauer-Clifford superalgebras ${\mathsf B}{\mathsf C}_{r,s}(q)$. The
superalgebra ${\mathsf {BC}}_{r,s}(q)$ is a quantum deformation of
the walled Brauer-Clifford superalgebra ${\mathsf {BC}}_{r,s}$ and a
super version of the quantum  walled Brauer algebra. We prove that
${\mathsf {BC}}_{r,s}(q)$ is  the centralizer superalgebra of the
action of  ${\mathfrak U}_{q}({\mathfrak q}(n))$ on the mixed tensor
space $\V_{q}^{r,s}= \V_{q}^{\otimes r} \otimes (\V_q^*)^{\otimes
s}$ when $n \ge r+s$, where  ${\mathbf V}_{q}=\C(q)^{(n|n)}$ is the
natural representation of the quantum enveloping superalgebra
${\mathfrak U}_{q}({\mathfrak q}(n))$ and $\V_q^*$ is its dual
space. We also provide a diagrammatic realization of ${\mathsf
{BC}}_{r,s}(q)$ as the $(r,s)$-bead tangle algebra ${\mathsf
{BT}}_{r,s}(q)$. Finally, we define the notion of $q$-Schur
superalgebras of type $\mathsf{Q}$ and establish their basic
properties.
\end{abstract}

\maketitle

\section*{Introduction}

{\it Schur-Weyl duality} has been one of the most inspiring themes
in representation theory,  because it reveals many hidden
connections between the representation theories of seemingly
unrelated algebras. By the duality functor, one algebra appears as
the centralizer of the other acting on a common representation
space.  Many  interesting and important algebras have been
constructed as centralizer algebras in this way.

For example, the group algebra $\C \Sigma_k$ of the symmetric group
$\Sigma_k$ appears as the centralizer of the ${\mathfrak
{gl}}(n)$-action on $\msV^{\otimes k}$, where $\msV=\C^n$ is the
natural representation of the general linear Lie algebra ${\mathfrak
{gl}}(n)$. Similarly, Hecke algebras, Brauer algebras,
Birman-Murakami-Wenzl algebras and Hecke-Clifford superalgebras are
the centralizer algebras of the action of corresponding Lie
(super)algebras or quantum (super)algebras on the tensor powers of
their natural representations.

There are further generalizations of Schur-Weyl duality on mixed
tensor powers.    Let $\msV^{r,s}=\msV^{\otimes r} \otimes
(\msV^*)^{\otimes s}$ be the mixed tensor space of the natural
representation $\msV$ of ${\mathfrak {gl}}(n)$ and its dual space
$\msV^*$. The centralizer algebra of the ${\mathfrak
{gl}}(n)$-action on $\msV^{r,s}$ is the {\it walled Brauer algebra}
${\mathsf {B}}_{r,s}(n)$. The structure and properties of ${\mathsf
B}_{r,s}(n)$ were first investigated in \cite{BCHLLS, Ko, T}. By
replacing ${\mathfrak {gl}}(n)$ by the quantum enveloping algebra
${\mathfrak U}_{q}({\mathfrak {gl}}(n))$ and $\msV=\C^n$ by
$\msV_{q} = \C(q)^n$, we obtain as the centralizer algebra the {\it
quantum walled Brauer algebra}  studied in \cite{DDS1, DDS2, Ha, KM,
L}.     Super versions of the above constructions have been
investigated with the following substitutions:   Replace ${\mathfrak
{gl}}(n)$ by  ${\mathfrak {gl}}(m|n)$, $\C^n$ by $\C^{(m|n)}$;
${\mathfrak U}_{q}({\mathfrak {gl}}(n))$ by ${\mathfrak
U}_{q}({\mathfrak {gl}}(m|n))$;  and $\C(q)^n$ by $\C(q)^{(m|n)}$ as
in  \cite{LSM1, LSM2}.

The Lie superalgebra ${\mathfrak {q}}(n)$ is commonly referred to as
the \emph{queer Lie superalgebra} because of its unique properties
and the fact that it has no non-super counterpart.  Its natural
representation is the superspace ($\Z_2$-graded vector space)
$\V=\C^{(n|n)}$. The corresponding centralizer algebra
$\End_{\mfq(n)}(\V^{\ot r})$ was studied by Sergeev in \cite{Se},
and it is often referred to as the {\it Sergeev algebra}. Using a
modified version of a technique of Fadeev, Reshetikhin and Turaev,
Olshanski introduced the {\it quantum queer superalgebra}
$\mfU_q(\mfq(n))$ and established an analogue of Schur-Weyl duality.
That is,  he showed that there is a surjective algebra homomorphism
$\rho_{n,q}^{r}: {\mathsf {HC}}_{r} (q) \rightarrow
\End_{\mfq(n)}(\V^{\ot r})$, where ${\mathsf {HC}}_{r} (q)$ is the
{\it Hecke-Clifford superalgebra}, a quantum version of the Sergeev
algebra. Moreover, $\rho_{n,q}^{r}$ is an isomorphism when $n\ge r$.

On the other hand, in \cite{JK}  Jung and Kang considered a super
version of the walled Brauer algebra. For the mixed tensor space
$\V^{r,s} = \V^{\otimes r} \otimes (\V^*)^{\otimes s}$, one can ask,
What is the centralizer of the ${\mathfrak {q}}(n)$-action on
$\V^{r,s}$?  In order to answer this question, they introduced two
versions of  the {\it walled Brauer-Clifford superalgebra}, (which
is called the {\it walled Brauer superalgebra} in \cite{JK}). The
first  is constructed using $(r,s)$-superdiagrams, and the second is
defined by generators and relations. The main results of \cite{JK}
show that these two definitions are equivalent and that there is a
surjective algebra homomorphism $\rho_{n}^{r,s}: {\mathsf
{BC}}_{r,s} \rightarrow \End_{\mfq(n)}(\V^{r,s})$, which is an
isomorphism whenever $n \ge r+s$.

The purpose of this paper is to combine the constructions in
\cite{Ol} and  \cite{JK} to determine the centralizer algebra of the
${\mathfrak {U}}_{q}({\mathfrak {q}}(n))$-action on the mixed tensor
space $\V_{q}^{r,s}:= \V_{q}^{\otimes r} \otimes (\V_{q}^*)^{\otimes
s}$. We begin by introducing  the {\it quantum walled
Brauer-Clifford superalgebra} ${\mathsf {BC}}_{r,s}(q)$ via
generators and relations.  The superalgebra ${\mathsf
{BC}}_{r,s}(q)$ contains the quantum walled Brauer algebra  and the
Hecke-Clifford superalgebra ${\mathsf {HC}}_{r}(q)$ as subalgebras.

We then define an action of ${\mathsf {BC}}_{r,s}(q)$ on the mixed
tensor space $\V_{q}^{r,s}$ that supercommutes with the action of
${\mathfrak {U}}_{q}({\mathfrak{q}}(n))$. As a result,  there is a
superalgebra homomorphism $\rho_{n,q}^{r,s}:{\mathsf {BC}}_{r,s}(q)
\rightarrow \End_{{\mathfrak
{U}}_{q}({\mathfrak{q}}(n))}(\V_{q}^{r,s})$. Actually, defining such
an action is quite subtle and complicated, and we regard this as one
of our main results (Theorem \ref{BCqact}).  We use  the fact that
${\mathsf {BC}}_{r,s}$ is the classical limit of ${\mathsf
{BC}}_{r,s}(q)$ to show that the homomorphism $\rho_{n,q}^{r,s}$ is
surjective and that it is an isomorphism whenever $n \ge r+s$
(Theorem \ref{centraliser}).

We also give a diagrammatic realization of ${\mathsf {BC}}_{r,s}(q)$
as the {\it $(r,s)$-bead tangle algebra ${\mathsf{BT}}_{r,s}(q)$}.
An $(r,s)$-bead tangle is a portion of a planar knot diagram
satisfying the conditions in Definition \ref{def:tangle}. The
algebra ${\mathsf{BT}}_{r,s}(q)$ is a quantum deformation of the
{\it $(r,s)$-bead diagram algebra $\BD_{r,s}$},  which is isomorphic
to the walled Brauer-Clifford superalgebra ${\mathsf{BC}}_{r,s}$
(Theorem \ref{th:BC and BD}). Modifying the arguments in \cite{K},
we prove that the algebra $\BT_{r,s}(q)$ is isomorphic to
${\mathsf{BC}}_{r,s}(q)$ (Theorem \ref{th:BC and BT}), so that
${\mathsf{BC}}_{r,s}(q)$ can be regarded as a diagram algebra.

In the final section, we  introduce  {\it $q$-Schur superalgebras of
type ${\mathsf{Q}}$} and prove that the classical results for
$q$-Schur algebras can be extended to this setting.

\vskip3mm
\section{The walled Brauer-Clifford superalgebras}

To begin, we recall the definition of  the Lie superalgebra
$\mathfrak{q}(n)$
 and its basic properties.
Let ${\tt I} = \{\pm i  \mid  i=1,\dots, n\}$, and set $\Z_2 =
\Z/2\Z$. The superspace  $\mathbf{V} = \C(n|n) = \C^n \oplus \C^n$
has a standard basis $\{v_i  \mid i \in {\tt I}\}$. We say that the
{\it parity} of $v_{i}$ equals $|i|:=|v_i|$, where $|v_i|=1$ if $i <
0$ and $|v_i|=0$ if $i> 0$.

The endomorphism algebra is $\Z_2$-graded,  $\End_{\C}(\V) =
\End_{\C}(\V)_0 \oplus \End_{\C}(\V)_1$,  and has a basis of matrix
units $E_{ij}$ with $-n \le i,j \le n$, $ij \neq 0$, where the
parity of $E_ {ij}$ is $|E_{ij}| = |i| + |j|$ ($\mathsf{mod}\, 2$).
The general linear Lie superalgebra  $\mfgl({n|n})$ is obtained from
$\End_{\C}(\V)$ by using the {\it supercommutator} \
 $$[X,Y] = XY - (-1)^{|X||Y|} YX$$ for homogeneous elements $X$,$Y$.
The map $\iota:  \mfgl({n|n})  \lra  \mfgl(n|n)$ given by $E_{ij}
\mapsto E_{-i,-j}$ is an involutive automorphism of $\mfgl(n|n)$.
Let $J = \sum_{a=1}^n ( E_{a,-a} - E_{-a,a}) \in \mfgl({n|n}) $.

\begin{definition}\label{qn}
The  \emph{queer Lie superalgebra $\mfq(n)$} can be defined
equivalently as either the centralizer of $J$ in $ \mfgl({n|n}) $
(under the supercommutator product)  or the fixed-point subalgebra
of $\mfgl({n|n})$ with respect to the involution $\iota$.
\end{definition}

Identifying $\V$ with the space of $(n|n)$ column vectors and $\{v_i
\mid i \in \tt{I}\}$  with the standard basis for the column
vectors, we have
 $J= \left(
\begin{matrix} 0 & I \\ -I & 0\end{matrix} \right)$,  and $\qn$ can be expressed in the matrix form as
\begin{align*}
\left\{ \left(
                          \begin{array}{cc}
                            A & B \\
                            B & A \\
                          \end{array}
                        \right)  \Big{|} \ A,B \text{ are arbitrary $n \times n$ complex matrices} \right\}.
\end{align*}
Then $\mathfrak{q}(n)$ inherits a $\Z_2$-grading from $\mfgl(n|n)$,
and a basis for $\mfq(n)_0$ is given by $\msE_{ab}^0 = E_{ab} +
E_{-a,-b}$ and  for $\mfq(n)_1$ by $\msE_{ab}^1 =
E_{a,-b}+E_{-a,b}$,  where  $1 \le a,b \le n$.

The superalgebra  $\qn$ acts naturally on $\V$ by matrix
multiplication on the column vectors, and  $\V $ is an irreducible
representation of $\qn$.    The action on $\V$ extends to one on
$\V^{\otimes{k}}$ by
\begin{align} \label{eq:action of qn on Vk}
   g  (v_{i_1} \otimes \cdots \otimes v_{i_k})
   =\sum_{j=1}^{k} (-1)^{(|v_{i_1}|+\cdots+|v_{i_{j-1}}|)|g|}
  v_{i_1}\otimes \cdots \otimes v_{i_{j-1}}  \otimes gv_{i_j}  \otimes v_{i_{j+1}} \otimes \cdots\otimes v_{i_k},
 \end{align}
where $g$ is homogeneous.  It also extends in a similar fashion to
the mixed tensor space $\V^{r,s} :=\V^{\otimes r} \otimes
(\V^*)^{\otimes s}$, where $\V^*$ is the dual representation of
$\V$, and  the action on $\V^*$ is given by
$$(g \w)(v):=-(-1)^{|g||\w|}\w(g v)$$  for homogeneous elements  $g \in \qn,  \w \in \V^*$,  and $v\in \V$.
We assume  $\{\w_i \mid i \in {\tt I} \}$ is the basis of $\V^*$
dual to the standard basis  $\{v_i \mid i \in {\tt I} \}$ of $\V$.

In an effort to construct the  centralizer superalgebra
$\End_{\qn}(\V^{r,s})$,  Jung and Kang \cite{JK}  introduced the
notion of the \emph{walled Brauer-Clifford superalgebra}
$\mathsf{BC}_{r,s}$. The superalgebra $\mathsf{BC}_{r,s}$ is
generated by  even generators
  $s_1, \ldots, s_{r-1},$ $s_{r+1}, \ldots,$ $s_{r+s-1}$,
  $e_{r,r+1}$ and odd generators $c_1, \ldots, c_{r+s}$,
 which satisfy the following defining relations (for $i,j$ in the allowable range):
\begin{equation}
 \begin{aligned}
    &s_i^2=1, \quad  s_is_{i+1}s_i=s_{i+1}s_is_{i+1}, \quad  s_is_j=s_js_i \ \  (|i-j|>1), \\
 & \er^2=0, \quad  \er s_j =s_j \er \ \ (j \neq r-1,r+1),\\
 & \er=\er s_{r-1} \er =\er s_{r+1} \er, \\
& s_{r-1} s_{r+1} \er s_{r+1} s_{r-1} \er  = \er s_{r-1} s_{r+1} \er s_{r+1} s_{r-1},  \\
  &c_i^2=-1 \ \ (1 \leq i \leq r), \quad c_i^2=1 \ \ (r+1 \leq i \leq r+s), \quad  c_i c_j=-c_j c_i \ \ (i \neq j),  \\
  & s_i c_i s_i=c_{i+1}, \quad  s_i c_j=c_j s_i \ \ (j \neq i, i+1),\\
   & c_r \er = c_{r+1} \er, \quad \er c_r =\er c_{r+1},  \\
   &\er c_r \er=0, \quad  \er c_j=c_j \er \ \ (j \neq r, r+1).
  \end{aligned}
\end{equation}

For $1 \le j \le r-1, \ r+1 \le k \le r+s-1,\  1 \le l  \le r$,  and
$r+1 \le m \le r+s$, the action of the generators of
$\mathsf{BC}_{r,s}$ on $\V^{r,s}$ (which is on the right)  is
defined as follows:
$$(v_{i_1} \otimes \cdots \otimes v_{i_r} \otimes \w_{i_{r+1}} \otimes \cdots \otimes \w_{i_{r+s}}) \ s_j
=(-1)^{|i_j||i_{j+1}|} v_{i_1} \otimes \cdots \otimes v_{i_{j-1}}
\otimes v_{i_{j+1}} \otimes v_{i_j} \otimes v_{i_{j+2}} \cdots
\otimes \w_{i_{r+s}},$$
$$(v_{i_1} \otimes \cdots \otimes v_{i_r} \otimes \w_{i_{r+1}} \otimes \cdots \otimes \w_{i_{r+s}}) \ s_k
=(-1)^{|i_k||i_{k+1}|} v_{i_1} \otimes \cdots \otimes \w_{i_{k-1}}
\otimes \w_{i_{k+1}} \otimes \w_{i_k} \otimes w_{i_{k+2}} \cdots
\otimes \w_{i_{r+s}},$$
$$(v_{i_1} \otimes \cdots \otimes v_{i_r} \otimes \w_{i_{r+1}} \otimes \cdots \otimes \w_{i_{r+s}}) \ e_{r,r+1}
=\delta_{i_r,i_{r+1}} (-1)^{|i_r|} \hspace{-0.5mm}\sum_{i=-n}^n
v_{i_1} \otimes \cdots \otimes v_{i_{r-1}} \otimes v_{i} \otimes
\w_{i} \otimes \w_{i_{r+2}} \cdots \otimes \w_{i_{r+s}},$$
$$(v_{i_1} \otimes \cdots \otimes v_{i_r} \otimes \w_{i_{r+1}} \otimes \cdots \otimes \w_{i_{r+s}})  \ c_l
= (-1)^{|i_1| + \cdots + |i_{l-1}|}
 v_{i_1} \otimes \cdots \otimes v_{i_{l-1}} \otimes Jv_{i_l} \otimes v_{i_{l+1}} \otimes  \cdots \otimes \w_{i_{r+s}},$$
$$(v_{i_1} \otimes \cdots \otimes v_{i_r} \otimes w_{i_{r+1}} \otimes \cdots \otimes w_{i_{r+s}}) \ c_m
= (-1)^{|i_1| + \cdots + |i_{m-1}|}
 v_{i_1} \otimes \cdots \otimes \w_{i_{m-1}} \otimes  J^{\tt T}\w_{i_m} \otimes \w_{i_{m+1}} \otimes  \cdots \otimes \w_{i_{r+s}},$$
where $J^{\tt T}$ is the supertranspose of $J$,  and the
\emph{supertranspose} is given by $E_{ij}^{\tt
T}:=(-1)^{(|i|+|j|)|i|}E_{ji}$. By direct calculation, it can be
verified that this action of the generators gives rise to an action
of $\mathsf{BC}_{r,s}$ on $\V^{r,s}$. Thus,  there is an
homomorphism of superalgebras,
\begin{equation*}
\rho_n^{r,s}: \mathsf{BC}_{r,s} \longrightarrow
\End_{\C}(\V^{r,s})^{\rm{op}}.
\end{equation*}

The next theorem, due to Jung and Kang \cite{JK},  identifies a basis for $\mathsf{BC}_{r,s}$. In stating it
and some subsequent results, we use the following notation:
For a nonempty subset $A = \{a_1 <  \cdots  < a_m\}$ of $\{1,\dots, r+s\}$, let  $c_A = c_{a_1}  \cdots  c_{a_m}$, and set $c_{\emptyset} = 1$.    Let  $e_{p,q}  = \vphi e_{r,r+1} \vphi^{-1}$,
where $\vphi=  s_{q-1} \cdots  s_{r+1}  s_p \cdots  s_{r-1}$
for $1 \le p \le r$ and $r+1 \le q \le r+s$.

\begin{theorem}  \hfill \label{repmixtensor}
\begin{itemize}
\item[\rm{(i)}] The actions of $\mathsf{BC}_{r,s}$ and $\qn$ on $\V^{r,s}$ supercommute with each other.
  Thus, there is an algebra homomorphism $\rho_n^{r,s}: \mathsf{BC}_{r,s} \rightarrow \End_{\qn}(\V^{r,s})^{\rm{op}}.$

\item[\rm{(ii)}]  The map $\rho_n^{r,s}$ is surjective.
Moreover, it is an isomorphism if $n \ge r+s$.

\item[\rm{(iii)}] The elements
$$ c_P  e_{p_1,q_1} \cdots  e_{p_a,q_a} \sigma  c_Q$$  such that
\begin{enumerate}\item  $1 \leq p_1 <\cdots <p_a \leq r$;
\item  $ r+1 \leq q_i \leq r+s$,  $i=1,\ldots, a$, are all distinct;
\item $\sigma \in \sum_r \times \sum_s,\
      \sigma^{-1}(p_1) < \cdots < \sigma^{-1}(p_a)$;  and
\item  $P \subseteq \{p_1, \ldots, p_a \}$,  $Q \subseteq \{ 1 ,\ldots, r \} \cup \{r+1, \ldots , r+s \} \backslash \{ \sigma^{-1}(q_1), \ldots, \sigma^{-1}(q_a)\}$
\end{enumerate}
\noindent comprise a basis of $\mathsf{BC}_{r,s}$.
 \item[\rm{(iv)}] $\dim_{\C} \mathsf{BC}_{r,s} =2^{r+s}(r+s)!$.
\end{itemize}
\end{theorem}

\vskip3mm
\section{The $(r,s)$-bead diagram algebra $\BD_{r,s}$}

In this section, we give a new diagrammatic realization of the
walled Brauer-Clifford superalgebra $\mathsf{BC}_{r,s}$.

\begin{definition}
  An \emph{$(r,s)$-bead diagram}, or simply a \emph{bead diagram}, is a graph consisting of two rows with $r+s$ vertices in each row such that
  the following conditions hold:
  \begin{enumerate}
\item[{(1)}]  Each vertex is connected by a strand to exactly one other vertex.
\item[{(2)}]  Each strand may (or may not)
have finitely many numbered beads. The bead numbers on a diagram
start with $1$ and are distinct consecutive positive integers.
\item[{(3)}] There is a vertical wall separating the first $r$ vertices from the last $s$ vertices in each row.
\item[{\rm (4)}]  A \emph{vertical strand} connects a vertex on the top row with one on the bottom row, and it
cannot cross the wall.  A  \emph{horizontal strand} connects
vertices in the same row, and  it must cross the wall.
\item[{\rm (5)}]  No loops are permitted  in an $(r,s)$-bead diagram.
\end{enumerate}
\end{definition}

The following diagram is an example of a $(3,2)$-bead diagram.

$$
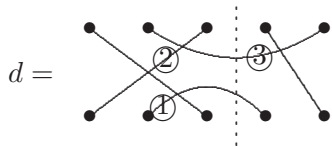
{\beginpicture
\setcoordinatesystem units <0.78cm,0.39cm> \setplotarea x from 0 to
5.5, y from -1.5 to 4 \put{$d = $} at 0 1.5 \put{$\bullet$} at  1 3
\put{$\bullet$} at  1 0 \put{$\bullet$} at  2 3  \put{$\bullet$} at
2 0 \put{$\bullet$} at  3 3  \put{$\bullet$} at  3 0 \put{$\bullet$}
at  4 3  \put{$\bullet$} at  4 0 \put{$\bullet$} at  5 3
\put{$\bullet$} at  5 0 \plot 1 3  3 0 / \plot 3 3  1 0 / \plot 4 3
5 0 /

\put{\textcircled{1}} at 2.25 0.3 \put{\textcircled{2}} at 2.3 1.9
\put{\textcircled{3}} at 3.9 2


\setdashes  <.4mm,1mm> \plot 3.5 -1   3.5 4 / \setsolid
\setquadratic \plot 2 3  3.5 2 5 3 / \plot 2 0 3 1  4 0 /
\endpicture}$$

Beads can slide along a given strand, but they cannot jump to
another strand nor can they interchange positions on a given strand.
For example, the following diagrams are the same as  $(2,1)$-bead
diagrams.
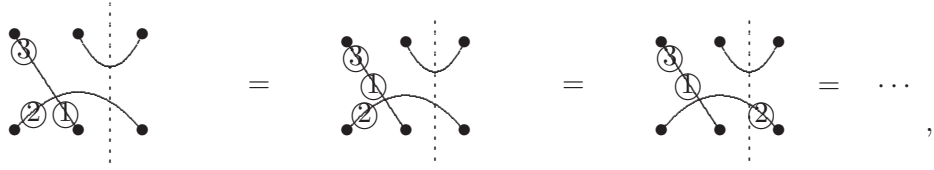
\begin{center}
${\beginpicture \setcoordinatesystem units <0.85cm,0.425cm>
\setplotarea x from 0 to 4.5, y from -1.5 to 4

\put{$\bullet$} at  1 3  \put{$\bullet$} at  1 0 \put{$\bullet$} at
2 3  \put{$\bullet$} at  2 0 \put{$\bullet$} at  3 3
\put{$\bullet$} at  3 0

\plot 1 3  2 0 /

\put{\textcircled{1}} at 1.8 0.5 \put{\textcircled{2}} at 1.3 0.5
\put{\textcircled{3}} at 1.15 2.45


\setdashes  <.4mm,1mm> \plot 2.5 -1   2.5 4 / \setsolid

\setquadratic \plot 2 3  2.5 2 3 3 / \plot 1 0 2 1.2  3 0 /
\endpicture}$
${\beginpicture \setcoordinatesystem units <0.78cm,0.39cm>
\setplotarea x from 0 to 4.5, y from -1.5 to 4 \put{$=$} at -0.5 1.5
\put{$\bullet$} at  1 3  \put{$\bullet$} at  1 0 \put{$\bullet$} at
2 3  \put{$\bullet$} at  2 0 \put{$\bullet$} at  3 3
\put{$\bullet$} at  3 0

\plot 1 3  2 0 /

\put{\textcircled{1}} at 1.45 1.5 \put{\textcircled{2}} at 1.3 0.5
\put{\textcircled{3}} at 1.15 2.45


\setdashes  <.4mm,1mm> \plot 2.5 -1   2.5 4 / \setsolid

\setquadratic \plot 2 3  2.5 2 3 3 / \plot 1 0 2 1.2  3 0 /
\endpicture}$
${\beginpicture \setcoordinatesystem units <0.78cm,0.39cm>
\setplotarea x from 0 to 5.5, y from -1.5 to 4 \put{$=$} at -0.5 1.5
\put{$= \quad \cdots$} at  4.5 1.5 \put{$\bullet$} at  1 3  \put{$\bullet$}
at  1 0 \put{$\bullet$} at  2 3  \put{$\bullet$} at  2 0
\put{$\bullet$} at  3 3  \put{$\bullet$} at  3 0

\plot 1 3  2 0 /

\put{\textcircled{1}} at 1.45 1.5 \put{\textcircled{2}} at 2.7 0.5
\put{\textcircled{3}} at 1.15 2.45

\setdashes  <.4mm,1mm> \plot 2.5 -1   2.5 4 / \setsolid

\setquadratic \plot 2 3  2.5 2 3 3 / \plot 1 0 2 1.2  3 0 /
\endpicture}$,
\end{center}
while the following diagrams are considered to be different

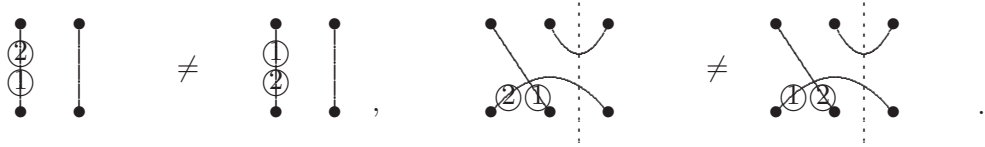
\begin{center}
${\beginpicture \setcoordinatesystem units <0.78cm,0.39cm>
\setplotarea x from 0 to 3.5, y from -1.5 to 4 \put{$\bullet$} at  1
3  \put{$\bullet$} at  1 0
\put{$\bullet$} at  2 3  \put{$\bullet$} at  2 0

\plot 1 3  1 0 / \plot 2 3 2 0 /

\put{\textcircled{1}} at 1 1 \put{\textcircled{2}} at 1 2


\endpicture}$
${\beginpicture \setcoordinatesystem units <0.78cm,0.39cm>
\setplotarea x from 0 to 3.5, y from -1.5 to 4 \put{$\neq$} at -0.5
1.5 \put{$\bullet$} at  1 3  \put{$\bullet$} at  1 0
\put{$\bullet$} at  2 3  \put{$\bullet$} at  2 0

\plot 1 3  1 0 / \plot 2 3 2 0 /

\put{\textcircled{2}} at 1 1 \put{\textcircled{1}} at 1 2

  \put{,} at  2.7 0
\endpicture}$
${\beginpicture \setcoordinatesystem units <0.78cm,0.39cm>
\setplotarea x from 0 to 4.5, y from -1.5 to 4 \put{$\bullet$} at  1
3  \put{$\bullet$} at  1 0 \put{$\bullet$} at  2 3  \put{$\bullet$}
at  2 0 \put{$\bullet$} at  3 3  \put{$\bullet$} at  3 0

\plot 1 3  2 0 /

\put{\textcircled{1}} at 1.8 0.5 \put{\textcircled{2}} at 1.3 0.5


\setdashes  <.4mm,1mm> \plot 2.5 -1   2.5 4 / \setsolid

\setquadratic \plot 2 3  2.5 2 3 3 / \plot 1 0 2 1.2  3 0 /
\endpicture}$
${\beginpicture \setcoordinatesystem units <0.78cm,0.39cm>
\setplotarea x from 0 to 5.5, y from -1.5 to 4 \put{$\neq$} at 0 1.5
\put{$\bullet$} at  1 3  \put{$\bullet$} at  1 0 \put{$\bullet$} at
2 3  \put{$\bullet$} at  2 0 \put{$\bullet$} at  3 3
\put{$\bullet$} at  3 0

\plot 1 3  2 0 /

\put {.}  at 4.5  0

\put{\textcircled{2}} at 1.8 0.5 \put{\textcircled{1}} at 1.3 0.5

\setdashes  <.4mm,1mm> \plot 2.5 -1   2.5 4 / \setsolid

\setquadratic \plot 2 3  2.5 2 3 3 / \plot 1 0 2 1.2  3 0 /

\endpicture}$
\end{center}

An  $(r,s)$-bead diagram having an even number of beads is regarded
as \emph{even} (resp. \emph{odd}).
Let $\widetilde{\BD}_{r,s}$ be
the superspace with basis consisting of the $(r,s)$-bead diagrams.
We define a multiplication on $\widetilde{\BD}_{r,s}$. For
$(r,s)$-bead diagrams $d_1,d_2$, we put $d_1$ under $d_2$ and
identify the top vertices of $d_1$ with the bottom vertices of
$d_2$. If there is a loop in the middle row, we say $d_1d_2=0$. If
there is no loop in the middle row,  we add the largest bead number
in $d_1$ to each bead number in $d_2$, so that
 if $m$ is the largest bead number in $d_1$,
then a bead numbered $i$ in $d_2$ is now numbered $m+i$  in
$d_1d_2$. Then we concatenate the diagrams.    For example, if
  $$
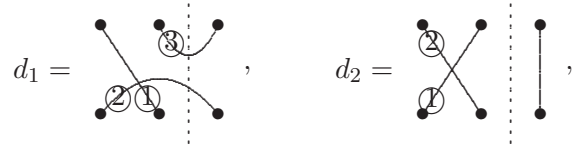
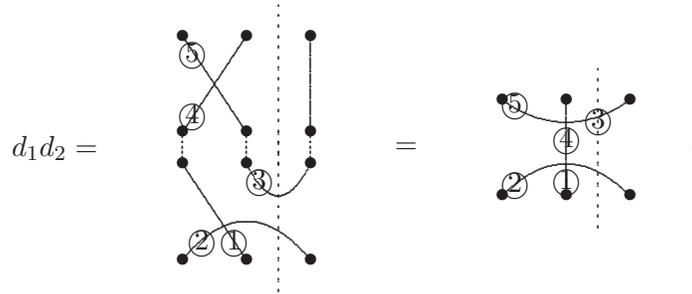
{\beginpicture
\setcoordinatesystem units <0.78cm,0.39cm> \setplotarea x from 0 to
4.5, y from -1.5 to 4 \put{$d_1=$} at 0 1.5 \put{$\bullet$} at  1 3
\put{$\bullet$} at  1 0 \put{$\bullet$} at  2 3  \put{$\bullet$} at
2 0 \put{$\bullet$} at  3 3  \put{$\bullet$} at  3 0 \put{,} at 3.5
1.5

\plot 1 3  2 0 /

\put{\textcircled{1}} at 1.8 0.5 \put{\textcircled{2}} at 1.3 0.5
\put{\textcircled{3}} at 2.2 2.4


\setdashes  <.4mm,1mm> \plot 2.5 -1   2.5 4 / \setsolid

\setquadratic \plot 2 3  2.5 2 3 3 / \plot 1 0 2 1.2  3 0 /
\endpicture} \hskip1em
{\beginpicture \setcoordinatesystem units <0.78cm,0.39cm>
\setplotarea x from 0 to 5.5, y from -1.5 to 4 \put{$d_2=$} at 0 1.5
\put{$\bullet$} at  1 3  \put{$\bullet$} at  1 0 \put{$\bullet$} at
2 3  \put{$\bullet$} at  2 0 \put{$\bullet$} at  3 3
\put{$\bullet$} at  3 0 \put{,} at 3.5 1.5

\plot 1 3  2 0 / \plot 2 3 1 0 / \plot 3 3 3 0 /

\put{\textcircled{1}} at 1.15 0.45 \put{\textcircled{2}} at 1.15 2.4

\setdashes  <.4mm,1mm> \plot 2.5 -1   2.5 4 / \setsolid
\endpicture}$$

then
\begin{center}
${\beginpicture \setcoordinatesystem units <0.85cm,0.425cm>
\setplotarea x from 0 to 4.5, y from -1.5 to 4 \put{$d_1d_2=$} at -1
3.5 \put{$\bullet$} at  1 3  \put{$\bullet$} at  1 0 \put{$\bullet$}
at  2 3  \put{$\bullet$} at  2 0 \put{$\bullet$} at  3 3
\put{$\bullet$} at  3 0

\plot 1 3  2 0 /

\put{\textcircled{1}} at 1.8 0.5 \put{\textcircled{2}} at 1.3 0.5
\put{\textcircled{3}} at 2.2 2.4

\setdashes  <.4mm,1mm> \plot 2.5 -1   2.5 4 / \setsolid

\put{$\bullet$} at  1 7  \put{$\bullet$} at  1 4 \put{$\bullet$} at
2 7  \put{$\bullet$} at  2 4 \put{$\bullet$} at  3 7
\put{$\bullet$} at  3 4

\plot 1 7  2 4 / \plot 2 7 1 4 / \plot 3 7 3 4 /

\put{\textcircled{4}} at 1.15 4.45 \put{\textcircled{5}} at 1.15 6.4

\setdashes  <.4mm,1mm> \plot 2.5 3   2.5 8 / \setsolid

\setdashes  <.4mm,.3mm> \plot 1 3 1 4 / \plot 2 3 2 4 / \plot 3 3 3
4 / \setsolid

\put{$=$} at 4.5 3.5 \put{$\bullet$} at  6 5  \put{$\bullet$} at  6
2 \put{$\bullet$} at  7 5  \put{$\bullet$} at  7 2 \put{$\bullet$}
at  8 5  \put{$\bullet$} at  8 2 \put{.} at 9 3.5

\plot 7 2 7 5 /

\put{\textcircled{1}} at 7 2.4 \put{\textcircled{2}} at 6.2 2.3
\put{\textcircled{3}} at 7.5 4.3 \put{\textcircled{4}} at 7 3.7
\put{\textcircled{5}} at 6.2 4.8

\setdashes  <.4mm,1mm> \plot 7.5 1 7.5 6 / \setsolid

\setquadratic \plot 6 2 7 3 8 2 / \plot 6 5  7 4.3 8 5 /

\setquadratic \plot 2 3  2.5 2 3 3 / \plot 1 0 2 1.2  3 0 /
\endpicture}$
\end{center}

Observe that $\widetilde{\BD}_{r,s}$ is closed under this product.
If the number of beads in $d_1$ (resp. $d_2$) is $m_1$ (resp. $m_2$)
and $d_1d_2 \neq 0$, then the number of beads in $d_1d_2$ is
$m_1+m_2$. Hence, the multiplication respects the $\Z_2$-grading.
Let $d_1,d_2,d_3 \in \widetilde{\BD}_{r,s}$.  Note that the
connections in $(d_1d_2)d_3$ and $d_1(d_2d_3)$ are the same.  The
strands where the beads are placed and the bead numbers are also the
same in $(d_1d_2)d_3$ and $d_1(d_2d_3)$. Therefore,
$(d_1d_2)d_3=d_1(d_2d_3)$. The identity element of
$\widetilde{\BD}_{r,s}$ is the diagram with no beads such that each
top vertex is connected to the corresponding bottom vertex.

For $1\leq i \leq r-1$, \ $r+1 \leq j \leq r+s-1$, \ $1 \leq k \leq
r$, \ $r+1 \leq l \leq r+s$, let $\BD'_{r,s}$ be the subalgebra of
$\widetilde{\BD}_{r,s}$ generated by the following diagrams:

${\beginpicture \setcoordinatesystem units <0.78cm,0.39cm>
\setplotarea x from 0 to 9, y from -1 to 4 \put{$\mathbf{s}_i:= $}
at 0 1.5 \put{$\bullet$} at  1 0  \put{$\bullet$} at  1 3
\put{$\bullet$} at  2 0  \put{$\bullet$} at  2 3 \put{$\bullet$} at
3 0  \put{$\bullet$} at  3 3 \put{$\bullet$} at  4 0
\put{$\bullet$} at  4 3 \put{$\bullet$} at  5 0  \put{$\bullet$} at
5 3 \put{$\bullet$} at  6 0  \put{$\bullet$} at  6 3 \put{$\bullet$}
at  7 0  \put{$\bullet$} at  7 3 \put{$\bullet$} at  8 0
\put{$\bullet$} at  8 3

\put{$\cdots$} at 1.5 1.5 \put{$\cdots$} at 5.5 1.5 \put{$\cdots$}
at 7.5 1.5 \put{{\scriptsize$i$}} at 3 4 \put{{\scriptsize$i+1$}} at
4 4 \put{,} at 8.5 1 \plot 1 3 1 0 / \plot 2 3 2 0 / \plot 3 3 4 0 /
\plot 4 3 3 0 / \plot 5 3 5 0 / \plot 6 3 6 0 / \plot 7 3 7 0 /
\plot 8 3 8 0 / \setdashes  <.4mm,1mm> \plot 6.5 -1   6.5 4 /
\setsolid
\endpicture}$
${\beginpicture \setcoordinatesystem units <0.78cm,0.39cm>
\setplotarea x from 0 to 9, y from -1 to 4 \put{$\mathbf{s}_j:= $}
at 0 1.5 \put{$\bullet$} at  1 0 \put{$\bullet$} at  1 3
\put{$\bullet$} at 2
 0  \put{$\bullet$} at  2 3 \put{$\bullet$} at  3 0 \put{$\bullet$}
at 3 3 \put{$\bullet$} at  4 0  \put{$\bullet$} at 4 3
\put{$\bullet$} at  5 0  \put{$\bullet$} at  5 3 \put{$\bullet$} at
6 0 \put{$\bullet$} at  6 3 \put{$\bullet$} at  7 0 \put{$\bullet$}
at 7 3 \put{$\bullet$} at  8 0  \put{$\bullet$} at 8 3

\put{$\cdots$} at 1.5 1.5 \put{$\cdots$} at 3.5 1.5 \put{$\cdots$}
at 7.5 1.5 \put{,} at 8.5 1 \put{{\scriptsize$j$}} at 5 4
\put{{\scriptsize$j+1$}} at 6 4

\plot 1 3 1 0 / \plot 2 3 2 0 / \plot 3 3 3 0 / \plot 4 3 4 0 /
\plot 5 3 6 0 / \plot 6 3 5 0 / \plot 7 3 7 0 / \plot 8 3 8 0 /
\setdashes  <.4mm,1mm> \plot 2.5 -1   2.5 4 / \setsolid
\endpicture}$\\

${\beginpicture \setcoordinatesystem units <0.78cm,0.39cm>
\put{$\mathbf{e}_{r,r+1} := $} at -0.5 1.5 \put{$\bullet$} at  1 0
\put{$\bullet$} at  1 3 \put{$\bullet$} at  3 0  \put{$\bullet$} at
3 3 \put{$\bullet$} at  4 0  \put{$\bullet$} at  4 3 \put{$\bullet$}
at  5 0  \put{$\bullet$} at  5 3 \put{$\bullet$} at  6 0
\put{$\bullet$} at  6 3 \put{$\bullet$} at  8 0  \put{$\bullet$} at
8 3

\put{$\cdots$} at 2 1.5 \put{$\cdots$} at 7 1.5 \put{,} at 8.5 1
\put{{\scriptsize $1$}} at 1 4 \put{{\scriptsize $r$}} at 4 4
\put{{\scriptsize $r+1$}} at 5 4 \put{{\scriptsize $r+s$}} at 8 4
\plot 1 3 1 0 / \plot 3 3 3 0 / \plot 8 3 8 0 / \plot 6 3 6 0 /
\setdashes  <.4mm,1mm> \plot 4.5 -1   4.5 4 / \setsolid
\setquadratic \plot 4 3 4.5 2 5 3 / \plot 4 0 4.5 1 5 0  /
\endpicture}$

\vskip3mm $ {\beginpicture \setcoordinatesystem units
<0.78cm,0.39cm> \setplotarea x from 0 to 8, y from -1 to 4
\put{$\mathbf{c}_k: = $} at 0 1.5
\put{$\bullet$} at  1 0    \put{$\bullet$} at  1 3
\put{$\bullet$} at  2 0  \put{$\bullet$} at  2 3
\put{$\bullet$} at  3 0  \put{$\bullet$} at  3 3
\put{$\bullet$} at  4 0  \put{$\bullet$} at  4 3
\put{$\bullet$} at  5 0 \put{$\bullet$} at  5 3
\put{$\bullet$} at  6 0  \put{$\bullet$} at 6 3
\put{$\bullet$} at  7 0  \put{$\bullet$} at  7 3

\put{$\cdots$} at 1.5 1.5
\put{$\cdots$} at 4.5 1.5
\put{$\cdots$} at 6.5 1.5
\put{{\scriptsize $k$}} at 3 4 \put{,} at 7.5 1
\plot 1 3 1 0 / \plot 2 3 2 0  / \plot 3 3 3 0 / \plot 4 3 4 0 / \plot 5 3 5 0
/ \plot 6 3 6 0 / \plot 7 3 7 0 / \put{\textcircled{1}} at 2.95 1.5
\setdashes  <.4mm,1mm> \plot 5.5 -1   5.5 4 /
\endpicture}$ \hskip2em
${\beginpicture \setcoordinatesystem units <0.78cm,0.39cm>
\setplotarea x from 0 to 7.5, y from -1 to 4
\put{$\mathbf{c}_l: =$} at 0 1.5

\put{$\bullet$} at  1 0 \put{$\bullet$} at  1 3
\put{$\bullet$} at  2 0 \put{$\bullet$} at  2 3
\put{$\bullet$} at 3 0 \put{$\bullet$} at  3 3
\put{$\bullet$} at  4 0 \put{$\bullet$} at 4 3
\put{$\bullet$} at  5 0 \put{$\bullet$} at  5 3
\put{$\bullet$} at  6 0 \put{$\bullet$} at  6 3
\put{$\bullet$} at 7 0 \put{$\bullet$} at  7 3

\put{$\cdots$} at 1.5 1.5 \put{$\cdots$} at 3.5 1.5
\put{$\cdots$} at 6.5 1.5 \put{.} at 7.5 1 \put{{\scriptsize $l$}} at 5 4

\put{\textcircled{1}} at 4.95 1.5

\plot 1 3 1 0 / \plot 2 3 2 0  / \plot 3 3 3 0 / \plot 4 3 4 0 /
\plot 5 3 5 0 / \plot 6 3 6 0 / \plot 7 3 7 0 /
 \setdashes  <.4mm,1mm> \plot 2.5 -1   2.5 4 /
\endpicture}$

 Assume  $\mathsf{L}$ is  the (two-sided)
ideal of $\BD'_{r,s}$ generated by the following homogeneous
elements,
\begin{align} \label{def:I_in_BD'}
\mathbf{c}_k^2+1 \ \ (1\leq k \leq r), \ \ \, \mathbf{c}_l^2 -1 \
\  (r+1 \leq l \leq r+s),  \  \text{and  } \
\mathbf{c}_i\mathbf{c}_j  + \mathbf{c}_j\mathbf{c}_i  \ \ \, (1 \le
i \neq j \le r+s),
\end{align}
and let $\BD_{r,s}$ be the quotient  superalgebra
$\BD'_{r,s}/{\mathsf{L}}$.     We say that  $\BD_{r,s}$ is the
\emph{$(r,s)$-bead diagram algebra}, or simply the \emph{bead
diagram algebra}. For simplicity, we identify cosets  in $\BD_{r,s}$
with their diagram representatives and  make the following identifications
in $\BD_{r,s}$:

${\beginpicture \setcoordinatesystem units <0.78cm,0.39cm>
\setplotarea x from 0 to 7.5, y from -1 to 4
\put{$\mathbf{c}_k^2 =$} at 0 1.5
\put{$\bullet$} at  1 0  \put{$\bullet$} at  1 3
\put{$\bullet$} at  2 0  \put{$\bullet$} at  2 3
\put{$\bullet$} at 3 0  \put{$\bullet$} at  3 3
\put{$\bullet$} at  4 0  \put{$\bullet$} at  4 3
\put{$\bullet$} at  5 0  \put{$\bullet$} at 5 3
\put{$\bullet$} at  6 0  \put{$\bullet$} at  6 3
\put{$\bullet$} at  7 0  \put{$\bullet$} at  7 3

\put{$\cdots$} at 1.5 1.5 \put{$\cdots$} at 4.5 1.5 \put{$\cdots$}
at 6.5 1.5 \put{{\scriptsize $k$}} at 3 4
\plot 1 3 1 0 / \plot 2 3 2 0  / \plot 3 3 3 0 / \plot 4 3 4 0 /
\plot 5 3 5 0 / \plot 6 3 6 0 / \plot 7 3 7 0 /

\put{\textcircled{1}} at 2.95 1 \put{\textcircled{2}} at 2.95 2

\setdashes  <.4mm,1mm> \plot 5.5 -1   5.5 4 /
\endpicture}$ ${\beginpicture
\setcoordinatesystem units <0.78cm,0.39cm> \setplotarea x from 0 to
8, y from -1 to 4
\put{$= \ -$} at 0 1.5 \put{$,$} at 8 1.5
\put{$\bullet$} at  1 0  \put{$\bullet$} at  1 3
\put{$\bullet$} at 2 0  \put{$\bullet$} at  2 3
\put{$\bullet$} at  3 0 \put{$\bullet$} at  3 3
\put{$\bullet$} at  4 0  \put{$\bullet$} at  4 3
\put{$\bullet$} at  5 0  \put{$\bullet$} at  5 3
\put{$\bullet$} at  6 0  \put{$\bullet$} at  6 3
\put{$\bullet$} at  7 0   \put{$\bullet$} at  7 3

\put{$\cdots$} at 1.5 1.5 \put{$\cdots$} at 4.5 1.5
\put{$\cdots$} at 6.5 1.5
\put{{\scriptsize $k$}} at 3 4

\plot 1 3 1 0 /  \plot 2 3 2 0  / \plot 3 3 3 0 / \plot 4 3 4 0 / \plot 5 3 5 0 / \plot 6 3 6 0
/ \plot 7 3 7 0 /


\setdashes  <.4mm,1mm> \plot 5.5 -1   5.5 4 /
\endpicture}$

${\beginpicture \setcoordinatesystem units <0.78cm,0.39cm>
\setplotarea x from 0 to 7.5, y from -1 to 4 \put{$\mathbf{c}_l^2 =
$} at 0 1.5 \put{$\bullet$} at  1 0   \put{$\bullet$} at  1 3
\put{$\bullet$} at  2 0   \put{$\bullet$} at  2 3 \put{$\bullet$} at
3 0    \put{$\bullet$} at  3 3 \put{$\bullet$} at  4 0
\put{$\bullet$} at 4 3 \put{$\bullet$} at  5 0     \put{$\bullet$}
at  5 3 \put{$\bullet$} at  6 0     \put{$\bullet$} at  6 3
\put{$\bullet$} at  7 0     \put{$\bullet$} at  7 3

\put{$\cdots$} at 1.5 1.5 \put{$\cdots$} at 3.5 1.5 \put{$\cdots$}
at 6.5 1.5
\put{{\scriptsize $l$}} at 5 4

\plot 1 3 1 0 / \plot 2 3 2 0  / \plot 3 3 3 0 / \plot 4 3 4 0 /
\plot 5 3 5 0 / \plot 6 3 6 0 / \plot 7 3 7 0 /

\put{\textcircled{1}} at 4.95 1 \put{\textcircled{2}} at 4.95 2

\setdashes  <.4mm,1mm> \plot 2.5 -1   2.5 4 /
\endpicture}$
${\beginpicture \setcoordinatesystem units <0.78cm,0.39cm>
\setplotarea x from 0 to 8, y from -1 to 4 \put{$= \ $} at 0 1.5
\put{$,$} at 8 1.5 \put{$\bullet$} at  1 0   \put{$\bullet$} at  1 3
\put{$\bullet$} at  2 0   \put{$\bullet$} at  2 3 \put{$\bullet$} at
3 0    \put{$\bullet$} at  3 3 \put{$\bullet$} at  4 0
\put{$\bullet$} at 4 3 \put{$\bullet$} at  5 0     \put{$\bullet$}
at  5 3 \put{$\bullet$} at  6 0     \put{$\bullet$} at  6 3
\put{$\bullet$} at  7 0     \put{$\bullet$} at  7 3

\put{$\cdots$} at 1.5 1.5 \put{$\cdots$} at 3.5 1.5 \put{$\cdots$}
at 6.5 1.5
\put{{\scriptsize $l$}} at 5 4

\plot 1 3 1 0 / \plot 2 3 2 0  / \plot 3 3 3 0 / \plot 4 3 4 0 /
\plot 5 3 5 0 / \plot 6 3 6 0 / \plot 7 3 7 0 /


\setdashes  <.4mm,1mm> \plot 2.5 -1   2.5 4 /
\endpicture}$

and

${\beginpicture \setcoordinatesystem units <0.78cm,0.39cm>
\setplotarea x from 0 to 8, y from -1 to 4 \put{$\mathbf{c}_i
\mathbf{c}_j =$} at -0.2 1.5

\put{$\bullet$} at  1 0   \put{$\bullet$} at  1 3 \put{$\bullet$} at
2 0    \put{$\bullet$} at  2 3 \put{$\bullet$} at  3 0
\put{$\bullet$} at  3 3 \put{$\bullet$} at  4 0     \put{$\bullet$}
at 4 3 \put{$\bullet$} at  5 0     \put{$\bullet$} at  5 3
\put{$\bullet$} at  6 0     \put{$\bullet$} at  6 3 \put{$\bullet$}
at  7 0     \put{$\bullet$} at  7 3 \put{$\bullet$} at  8 0
\put{$\bullet$} at  8 3

\put{$\cdots$} at 1.5 1.5 \put{$\cdots$} at 4.5 1.5 \put{$\cdots$}
at 7.5 1.5
\put{{\scriptsize $i$}} at 3 4 \put{{\scriptsize $j$}} at 6 4

\plot 1 3 1 0 / \plot 2 3 2 0  / \plot 3 3 3 0 / \plot 4 3 4 0 /
\plot 5 3 5 0 / \plot 6 3 6 0 / \plot 7 3 7 0 / \plot 8 3 8 0 /

\put{\textcircled{1}} at 2.95 1.5 \put{\textcircled{2}} at 5.95 1.5
\endpicture}$
${\beginpicture \setcoordinatesystem units <0.78cm,0.39cm>
\setplotarea x from 0 to 8, y from -1 to 4 \put{$= \ - $} at -0.2
1.5 \put{$= \ - \mathbf{c}_j \mathbf{c}_i .$} at 9.5 1.5
\put{$\bullet$} at  1 0   \put{$\bullet$} at  1 3 \put{$\bullet$} at
2 0    \put{$\bullet$} at  2 3 \put{$\bullet$} at  3 0
\put{$\bullet$} at  3 3 \put{$\bullet$} at  4 0     \put{$\bullet$}
at 4 3 \put{$\bullet$} at  5 0     \put{$\bullet$} at  5 3
\put{$\bullet$} at  6 0     \put{$\bullet$} at  6 3 \put{$\bullet$}
at  7 0     \put{$\bullet$} at  7 3 \put{$\bullet$} at  8 0
\put{$\bullet$} at  8 3

\put{$\cdots$} at 1.5 1.5 \put{$\cdots$} at 4.5 1.5 \put{$\cdots$}
at 7.5 1.5
\put{{\scriptsize $i$}} at 3 4 \put{{\scriptsize $j$}} at 6 4

\plot 1 3 1 0 / \plot 2 3 2 0  / \plot 3 3 3 0 / \plot 4 3 4 0 /
\plot 5 3 5 0 / \plot 6 3 6 0 / \plot 7 3 7 0 / \plot 8 3 8 0 /

\put{\textcircled{2}} at 2.95 1.5 \put{\textcircled{1}} at 5.95 1.5
\endpicture}$


Our aim is to prove that the walled Brauer-Clifford superalgebra $\mathsf{BC}_{r,s}$ is
isomorphic to $\BD_{r,s}$.
 Towards this purpose, we adopt the following conventions:
\begin{itemize}
\item[{\rm (i)}] The vertices on the top row (and on the bottom row) of a bead diagram are
 labeled  $1,2, \dots, r+s$ from left to right.
\item[{\rm (ii)}] The top vertex of a vertical strand or the left vertex of a horizontal strand is the  {\it good} vertex of the strand.
\item[{\rm (iii)}] A bead on a horizontal bottom row strand is a bead of \emph{type I}. All other beads are of \emph{type II}.
 \end{itemize}

We construct a bead diagram $\widetilde{d}$ from the bead diagram $d$
by performing the following  steps:
\begin{enumerate}
\item  Keep the same connections between vertices as in $d$.

\item  If the number of beads along a strand is even, delete all beads on that strand.
If there is an odd number of beads on a strand, leave only one bead on it.
Repeat this process for all strands in $d$.
(Hence, there is at most one bead on any strand of $\widetilde{d}$.)
Associate to each remaining bead the good vertex of its strand.

\item  Renumber (starting with $1$) the beads of type I
according to the position of its good vertex from left to right.

\item  Let $m$ be the maximum of bead numbers after Step 3.
Renumber (starting with $m+1$) the beads of type II
according to the position of its good vertex from left to right. \end{enumerate}

The resulting diagram is $\widetilde d$.  Since the definition of $\widetilde{d}$ depends only on the number of beads along a strand
and the good vertices of strands having an odd number of beads, $\widetilde{d}$ does not change
when we slide beads along a strand, so $\widetilde{d}$ is well defined.

\begin{example} \label{ex:d_and_widetilde(d)}
 If  $d$ is as pictured below
$$
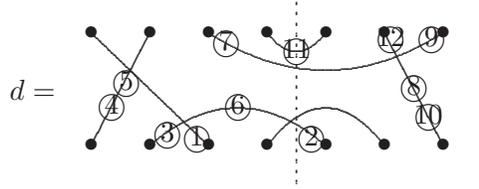
{\beginpicture
\setcoordinatesystem units <0.78cm,0.5cm>
\setplotarea x from 0 to 5.5, y from -1.5 to 4

\put{$d = $} at 0 1.5
\put{$\bullet$} at  1 3 \put{$\bullet$} at  1 0
\put{$\bullet$} at  2 3  \put{$\bullet$} at 2 0
\put{$\bullet$} at  3 3  \put{$\bullet$} at  3 0
\put{$\bullet$} at  4 3  \put{$\bullet$} at  4 0
\put{$\bullet$} at  5 3 \put{$\bullet$} at  5 0
\put{$\bullet$} at  6 3 \put{$\bullet$} at  6 0
\put{$\bullet$} at  7 3 \put{$\bullet$} at  7 0
\put{,} at 7.7 1.5

\plot 2 3  1 0 /
\plot 1 3  3 0 /
\plot 6 3 7 0 /

\put{\textcircled{1}} at 2.8 0.15
\put{\textcircled{2}} at 4.75 0.15
\put{\textcircled{3}} at 2.3 0.25
\put{\textcircled{4}} at 1.35 1
\put{\textcircled{5}} at 1.6 1.65
\put{\textcircled{6}} at 3.5 1
\put{\textcircled{7}} at 3.3 2.7
\put{\textcircled{8}} at 6.5 1.5
\put{\textcircled{9}} at 6.8 2.8
\put{\textcircled{10}} at 6.75 0.75
\put{\textcircled{11}} at 4.5 2.5
\put{\textcircled{12}} at 6.1 2.8

\setdashes  <.4mm,1mm>
\plot 4.5 -1   4.5 4 /
\setsolid

\setquadratic
\plot 3 3  5 2 7 3 /
\plot 4 3  4.5 2.5  5 3 /
\plot 2 0 3.5 1  5 0 /
\plot 4 0 5 1  6 0 /
\endpicture}$$
then
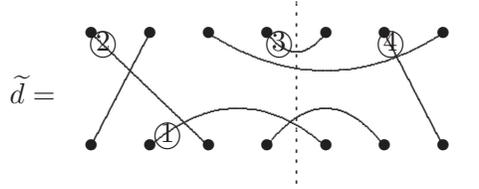
$${\beginpicture
\setcoordinatesystem units <0.78cm,0.5cm>
\setplotarea x from 0 to 5.5, y from -1.5 to 4

\put{$\widetilde{d} = $} at 0 1.5
\put{$\bullet$} at  1 3 \put{$\bullet$} at  1 0
\put{$\bullet$} at  2 3  \put{$\bullet$} at 2 0
\put{$\bullet$} at  3 3  \put{$\bullet$} at  3 0
\put{$\bullet$} at  4 3  \put{$\bullet$} at  4 0
\put{$\bullet$} at  5 3 \put{$\bullet$} at  5 0
\put{$\bullet$} at  6 3 \put{$\bullet$} at  6 0
\put{$\bullet$} at  7 3 \put{$\bullet$} at  7 0
\put{.} at 7.7 1.5

\plot 2 3  1 0 /
\plot 1 3  3 0 /
\plot 6 3 7 0 /

\put{\textcircled{1}} at 2.3 0.25
\put{\textcircled{2}} at 1.2 2.7
\put{\textcircled{3}} at 4.2 2.7
\put{\textcircled{4}} at 6.1 2.7

\setdashes  <.4mm,1mm>
\plot 4.5 -1   4.5 4 /
\setsolid

\setquadratic
\plot 3 3  5 2 7 3 /
\plot 4 3  4.5 2.5  5 3 /
\plot 2 0 3.5 1  5 0 /
\plot 4 0 5 1  6 0 /
\endpicture}$$
\end{example}

 Next we assign a nonnegative integer $\gamma(d)$ to the bead diagram $d$ in the following way:
\begin{enumerate}
\item Assume the bead numbers of type I in $d$ are $1 \le \eta_1 < \eta_2 < \cdots < \eta_p$.
Let $a_{j}$ be the label of the good vertex of the strand in $d$ with
the bead  $\eta_j$.  This determines a  sequence $a_{1}, \ldots, a_{p}$.
Let $\ell_1(d):= | \{(j,k) \ | \ j < k , \ a_{j} > a_{k} \}|$.

\item  Assume the bead numbers of type II in $d$ are $1 \le \vartheta_1 < \vartheta_2 < \cdots < \vartheta_q$.
Let $b_{j}$ be the label of the good vertex of the strand in $d$ with the bead  $\vartheta_j$.
This determines a sequence $b_{1}, \ldots,  b_{q}$. Let $\ell_2(d):= | \{(j,k) \ | \ j < k , \  b_{j} > b_{k} \}|$.

\item  Let $\rho_1(d):= \displaystyle {\sum_{i=1}^r} \ \Bigg \lfloor \frac{ \big |\{  j \in \{ 1, \ldots, p \}\mid  a_j=i \} \big  |}{2} \Bigg \rfloor$, where $\lfloor x \rfloor$ denotes the largest integer not greater than  $x$.

\item Let $\rho_2(d) = \displaystyle {\sum_{i=1}^r} \ \Bigg \lfloor \frac{\big |\{  j \in \{ 1, \ldots, q \}  \ | \  b_j = i \} \big  |}{2} \Bigg \rfloor$.

\item   For each bead  $\eta_j$,  its
{\it passing number} counts the number of  beads $\eta_k$ such that
$\eta_k > \eta_j$ when $\eta_j$ moves to the good vertex of  its
strand.
 Let $p_1(d)$ be the sum of the passing numbers for all $\eta_1, \ldots, \eta_p$.

\item  For each  bead $\vartheta_j$,  its
{\it passing number} counts the number of beads $\vartheta_k$
 with $\vartheta_k < \vartheta_j$ when  $\vartheta_j$ moves to the good vertex of its strand.
 Let $p_2(d)$ be the sum of the passing numbers for all $\vartheta_1, \ldots, \vartheta_q$.

\item  For each bead $\vartheta_j$,  count the number of beads $\eta_k$ such that $\eta_k > \vartheta_j$;  then
let $c(d)$ be the sum of those numbers for all beads $\vartheta_1, \ldots, \vartheta_q$.

\item Let $\alpha(d) := \displaystyle {\sum_{i=r+1}^{r+s}} \ \Bigg \lfloor \frac{ \big | \{  j \in \{ 1, \ldots, q \}  \ | \  b_j = i \} \big  |}{2} \Bigg \rfloor$.

\item  Now set $\beta(d):=\ell_1(d)+\ell_2(d)+\rho_1(d)+\rho_2(d)+p_1(d)+p_2(d)+c(d)$ and $\gamma(d):=\beta(d) + \alpha(d)$.
\end{enumerate}

Since the definition of $\gamma(d)$ depends only on the bead numbers, the number of beads on a strand,
and the good vertex of a strand having a bead,
$\gamma(d)$ is well defined.
All values including $\beta(d),\alpha(d)$, and hence $\gamma(d)$,
are nonnegative integers.

\begin{example} Consider the bead diagram $d$ in Example \ref{ex:d_and_widetilde(d)}.
Three beads $\textcircled{2}, \textcircled{3}, \textcircled{6}$ are of type I, and the rest are of
 type II.
The sequence of labels for the good vertices obtained in Step 1 (resp.  Step 2)
is $a_1,a_2,a_3=2,2,2$ (resp. $b_1, \ldots, b_9=1,2,2,3,6,3,6,4,6$).
From these sequences
we see that $$\ell_1(d)=0, \ \  \ell_2(d)=3, \ \  \rho_1(d)=1, \ \  \rho_2(d)=2,  \text{ and } \alpha(d)=1.$$
\indent When the bead $\textcircled{2}$ moves to the good vertex on its strand,
it must pass the two beads $\textcircled{3}, \textcircled{6}$.
When the beads $\textcircled{3}$ and $\textcircled{6}$ move to that same good vertex,
they do not have to pass a bead with a larger label.  Hence $p_1(d)=2$.
Similarly $p_2(d) = 2$,  since only the beads $\textcircled{9}$ (passing  $\textcircled{7}$)  and $\textcircled{10}$ (passing $\textcircled{8}$)
contribute to  $p_2(d)$.

Only the following beads contribute to $c(d)$:   bead $\textcircled{1}$ with $\textcircled{2}, \textcircled{3}, \textcircled{6}$
and beads $\textcircled{4}$ and $\textcircled{5}$ with $\textcircled{6}$.  Therefore,  $c(d)=5$.

Consequently,   $\beta(d)= 3+1+2+2+2+5 =15, \ \alpha(d)=1$,  and $\gamma(d)=16$.
\end{example}

The set of $(r,s)$-bead diagrams
without any beads or horizontal strands forms  a group under the
multiplication defined in $\widetilde{\BD}_{r,s}$ which is isomorphic to
the product  $\Sigma_r \times \Sigma_s$ of symmetric groups. In what follows,
we identify that group with $\Sigma_r \times \Sigma_s$ but use boldface when
we are regarding an element of  $\Sigma_r \times \Sigma_s$  as a diagram.    We adopt the following conventions
analogous to those for the basis elements of $\mathsf{BC}_{r,s}$, but here we are using the generators for the
subalgebra $\BD_{r,s}'$ of $\widetilde{\BD}_{r,s}$:

For a nonempty subset $A = \{a_1 <  \cdots  < a_m\}$ of $\{1,\dots, r\} \cup \{r+1, \dots, r+s\}$, set  $\mbc_A:= \mbc_{a_1}  \cdots  \mbc_{a_m}$, and let $\mbc_\emptyset = 1$.
Let  $\mbe_{p,q}: = \boldsymbol{\vphi} \mbe_{r,r+1}  \boldsymbol{\vphi}^{-1}$,
where $ \boldsymbol{\vphi}=  \mbs_{q-1} \cdots  \mbs_{r+1}  \mbs_p \cdots  \mbs_{r-1}$
for $1 \le p \le r$ and $r+1 \le q \le r+s$.

\begin{lemma} \label{lem:property_of_widetilde(d)_and_gamma(d)}
\begin{itemize}
\item[{(i)}]  For a bead diagram $d$, the associated diagram $\widetilde{d}$ has an expression of the form
\begin{center} $\mbc_P \mbe_{p_1,q_1} \cdots \mbe_{p_a,q_a} \boldsymbol{\sigma} \mbc_Q,  $\end{center}
  \noindent where
\begin{enumerate}
\item $1 \leq p_1 <\cdots <p_a \leq r$;
\item $r+1 \leq q_i \leq r+s$,  $i=1,\ldots, a$, are all distinct;
\item  $\boldsymbol{\sigma} \in \Sigma_r \times \Sigma_s, \
   \sigma^{-1}(p_1) < \cdots < \sigma^{-1}(p_a)$;  and
\item $P \subseteq \{p_1, \ldots, p_a \}$,
$Q \subseteq \{1, \ldots, r \} \cup \{r+1, \ldots, r+s \} \backslash \{\sigma^{-1}(q_1), \ldots, \sigma^{-1}(q_a) \}$.
\end{enumerate}
Hence,  $\widetilde{d} \in \BD'_{r,s}$.
\item[{(ii)}]  $\gamma(d)=0$ if and only if $d=\widetilde{d}$.
In particular, $\gamma(\widetilde{d})=0$ for all bead diagrams $d$. \end{itemize}
\end{lemma}

\begin{proof}  (i) A diagram without beads can be written as a product
  $\mathbf e_{p_1,q_1} \cdots \mathbf e_{p_a,q_a} \boldsymbol{\sigma}$ which satisfies conditions (1), (2),  and (3).
  Indeed, $a$ is the number of horizontal strands in $\widetilde{d}$.
  The number $p_i$ (resp.~$q_i$) is the label of the  good vertex (resp.~right vertex) of the $i$th horizontal bottom row strand
  from left to right.
  The number  $\sigma^{-1}(p_i)$ (resp. $\sigma^{-1}(q_i)$) is the label of the good vertex
  (resp.~right vertex) of the $i$th horizontal top row strand from left to right.
  For example,
  $$
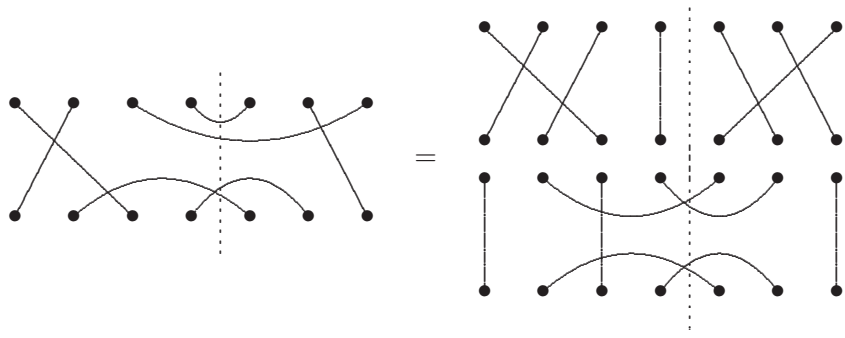
{\beginpicture
\setcoordinatesystem units <0.78cm,0.5cm>
\setplotarea x from 0 to 5.5, y from -1.5 to 4

\put{$\bullet$} at  1 3 \put{$\bullet$} at  1 0
\put{$\bullet$} at  2 3  \put{$\bullet$} at 2 0
\put{$\bullet$} at  3 3  \put{$\bullet$} at  3 0
\put{$\bullet$} at  4 3  \put{$\bullet$} at  4 0
\put{$\bullet$} at  5 3 \put{$\bullet$} at  5 0
\put{$\bullet$} at  6 3 \put{$\bullet$} at  6 0
\put{$\bullet$} at  7 3 \put{$\bullet$} at  7 0

\plot 2 3  1 0 /
\plot 1 3  3 0 /
\plot 6 3 7 0 /

\setdashes  <.4mm,1mm>
\plot 4.5 -1   4.5 4 /
\setsolid

\put{$ = $} at 8 1.5

\put{$\bullet$} at  9 -2 \put{$\bullet$} at  9 1
\put{$\bullet$} at  10 -2  \put{$\bullet$} at 10 1
\put{$\bullet$} at  11 -2  \put{$\bullet$} at  11 1
\put{$\bullet$} at  12 -2  \put{$\bullet$} at  12 1
\put{$\bullet$} at  13 -2 \put{$\bullet$} at  13 1
\put{$\bullet$} at  14 -2 \put{$\bullet$} at  14 1
\put{$\bullet$} at  15 -2 \put{$\bullet$} at  15 1

\plot 9 -2  9 1  /
\plot 11 -2  11 1 /
\plot 15 -2 15 1 /

\setdashes  <.4mm,1mm>
\plot 12.5 -3   12.5 2 /
\setsolid

\put{$\bullet$} at  9 2 \put{$\bullet$} at  9 5
\put{$\bullet$} at  10 2  \put{$\bullet$} at 10 5
\put{$\bullet$} at  11 2  \put{$\bullet$} at  11 5
\put{$\bullet$} at  12 2  \put{$\bullet$} at  12 5
\put{$\bullet$} at  13 2 \put{$\bullet$} at  13 5
\put{$\bullet$} at  14 2 \put{$\bullet$} at  14 5
\put{$\bullet$} at  15 2 \put{$\bullet$} at  15 5
\put{.} at 16 1.5

\plot 9 2  10 5  /
\plot 10 2  11 5 /
\plot 11 2 9 5 /
\plot 12 2  12 5  /
\plot 13 2  15 5 /
\plot 14 2 13 5 /
\plot 15 2  14 5 /

\setdashes  <.4mm,1mm>
\plot 12.5 1   12.5 6 /
\setsolid

\setquadratic
\plot 3 3  5 2 7 3 /
\plot 4 3  4.5 2.5  5 3 /
\plot 2 0 3.5 1  5 0 /
\plot 4 0 5 1  6 0 /

\plot 10 -2  11.5 -1 13 -2 /
\plot 12 -2  13 -1  14 -2 /
\plot 10 1 11.5 0  13 1 /
\plot 12 1 13 0  14 1 /
\endpicture}$$

 From Steps 3 and 4 of the construction of $\widetilde{d}$ from $d$, we have
  $\widetilde{d}=\mbc_P \mbe_{p_1,q_1} \cdots \mbe_{p_a,q_a} \boldsymbol{\sigma} \mbc_Q$,
where $P$ is a set of labels for the good vertices of strands with beads of type I,
and $Q$ is a set of labels for the good vertices of strands with beads of type II in $\widetilde{d}$.
  Since we can slide a bead to the good vertex on its strand, condition (4) above can be satisfied.

(ii) ($\Rightarrow$)
From the assumption that $\gamma(d)=0$, we have that  $\rho_1(d)=\rho_2(d)=\alpha(d)=0$.
Hence, there is at most one bead on each strand in $d$.
Since $c(d)=0$, the bead numbers of the beads of  type I are less than all the bead numbers of the beads of type II.
From  $\ell_1(d)=\ell_2(d)=0$, we deduce that
the sequence of good vertices for the stands with beads of the same type are  arranged in increasing size.
That is, the good vertex  having the  smaller label  is connected to the strand with the bead having the smaller bead number.  From these properties, we determine that nothing is changed when $\widetilde{d}$ is constructed
from $d$, so that $\widetilde{d}=d$.

$(\Leftarrow)$   It is enough to argue that $\gamma(\widetilde{d})=0$.
By Step 2  of the construction of  $\widetilde{d}$, we have
$\rho_1(\widetilde{d})=\rho_2(\widetilde{d})=\alpha(\widetilde{d})=0$.
Also, since there is at most one bead on each strand in $\widetilde{d}$,
 $p_1(\widetilde{d})=p_2(\widetilde{d})=0$.
By Steps 3 and 4, we see that $c(\widetilde{d})=\ell_1(\widetilde{d})=\ell_2(\widetilde{d})=0$.
As a result,  $\beta(\widetilde{d})=\alpha(\widetilde{d})=0$, and hence $\gamma(\widetilde{d})=0$.
\end{proof}

\begin{example}
 For the diagram  $\widetilde{d}$  in Example \ref{ex:d_and_widetilde(d)}, we have
 $p_1=2, p_2=4, q_1=5, q_2=6$,
  $\boldsymbol{\sigma}= \mathbf s_2 \mathbf  s_1 \mathbf  s_5 \mathbf  s_6 $,
  $\sigma^{-1}(2)=3, \sigma^{-1}(4)=4, \sigma^{-1}(5)=7, \sigma^{-1}(6)=5$,
  $P= \{ 2\}$,  and $Q=\{ 1,4,6\}$.
 Consequently, $\widetilde{d}= \mathbf c_2 \mathbf  e_{2,5} \mathbf e_{4,6} \mathbf s_2 \mathbf  s_1 \mathbf  s_5 \mathbf  s_6 \mathbf  c_1 \mathbf  c_4 \mathbf  c_6 \in \BD'_{4,3}$.
  \end{example}

\begin{lemma} \label{lem:L subset M in BD'} The subspace $\msM$ spanned by
$\{d-(-1)^{\beta(d)} \widetilde{d} \mid d \ \hbox{\rm is a bead diagram in} \ \BD'_{r,s}\}$  contains the  two-sided ideal  $\msL$ of
$\BD_{r,s}'$  generated by the elements  $\mbc_k^2+1, \mbc_l^2 -1$,  and
  $\mbc_i \mbc_j + \mbc_j \mbc_i$ in (\ref{def:I_in_BD'}).   \end{lemma}

 \begin{proof}   It suffices to show that $e \mbc_k^2 f+ ef, \  e \mathbf c_l^2f -ef$,
  and $e \mathbf c_i \mathbf c_j f + e \mathbf c_j \mathbf c_i f$ belong to $\msM$ for any two bead diagrams
  $e,f \in \BD'_{r,s}$.

 (i) \,  First, we consider  $e \mathbf c_k^2 f + ef$.   In constructing the diagram  $\widetilde{d}$ from a diagram $d$, we delete an even number of beads along each strand.
  Therefore, $\widetilde{e \mbc_k^2 f}=\widetilde{ef}$.

  If the product  $\mbc_k^2$ in  $e\mbc_k^2f$ creates beads of  type I, then $\rho_1(e\mbc_k^2 f)=\rho_1(ef)+1$.
  In this case, $c(e \mathbf c_k^2 f) \equiv c(ef), \ell_1(e \mathbf c_k^2 f) \equiv \ell_1(ef)$,
  and $p_1(e \mathbf c_k^2 f) \equiv p_1(ef) \mod 2$.
  The other values $\ell_2(e \mathbf c_k^2 f)$, $\rho_2(e \mathbf c_k^2 f)$, $p_2(e \mathbf c_k^2 f)$
are the same as $\ell_2(ef), \rho_2(ef), p_2(ef)$, respectively.
  Thus,  $\beta(e\mathbf c_k^2 f) \equiv \beta(ef)+1 \mod 2$.

  If the product $\mbc_k^2$ in $e\mbc_k^2f$ creates beads along a vertical strand on the right-hand side of the wall, then
  $\rho_2(e \mbc_k^2 f)=\rho_2(ef)$.   In this case,
  $p_2(e \mathbf c_k^2 f) \equiv p_2(ef)+1, \ell_2(e \mathbf c_k^2 f) \equiv \ell_2(ef),
  c(e \mathbf c_k^2 f)\equiv c(ef)  \mod 2$. The other values $\ell_1, \rho_1, p_1$ remain the same for
  $e \mbc_k^2 f$ as for $ef$.
  Hence,  $\beta(e \mathbf c_k^2 f) \equiv \beta(ef)+1 \mod 2$.
 The other cases can be checked in a similar manner.

  As a consequence,
    \begin{equation} \label{eq:ec_k^2f+ef}
    \begin{aligned}
   e \mbc_k^2f +ef &=
    e \mbc_k^2 f  -(-1)^{\beta(e \mbc_k^2 f )} \widetilde{e \mbc_k^2f}
      +(-1)^{\beta(e\mbc_k^2 f )} \widetilde{e \mbc_k^2 f } +ef \\
   & =e  \mbc_k^2 f  -(-1)^{\beta(e \mbc_k^2 f)} \widetilde{e \mathbf c_k^2 f }
     +ef-(-1)^{\beta(ef)} \widetilde{ef} \in \msM.
  \end{aligned}
  \end{equation}

 (ii) \, To verify $e \mbc_l^2 f -ef \in \msM$, we can show that $\widetilde{e \mbc_l^2 f}=\widetilde{ef}$
and $\beta(e \mbc_l^2 f)\equiv \beta(ef) \mod 2$ as in (i) and then apply a calculation similar to that in
\eqref{eq:ec_k^2f+ef}.

 (iii) \, To argue that  $e\mathbf{c}_i \mathbf c_j f + e \mathbf c_j \mathbf c_i f \in \msM$, assume
$\mathbf{c}_i$ creates a bead indexed by $a$ and $\mathbf c_j$  a bead indexed by $a+1$  in $e \mathbf{c}_i \mathbf c_j f$.
If we switch the beads containing $a$ and $a+1$,
we obtain the bead diagram $e \mbc_j \mbc_i f$.
Since the number of beads along each strand does not change,
$\widetilde{e \mbc_i \mbc_j f}=\widetilde{e \mbc_j \mbc_i f}.$

 We will show that $\beta(e \mbc_i \mathbf c_j f) \equiv \beta(e \mbc_i \mbc_j f)+1 \mod 2.$
Suppose the beads created by $\mbc_i$ and $ \mbc_j$  are of different types, say type I
for $\mbc_i$ and type II for $\mbc_j$.
Then $c(e \mbc_i \mbc_j f)=c(e \mbc_j \mbc_i f)-1$.
The other values $\ell_1, \ell_2, \rho_1, \rho_2, p_1, p_2$,  and $\alpha$ are the same
in $e \mbc_i \mbc_j f$ and $e \mbc_j \mbc_i f$.
Hence,  $\beta(e \mbc_i \mbc_j f) \equiv
\beta(e \mbc_j \mbc_i f)+1 \mod 2$.  The reverse possibility ($\mbc_i$ type II
and $\mbc_j$ type I)  can be treated similarly.

Now assume both $\mbc_i$ and $ \mbc_j$ create  beads of type I.     If the beads are on  the same strand, then
$$p_1(e \mbc_i \mbc_j f) \equiv p_1(e\mbc_j \mbc_i f) +1 \mod 2,$$
and the other values do not change.
If they lie on different strands, then
$\ell_1(e \mbc_i \mbc_j f) \equiv \ell_1(e \mbc_j \mbc_i f)+1 \mod 2$, and the other values are unchanged.
Therefore,  $\beta(e \mbc_i \mbc_j f) \equiv \beta(e \mbc_j \mbc_i f)+1 \mod 2.$
The case that $\mbc_i$ and $\mbc_j$ create beads of type II can be handled similarly.

Applying a computation like the one in \eqref{eq:ec_k^2f+ef}, we obtain
$e \mbc_i \mathbf c_j f + e \mbc_j \mbc_i f \in \msM$.
\end{proof}

We now prove the main theorem of this section.

\begin{theorem}\label{th:BC and BD}
   The walled Brauer-Clifford superalgebra $\mathsf{BC}_{r,s}$ is isomorphic to the $(r,s)$-bead diagram algebra $\BD_{r,s}$   as associative superalgebras.
\end{theorem}

\begin{proof}  Using the defining relations, we see that
the linear map  $\phi_{r,s}: \mathsf{BC}_{r,s} \rightarrow \BD_{r,s}$ specified by
\begin{equation}\label{eq:phi-rs}s_i \mapsto \mathbf{s}_i, \ \ s_j \mapsto \mathbf{s}_j, \ \
e_{r,r+1} \mapsto \mathbf{e}_{r,r+1},  \ \ c_k \mapsto \mathbf{c}_k,   \text{  and  }
c_{l} \mapsto \mathbf{c}_l,\end{equation} is
a well-defined superalgebra epimorphism.

Recall that $\BD_{r,s}:=\BD'_{r,s}/\msL$, where $\msL$ is the two-sided ideal
generated by the elements in  (\ref{def:I_in_BD'}).    As $\msL \subseteq \msM$,
by Lemma \ref{lem:L subset M in BD'},  there is  well-defined linear map
$\pi_{r,s}:\BD_{r,s} \rightarrow \BD'_{r,s}/\msM$ such that $\pi_{r,s}(d +\msL)=d+\msM$ for $d\in \BD'_{r,s}$.
Since $\{c_P e_{p_1, q_1} \cdots e_{p_a,q_a} \sigma c_Q \}$ is a basis of $\mathsf{BC}_{r,s}$
by Theorem \ref{repmixtensor} (iii),  we can define a linear map
$\psi_{r,s}: \mathsf{BC}_{r,s} \rightarrow \BD'_{r,s}/\msM$ such that
$$\psi_{r,s}(c_P e_{p_1, q_1} \cdots e_{p_a,q_a} \sigma c_Q)
= \mathbf c_P \mathbf e_{p_1, q_1} \cdots \mathbf e_{p_a,q_a} \boldsymbol{\sigma} \mathbf c_Q +\msM.$$
Moreover,  $\pi_{r,s} \circ \phi_{r,s} =\psi_{r,s}$. Therefore, if  we can show that
$\psi_{r,s}$ is injective,  it will follow that $\phi_{r,s}$ is injective (hence, an isomorphism).

By Lemma \ref{lem:property_of_widetilde(d)_and_gamma(d)}\,(ii),  when  $\gamma(d)=0$ for a bead diagram
$d$,   then
$d-(-1)^{\beta(d)} \widetilde{d}=0$.
Thus, $\msM$ is
spanned by the elements $d-(-1)^{\beta(d)}\widetilde{d}$ with $\gamma(d) \ge 1$.
Note that
$$\gamma(\mathbf c_P \mathbf e_{p_1, q_1} \cdots \mathbf e_{p_a,q_a} \boldsymbol{\sigma} \mathbf c_Q)=0.$$
Therefore,  the set $\{\mathbf c_P \mathbf e_{p_1, q_1} \cdots \mathbf e_{p_a,q_a} \boldsymbol{\sigma} \mathbf c_Q + \msM \}$ of
these elements  is linearly independent in $\BD'_{r,s}/\msM$,
so  $\psi_{r,s}$ is indeed injective.
\end{proof}

\begin{corollary}  The relation $\msL=\msM$ holds.  In particular,
\begin{equation}\label{eq:Lbasis} \{d-(-1)^{\beta(d)} \widetilde{d} \mid d \ \hbox{\rm is a bead diagram in} \ \BD'_{r,s}, \ \gamma(d) \ge 1\}\end{equation}  is a basis of the two-sided ideal $\msL$.
\end{corollary}

\begin{proof}
By Lemma \ref{lem:L subset M in BD'} we have that $\msL \subseteq \msM$.
  Since  the mapping $\phi_{r,s}$ in \eqref{eq:phi-rs}  is an isomorphism and $\pi_{r,s}=\psi_{r,s} \circ \phi^{-1}_{r,s}$,
 we know that $\pi_{r,s}$ is injective.
  Thus, for $m \in \msM$, \  $\pi_{r,s}(m+\msL)=0+\msM$ implies that $m \in\msL$.

  From the proof of Theorem \ref{th:BC and BD}, we have
  that  the set in \eqref{eq:Lbasis} spans $\msM$ ($=\msL$).
  Since $\gamma(d) \ge 1$ and $\gamma(\widetilde{d})=0$,  it follows that \eqref{eq:Lbasis} is
 a linearly independent set.
\end{proof}

\vskip3mm
\section{The quantum walled Brauer-Clifford superalgebra $\BC$}

Let $q$ be an indeterminate and $\C(q)$ be the field of rational functions in $q$.
Set $\V_q = \C(q) \ot_{\C} \V =  \C(q) \ot_{\C} \C(n|n)$.  Corresponding to any $X = \sum_{k} Y_k \ot  Z_k
\in \left(\End_{\C(q)}(\V_q)\right)^{\ot 2}$,  let   $X^{12} = \sum_k  Y_k \ot Z_k \ot \id$,
$X^{13} = \sum_k  Y_k \ot  \id \ot  Z_k$,   and $X^{23} =  \sum_k   \id \ot Y_k \ot  Z_k$ in $\left(\End_{\C(q)}(\V_q) \right)^{\ot 3}$, where $\id = \id_{\V_q}$.

Let $\xi = q - q^{-1}$ and define $S  = \sum_{i,j\in {\tt I}} S_{ij} \ot E_{ij} \in \left(\End_{\C(q)}(\V_q)
\right)^{\ot 2}$ by
\begin{equation}
S = \sum_{i,j\in {\tt I}} q^{(\delta_{ij} + \delta_{i,-j})(1-2|j|)}
E_{ii} \ot E_{jj} + \eps \sum_{i,j\in {\tt I},\, i<j} (-1)^{|i|}
(E_{ji}+E_{-j,-i}) \ot E_{ij}. \label{Sfor2}
\end{equation}
$S$ is known to satisfy the quantum Yang-Baxter equation $S^{12}S^{13}S^{23} = S^{23}S^{13}S^{12}$. In \cite{Ol}, Olshanski constructed  a quantization of $\mfU(\mfq(n))$ of $\mfq(n)$ in terms of $S$.

\begin{definition}{\rm{\cite{Ol}}} \label{Uqq}
The \emph{quantum queer superalgebra} $\mfU_q(\mfq(n))$  is the unital associative
superalgebra over $\C(q)$ generated by elements $\msu_{ij}$ with $i\le j$ and
$i,j\in {\tt I} = \{\pm i  \mid  i=1,\dots, n\}$,  which satisfy the following relations: \begin{equation} \msu_{ii} \msu_{-i,-i} = 1 = \msu_{-i,-i}\msu_{ii}, \qquad  U^{12} U^{13} S^{23} = S^{23} U^{13} U^{12},  \label{rttf} \end{equation} where $U= \sum_{i,j \in {\tt I}, \, i\le j} \msu_{ij} \ot E_{ij}$,
and the last equality holds in $\mfU_q(\mfq(n)) \ot_{\C(q)} \left(\End_{\C(q)}(\V_q)\right)^{\ot 2}$.
The $\Z_2$-degree of $\msu_{ij}$ is $|i|+|j|$.
\end{definition}

By the construction, the assignment $\msu_{ij} \mapsto S_{ij}$ is a representation of $\mfU_q(\mfq(n))$
on $\V_q$ (see \cite[Sec.~4]{Ol}).
The superalgebra $\mfU_q(\mfq(n))$ is a Hopf superalgebra with coproduct $\Delta(U) = U^{13} U^{23} \in (\mfU_q(\mfq(n)))^{\ot 2} \ot_{\C(q)} \End_{\C(q)}(\V_q)$,  or more explicitly, $ \Delta(\msu_{ij}) = \sum_{k \in {\tt I}} (-1)^{(|i|+|k|)(|k|+|j|)} \msu_{ik} \ot \msu_{kj}$. The counit is given by $\varepsilon(U) = 1$ and the  antipode by $U \mapsto U^{-1}$.

Let $\V_q^{r,s}:=\V_q^{\ot r} \ot_{\C(q)} (\V_q^*)^{\ot s}$ be the mixed tensor space of $\V_q$ and $\V_q^*$.
Then $\V_q^{r,s}$ is a representation of $\mfU_q(\mfq(n))$ via the coproduct and antipode mappings.
To describe the structure of the centralizer superalgebra
$\End_{\mathfrak{U}_q(\mathfrak{q}(n))}(\V_q^{r,s})$,
we introduce the \emph{quantum walled Brauer-Clifford superalgebra} $\BC$.

\begin{definition} \label{def:BCq}
  The \emph{quantum walled Brauer-Clifford superalgebra} $\BC$ is the associative
superalgebra over $\C(q)$  generated by even elements  $\mst_1,\mst_2,\ldots, \mst_{r-1},\mst_1^*, \mst_2^*, \ldots, \mst_{s-1}^*$,  $\mse$ and odd elements
   $\msc_1,\msc_2,\ldots, \msc_r, \msc_1^*, \msc_2^*, \ldots, \msc_s^*$
satisfying the following defining relations (for $i,j$ in the allowable range):
\begin{equation}
\begin{aligned}  \label{def:BC(q)}
\allowdisplaybreaks
&\mst_i^2-(q-q^{-1})\mst_i-1=0, \; \mst_i\mst_{i+1}\mst_i = \mst_{i+1} \mst_i \mst_{i+1}, &&(\mst_i^*)^2-(q-q^{-1})\mst_i^*-1=0, \; \mst_i^*\mst^*_{i+1}\mst^*_i = \mst^*_{i+1} \mst^*_i \mst^*_{i+1},\\
& \mst_i \mst_j = \mst_j\mst_i  \quad (|i-j|>1), \;\; \mst_i \mst_j^*=\mst_j^*\mst_i,  &&\mst^*_i \mst^*_j = \mst^*_j\mst^*_i  \quad (|i-j|>1), \\
&\mse^2=0, \; \mse\mst_{r-1}\mse = \mse, \; \mse \mst_j=\mst_j \mse \quad (j \neq r-1),&&  \mse\mst_1^*\mse=\mse, \; \mse \mst_j^*=\mst_j^* \mse \quad (j \neq 1 ), \\
&\mse \mst_{r-1}^{-1}\mst_1^* \mse \mst_1^* \mst_{r-1}^{-1}  =\mst_{r-1}^{-1} \mst_1^* \mse \mst_1^* \mst_{r-1}^{-1}  \mse, &&  \\
&\msc_i^2=-1, \; \msc_i \msc_j = -\msc_j \msc_i \ \ (i\neq j), \; \msc_i\msc_j^*=-\msc_j^*\msc_i,  && (\msc_i^*)^2=1, \; \msc_i^* \msc_j^*=-\msc_j^*\msc_i^* \ \  (i\neq j), \\
&\mst_i\msc_i=\msc_{i+1}\mst_i, \; \mst_i\msc_j=\msc_{j}\mst_i \ \ (j \neq i, i+1),   &&\mst^*_i\msc^*_i=\msc^*_{i+1}\mst^*_i, \; \mst^*_i\msc^*_j=\msc^*_{j}\mst^*_i \ \  (j \neq i, i+1), \\
& \mst_i\msc_j^*=\msc_{j}^*\mst_i, \; \msc_r \mse=\msc_1^* \mse, \; \msc_j \mse=\mse \msc_j \quad ( j \neq r), &&  \mst^*_i \msc_j=\msc_{j}\mst^*_i, \; \mse \msc_r =\mse \msc_1^*, \; \msc_j^* \mse =\mse \msc_j^* \quad (j \neq 1),\\
& \mse \msc_r \mse =0. &&
\end{aligned}
\end{equation}
\end{definition}

\begin{definition} \rm{(i)} \   The \emph{(finite) Hecke-Clifford superalgebra} $\Heckr$ in \cite{Ol}  is the associative superalgebra over $\C(q)$
  generated by the  even  elements $\mst_1,\mst_2,\ldots, \mst_{r-1}$ and  the odd elements
   $\msc_1,\msc_2,\ldots, \msc_r$
with the following defining relations (for allowable $i,j$):
\begin{equation}
\begin{aligned}  \label{def:HC(q)}
\allowdisplaybreaks
&\mst_i^2-(q-q^{-1})\mst_i-1=0, \; \mst_i\mst_{i+1}\mst_i = \mst_{i+1} \mst_i \mst_{i+1}, \\
& \mst_i \mst_j = \mst_j\mst_i  \quad (|i-j|>1), \;\;  \\
&\msc_i^2=-1, \; \msc_i \msc_j = -\msc_j \msc_i \quad (i\neq j), \;  \\
&\mst_i\msc_i=\msc_{i+1}\mst_i, \; \mst_i\msc_j=\msc_{j}\mst_i \quad (j \neq i, i+1).  \\
\end{aligned}
\end{equation}

\rm{(ii)} The \emph{quantum walled Brauer algebra} $\mathsf{H}_{r,s}^0(q)$ in \cite{KM}  is
the associative algebra over $\C(q)$
  generated by the elements $\mst_1,\mst_2,\ldots, \mst_{r-1},  \mst_1^*, \ldots, \mst_{s-1}^*$ and $\mse$
 which satisfy the first four lines in \eqref{def:BC(q)}.
\end{definition}

\begin{remark} \label{rem:original relations in BC(q)}
The relations in the first three lines in \eqref{def:BC(q)} appear
in Definition 2.1 of \cite{KM}.    In line 4 of \eqref{def:BC(q)},  we have the one relation
\begin{align} \label{rel:new}
   \mse\mst_{r-1}^{-1}\mst_1^*\mse \mst_1^*\mst_{r-1}^{-1}
=\mst_{r-1}^{-1}\mst_1^*\mse \mst_1^*\mst_{r-1}^{-1} \mse.
 \end{align}
instead of the following two relations of \cite{KM}:
  \begin{align}\label{rel:original}
   \mse\mst_{r-1}^{-1}\mst_1^*\mse\mst_{r-1}=\mse\mst_{r-1}^{-1}\mst_1^*\mse\mst_1^*, \quad \quad
    \mst_{r-1}\mse\mst_{r-1}^{-1}\mst_1^*\mse=\mst_1^*\mse\mst_{r-1}^{-1}\mst_1^*\mse.
 \end{align}

The relations in  \eqref{rel:original} can be derived using \eqref{rel:new} and various identities from \eqref{def:BC(q)}
(especially the fact that $\mst_{r-1}$ and $\mst_1^*$ commute) in the following way:
 \begin{align*}
  \mse \mst_{r-1}^{-1} \mst_1^* \mse&= (\mse \mst_{r-1}^{-1}\mse) \  \mst_{r-1}^{-1} \mst_1^* \mse
                    = \mse ((\mst_1^*)^{-1}\mst_1^*)  \ \mst_{r-1}^{-1}  \mse \mst_{r-1}^{-1} \mst_1^*\mse \\
                   & =  \mse (\mst_1^*)^{-1} \ ( \mst_1^*  \ \mst_{r-1}^{-1}  \mse \mst_{r-1}^{-1} \mst_1^*\mse )
                    = \mse (\mst_1^*)^{-1} \ ( \mse \mst_1^*  \ \mst_{r-1}^{-1}  \mse \mst_{r-1}^{-1} \mst_1^* ) \\
                        &= (\mse (\mst_1^*)^{-1} \mse) \ \mst_1^* \mst_{r-1}^{-1} \mse \mst_{r-1}^{-1}  \mst_1^*
                        =  \mse \mst_1^* \mst_{r-1}^{-1} \mse \mst_{r-1}^{-1} \mst_1^*  = \mse\mst_{r-1}^{-1}\mst_1^*\mse \mst_1^*
                        \mst_{r-1}^{-1} = \mst_{r-1}^{-1}\mst_1^* \mse   \mst_{r-1}^{-1} \mst_1^*  \mse, \end{align*}
implying both relations in \eqref{rel:original}.
\end{remark}

The following simple expression will be useful in several calculations henceforth.
\begin{lemma}
With the conventions  $(-1)^{|0|}=0$, \  $|j| = 1$ for any integer  $j < 0$, \  and $|j| = 0$ for $j > 0$,
we have
\begin{equation}\label{sumq}
    \eps\sum_{i < j < k}(-1)^{|j|}q^{2j(1-2|j|)}=q^{(2k-1)(1-2|k|)}-q^{(2i+1)(1-2|i|)}
\end{equation}
for any nonzero integers $i < k$, where $\eps = q-q^{-1}$ as above.
\end{lemma}
\begin{proof} This can be checked by considering the three cases $0 < i < k$, $i< k <0$ and $i < 0 < k$.
\end{proof}

In order to construct an action of $\BC$  on the mixed tensor space $\V_q^{r,s}$,
we will need a number of $\mfU_q(\mfq(n))$-module homomorphisms.
Note that $\C(q)$ becomes a $\mfU_q(\mfq(n))$-module by sending $U$ to the identity map
 in $\End_{\C(q)}(\C(q)\otimes_{\C(q)}\V_q)$.

\begin{lemma}
There are $\mfU_q(\mfq(n))$-module homomorphisms
$\cap:\C(q)\rightarrow \V_q\ot \V_q^*$ and $\cup:\V_q\ot \V_q^*\rightarrow\C(q)$ given by
\begin{equation*}
    \cap(1) = \sum_{i \in {\tt I}} v_i \ot \w_i,\quad  \cup(v_i\ot \w_j) = (-1)^{|i|}q^{2i(1-2|i|)-(2n+1)}\delta_{ij}.
\end{equation*}
\end{lemma}

\begin{proof}
Since $\cap$ is the canonical map $\C(q)\rightarrow \V_q\ot \V_q^*$, we have
\begin{equation}\label{tcap}
    (X \ot \id)\cap=(\id \ot X^{\tt T})\cap
\end{equation}
for any $X \in\End_{\C(q)}(\V_q)$, where ${}^{\tt T}$ denotes the supertranspose. Since $\cap$ is even, it follows that
\[
    \left((S^{23})^{-1}\right)^{{\tt T}_2}(\cap\ot\id)=(S^{13})^{-1}(\cap\ot\id)
\]
in $\Hom_{\C(q)}(\C(q) \ot \V_q, \V_q\ot \V_q^*\ot \V_q)$, where ${}^{{\tt T}_2}$ indicates taking the supertranspose on the second factor. Thus
\[
    S^{13}\left((S^{23})^{-1}\right)^{{\tt T}_2}(\cap\ot\id)=(\cap\ot\id).
\]
The action of $\mfU_q(\mfq(n))$ on $\V_q\ot \V_q^*$  (resp. on $\C(q)$) is defined by
sending $U$ to $S^{13}\left((S^{23})^{-1}\right)^{{\tt T}_2}$ (resp.  to $\id$),
 so this shows that $\cap$ is a $\mfU_q(\mfq(n))$-module homomorphism.

To check that $\cup$ is a homomorphism, we require an explicit expression for $S^{13}\left((S^{23})^{-1}\right)^{{\tt T}_2}$. We have
\[
    S^{-1}=\sum_{i,j\in {\tt I}} q^{-(\delta_{ij} + \delta_{i,-j})(1-2|j|)} E_{ii} \ot E_{jj}
        -\eps \sum_{i,j\in {\tt I},\, {i < j}} (-1)^{|i|}  (E_{ji}+E_{-j,-i}) \ot E_{ij}.
\]
If we identify $\V_q$ with $\V_q^*$ via $v_i\mapsto \w_i$, then $(S^{-1})^{{\tt T}_1}$ becomes identified
with an endomorphism $S^*$ of $\V_q\ot \V_q$ given by
\begin{equation}\label{S*for}
    S^*=\sum_{i,j\in {\tt I}} q^{-(\delta_{ij} + \delta_{i,-j})(1-2|j|)} E_{ii} \ot E_{jj}
        -\eps \sum_{i,j\in {\tt I},\, i < j} (-1)^{|i||j|}  ((-1)^{|i|+|j|}E_{ij}+E_{-i,-j}) \ot E_{ij}.
\end{equation}
Therefore,  identifying $\V_q\ot \V_q^*$ with $\V_q\ot \V_q$, we have that the action on $\V_q\ot \V_q$ is defined by sending $U$ to
\begin{align*}
    S^{13}(S^*)^{23}
        = & \;\sum_{i,j,k \in {\tt I}} q^{(\delta_{ij} + \delta_{i,-j}-\delta_{jk} - \delta_{j,-k})(1-2|j|)}  E_{ii} \ot E_{kk} \ot E_{jj} \\
        &\;\hspace{-.2truecm} - \eps \sum_{i\in {\tt I}} \,\sum_{j,k \in {\tt I}, j< k} (-1)^{|k||j|}  q^{(\delta_{ij} + \delta_{-i,j})(1-2|j|)}  E_{ii} \ot \big((-1)^{|j|+|k|}E_{jk} + E_{-j,-k}\big) \ot E_{jk} \\
        & \;\hspace{-.2truecm}+ \eps \sum_{j,k \in {\tt I}, j < k}\, \sum_{i\in {\tt I}} (-1)^{|j|} q^{-(\delta_{ik}+\delta_{i,-k})(1-2|k|)}\big(E_{kj} + E_{-k,-j}\big) \ot E_{ii} \ot E_{jk}  \\
        &\;\hspace{-.2truecm}- \eps^2 \sum_{i,j,k \in {\tt I},\, j <  i < k}(-1)^{|i||j|+|j||k|+|j|}\big(E_{ij} + E_{-i,-j}\big) \ot \big((-1)^{|k|}E_{ik}+ (-1)^{|i|} E_{-i,-k}\big) \ot E_{jk}.
\end{align*}
The map $\cup$ can be identified with the map $q^{-(2n+1)}\sum_{i\in {\tt I}} q^{2i(1-2|i|)}\w_i\ot \w_i$,
and $\w_kE_{ij} = \delta_{k,i}\w_j.$   Moreover, direct calculations show
\begin{align*}
    &q^{2n+1}(\cup\ot\id)S^{13}(S^*)^{23}-q^{2n+1}(\cup\ot\id)\hspace{-55mm}\\
        &  \;\; =-\eps\sum_{j < k} (-1)^{|k||j|}q^{(1-2|j|)}q^{2j(1-2|j|)}\big((-1)^{|j|+|k|}\w_j\ot  \w_k+ \w_{-j}\ot \w_{-k}\big) \ot E_{jk} \\
        &  \;\;{}+\eps\sum_{j < k}(-1)^{|j|} q^{-(1-2|k|)}q^{2k(1-2|k|)}\big((-1)^{|k|(|k|+|j|)} \w_j\ot \w_k+(-1)^{(|k|+1)(|k|+|j|)} \w_{-j}\ot \w_{-k}\big)\ot E_{jk}  \\
        &  \;\;{}-\eps^2 \sum_{j< i <k}(-1)^{|i||j|+|j||k|+|j|}q^{2i(1-2|i|)}\big((-1)^{|k|+|i|(|i|+|j|)} \w_j\ot \w_k\\
        &  \;\;\hspace{10mm}{}+ (-1)^{|i|+(|i|+1)(|i|+|j|)} \w_{-j}\ot \w_{-k}\big) \ot E_{jk}\\
        &  \;\;=\eps\sum_{j < k}\Bigg(q^{(2k-1)(1-2|k|)}-q^{(2j+1)(1-2|j|)}-\eps\sum_{k>i>j}(-1)^{|i|}q^{2i(1-2|i|)}\Bigg)\\
        &  \;\;\hspace{10mm}\left((-1)^{|k||j|+|k|+|j|} \w_j\ot \w_k+(-1)^{|k||j|} \w_{-j}\ot \w_{-k}\right)\ot E_{jk}\\
        &  \;\;=0\text{ by (\ref{sumq})}.
\end{align*}
Therefore $(\cup\ot\id)S^{13}(S^*)^{23} =(\cup\ot\id)$, so $\cup$ defines a $\mfU_q(\mfq(n))$-module homomorphism.
\end{proof}

\begin{theorem}\label{BCqact}
There is an action of $\BC$ on $\V_q^{r,s}$
which supercommutes with the action of $\mfU_q(\mfq(n))$, such that the action of each generator is given by
\begin{eqnarray*}
    \mst_i&\mapsto&\id^{\otimes(i-1)}\otimes PS\otimes\id^{\otimes(r+s-1-i)},\qquad \qquad  \msc_i \, \mapsto \, \id^{\otimes(i-1)}\otimes \Iphi\otimes\id^{\otimes(r+s-i)},\\
    \mst_i^*&\mapsto&\id^{\otimes(r+i-1)}\otimes P^{\tt T}S^{\tt T}\otimes\id^{\otimes(s-1-i)},\quad \, \qquad   \msc_i^*\, \mapsto \, \id^{\otimes(r+i-1)}\otimes \Iphi^{\tt T}\otimes\id^{\otimes(s-i)},\\
    \mse& \mapsto &\id^{\otimes(r-1)}\otimes\cap\cup\otimes\id^{\otimes(s-1)},
\end{eqnarray*}
where $\id = \id_{\V_q}$,
\begin{equation} \label{eqn:Iphi}
    P = \sum_{i,j\in {\tt I}}(-1)^{|j|} E_{ij} \ot E_{ji}\in\End(\V_q\ot \V_q),\quad  \Iphi = \sum_{i\in {\tt I}} (-1)^{|i|} E_{i,-i} \in  \End(\V_q),
\end{equation}
and  ${}^{\tt T}$ is the supertranspose. Explicitly, identifying $\V_q$ with $\V_q^*$ as above, we have
\begin{eqnarray}
    PS&=&\sum_{i,j \in {\tt I}}(-1)^{|i|}q^{(\delta_{ij} + \delta_{-i,j})(1-2|j|)} E_{ji} \ot E_{ij} \notag \\
    && \qquad
        +\eps \sum_{{i,j \in \tt I}, i < j}  (E_{ii}\ot E_{jj}-  (-1)^{|i|+|j|}E_{i,-i}\ot E_{-j,j}), \label{eqn:PS} \\
    P^{\tt T}S^{\tt T}&=&\sum_{i,j \in {\tt I}}(-1)^{|i|}q^{(\delta_{ij} + \delta_{-i,j})(1-2|j|)} E_{ji} \ot E_{ij}
        +\eps \sum_{i,j \in {\tt I}, j < i}(E_{ii}\ot E_{jj}-E_{i,-i}\ot E_{-j,j}), \notag \\
    \Iphi^{\tt T}&=&\sum_{i \in \tt I} E_{i,-i},\; \; \cap\cup = q^{-(2n+1)}\sum_{i,j \in {\tt I}}(-1)^{|i||j|}q^{2j(1-2|j|)}E_{ij}\ot E_{ij}. \notag
\end{eqnarray}
\end{theorem}

\begin{proof}
The fact that the actions of $\mst_i$ and $\msc_i$ are $\mfU_q(\mfq(n))$-module endomorphisms satisfying the relations of $\Heckr$ is shown in \cite{Ol}. Consider the linear map given by the cyclic permutation $    \sigma:\V_q^{\otimes s} \rightarrow \V_q^{\otimes s}, \; v_1\ot\cdots\ot v_s \mapsto (-1)^{\sum_{i<j}|v_i||v_j|}v_s\ot\cdots\ot v_1$. Conjugating by $\sigma$, we obtain another action of $\Hecks$ on $\V_q^{\ot s}$ specified  by
\begin{equation*}
    \mst_i \mapsto \id^{\otimes(s-1-i)}\otimes SP\otimes\id^{\otimes(i-1)},\quad  \msc_i \mapsto \id^{\otimes(s-i)}\otimes {\Iphi} \otimes{\id}^{\otimes(i-1)}.
\end{equation*}
These maps are also $\mfU_q(\mfq(n))$-module endomorphisms (even though $\sigma$ is not). Applying the antiautomorphism of  $\Hecks$ that sends $\mst_i$ to $\mst_{s-i}$ and $\msc_i$ to $\msc_{s+1-i}$, we see that
\begin{equation*}
\mst_i^* \mapsto \id^{\otimes(i-1)}\otimes P^{\tt T}S^{\tt T}\otimes\id^{\otimes(s-1-i)},\quad    \msc_i^* \mapsto \id^{\otimes(i-1)}\otimes {\Iphi}^{\tt T}\otimes\id^{\otimes(s-i)}
\end{equation*}
satisfy the required relations.

Since $\cap$ and $\cup$ are $\mfU_q(\mfq(n))$-module homomorphisms, the same is true of $\mse$. We have
\begin{equation*}
    \cup(\id\ot \Iphi^{\tt T}) = q^{-(2n+1)}\sum_{i\in{\tt I}}q^{2i(1-2|i|)} \w_{-i}\ot \w_i = \cup(\Iphi \ot\id)
\end{equation*}
Thus,  $\cup(\id\ot \Iphi^{\tt T})=\cup(\Iphi\ot\id)$, so $\mse \msc_{r}=\mse \msc_1^*$. Also $(\id\ot \Iphi^{\tt T})\cap=(\Iphi\ot\id)\cap$ by (\ref{tcap}), so $\msc_{r}\mse=\msc_1^*\mse$. We have
\begin{eqnarray*}
    \cup\cap(1)
        &=&q^{-(2n+1)}\Bigg(\sum_{j\in{\tt I}}q^{2j(1-2|j|)} \w_j\ot \w_j\Bigg)\Bigg(\sum_{i\in{\tt I}}v_i\ot v_i\Bigg)\\
        &=&q^{-(2n+1)}\sum_{i\in{\tt I}}(-1)^{|i|}q^{2i(1-2|i|)} = 0 \\
    \cup({\Iphi}\ot\id)\cap(1)
        &=&q^{-(2n+1)}\Bigg(\sum_{i\in{\tt I}}q^{2i(1-2|i|)} \w_{-i}\ot \w_i\Bigg)\Bigg(\sum_{j\in{\tt I}}v_j\ot v_j\Bigg) = 0,
\end{eqnarray*}
so $\mse^2=\mse \msc_{r}\mse=0$. Now using $q^{2n+1}\cup(E_{ij}\ot\id)\cap(1)=(-1)^{|i|}q^{2i(1-2|i|)}\delta_{ij}$ and identifying $\V_q$ with $\V_q\ot\C(q)$, we have

\begin{eqnarray*}
    q^{2n+1}(\id\ot\cup)(PS\ot\id)(\id\ot\cap)
        &=&\sum_{j\in{\tt I}}q^{(2j+1)(1-2|j|)}E_{jj}+\eps\sum_{i,j \in{\tt I}, \, j < i}(-1)^{|i|}q^{2i(1-2|i|)}E_{jj}\\
        &=&\sum_{j\in{\tt I}}\left(q^{(2j+1)(1-2|j|)}+\eps\sum_{i,j \in {\tt I}, \, j < i}(-1)^{|i|}q^{2i(1-2|i|)}\right)E_{jj}\\
        &=&q^{2n+1}\id \ \ \text{ by (\ref{sumq})}.
\end{eqnarray*}
Thus $(\id\ot\cap\cup)(PS\ot\id)(\id\ot\cap\cup)=(\id\ot\cap\cup)$, so $\mse\mst_{r-1}\mse=\mse$. Similarly,
\begin{equation*}
    q^{2n+1}(\cup\ot\id)(\id\ot P^{\tt T}S^{\tt T})(\cap\ot\id)
         = \sum_jq^{(2j+1)(1-2|j|)}E_{jj}+\eps\sum_{i > j}(-1)^{|i|}q^{2i(1-2|i|)}E_{jj} =q^{2n+1}\id.
\end{equation*}
Thus $(\cap\cup\ot\id)(\id\ot P^{\tt T}S^{\tt T})(\cap\cup\ot\id)=(\cap\cup\ot\id)$, so $\mse \mst_1^*\mse=\mse$.

Identifying $\V_q\ot \V_q^*$ with $\V_q\ot\C(q)\ot \V_q^*$, we have
\begin{eqnarray*}
    (P\ot\id_{\V_q^*\ot \V_q^*})(\id_{\V_q}\ot\cap\ot\id_{\V_q^*})\cap(1)
        &=&(P\ot\id)\left(\sum_{i,j}v_j\ot v_i\ot \w_i\ot \w_j\right)\\
        &=&\sum_{i,j}(-1)^{|i||j|}v_i\ot v_j\ot \w_i\ot \w_j.
\end{eqnarray*}
The above is the canonical map $\C(q)\rightarrow \V_q\ot \V_q\ot \V_q^*\ot \V_q^*$, so by the same reasoning that led to (\ref{tcap}), we know
\[
    (SP\ot\id)(P\ot\id)(\id\ot\cap\ot\id)\cap=(\id\ot P^{\tt T}S^{\tt T})(P\ot\id)(\id\ot\cap\ot\id)\cap.
\]
Thus
\begin{equation}\label{Tee=T*ee}
    (PS\ot\id)(\id\ot\cap\ot\id)\cap=(\id\ot P^{\tt T}S^{\tt T})(\id\ot\cap\ot\id)\cap.
\end{equation}
To prove the corresponding expression for $\cup$, we must explicitly compute the following:
\begin{eqnarray*}
    q^{4n+2}\cup(\id\ot\cup\ot\id)(PS\ot\id)\hspace{-40mm}\\
        &=&q^{4n+2}\cup(\id\ot\cup\ot\id)\bigg(\sum_{i,j}(-1)^{|i|}q^{(\delta_{ij} + \delta_{-i,j})(1-2|j|)} E_{ji} \ot E_{ij}\ot\id\ot\id\\
        &&{}\qquad  +\eps\sum_{i > j}(E_{jj}\ot E_{ii}-(-1)^{|i|+|j|}E_{j,-j}\ot E_{-i,i})\ot\id\ot\id\bigg)\\
        &=&\sum_{i,j}(-1)^{|i||j|}q^{(\delta_{ij} + \delta_{-i,j}+2j)(1-2|j|)+2i(1-2|i|)}  \w_i\ot \w_j\ot \w_i\ot \w_j\\
        &&\hspace{-5mm} +\eps\sum_{i > j}q^{2i(1-2|i|)+2j(1-2|j|)}\left(\w_j\ot \w_i\ot \w_i\ot \w_j+(-1)^{|j|} \w_{-j}\ot \w_i\ot \w_{-i}\ot \w_j\right),
\end{eqnarray*}
\begin{eqnarray*}
    q^{4n+2}\cup(\id\ot\cup\ot\id)(\id\ot P^{\tt T}S^{\tt T})\hspace{-30mm}\\
        &=&q^{4n+2}\cup(\id\ot\cup\ot\id)\bigg(\sum_{i,j}(-1)^{|i|}q^{(\delta_{ij} + \delta_{-i,j})(1-2|j|)} \id\ot\id\ot E_{ji} \ot E_{ij}\\
        &&{}+\eps \sum_{i > j}\id\ot\id\ot\left(E_{ii}\ot E_{jj}-E_{i,-i}\ot E_{-j,j}\right)\bigg)\\
        &=&q^{2n+1}\cup\bigg(\sum_{i,j}(-1)^{|i|}q^{(\delta_{ij} + \delta_{-i,j}+2j)(1-2|j|)} \id\ot \w_j\ot \w_i\ot E_{ij}\\
        &&{}+\eps \sum_{i > j}q^{2i(1-2|i|)}\left(\id\ot \w_i\ot \w_i\ot E_{jj}-\id\ot \w_i\ot \w_{-i}\ot E_{-j,j}\right)\bigg)\\
        &=& q^{4n+2}\cup(\id\ot\cup\ot\id)(PS\ot\id).
\end{eqnarray*}
Thus,
\begin{equation}\label{eeT=eeT*}
    \cup(\id\ot\cup\ot\id)(PS\ot\id)=\cup(\id\ot\cup\ot\id)(\id\ot P^{\tt T}S^{\tt T}).
\end{equation}
Finally,  using
\[
    S^{-1}P=\sum_{i,j}(-1)^{|j|}q^{-(\delta_{ij}+\delta_{-i,j})(1-2|j|)}E_{ij}\ot E_{ji}-\eps\sum_{i>j} \left( E_{ii}\ot E_{jj}+(-1)^{|i|+|j|}E_{-i,i}\ot E_{j,-j}\right)
\]
and
\[
    q^{2n+1}\cup(E_{ij}\ot E_{kl})\cap=\delta_{ik}\delta_{jl}(-1)^{|i||j|+|i|+|j|}q^{2i(1-2|i|)},
\]
we have (with the help of (\ref{sumq}))
\begin{align*}
    q^{2n+1}(\id\ot\cup\ot\id)&(S^{-1}P\ot P^{\tt T}S^{\tt T})(\id\ot\cap\ot\id)\hspace{-10mm}\\
        &=\sum_{i,j} q^{2j(1-2|j|)}(-1)^{|i||j|}E_{ij}\ot E_{ij}\\
        &{}+\eps\sum_{i>j}(q^{(2i-1)(1-2|i|)} - q^{(2j+1)(1-2|j|)})\left(E_{ii}\ot E_{jj}+(-1)^{|i|}E_{-i,i}\ot E_{-j,j}\right)\\
        &{}-\eps^2\sum_{i>k>j}(-1)^{|k|}q^{2k(1-2|k|)}\left(E_{ii}\ot E_{jj}+(-1)^{|i|}E_{-i,i}\ot E_{-j,j}\right)\\
        &=q^{2n+1}\cap\cup.
\end{align*}
Combining this with (\ref{eeT=eeT*}) gives
\begin{eqnarray*}
    (\id\ot\cap\cup\ot\id)(S^{-1}P\ot P^{\tt T}S^{\tt T})(\id\ot\cap\cup\ot\id)(PS\ot\id)\hspace{-10mm}\\
        &&\hspace{-40mm =}(\id\ot\cap\ot\id)\cap\cup(\id\ot\cup\ot\id)(PS\ot\id)\\
        &&\hspace{-40mm =}(\id\ot\cap\cup\ot\id)(S^{-1}P\ot P^{\tt T}S^{\tt T})(\id\ot\cap\cup\ot\id)(\id\ot P^{\tt T}S^{\tt T}).
\end{eqnarray*}
Thus,  $\mse\mst_{r-1}^{-1}\mst_1^*\mse\mst_{r-1}=\mse\mst_{r-1}^{-1}\mst_1^*\mse\mst_1^*$. Similarly combining with (\ref{Tee=T*ee}) shows that
\begin{eqnarray*}
    (PS\ot\id)(\id\ot\cap\cup\ot\id)(S^{-1}P\ot P^{\tt T}S^{\tt T})(\id\ot\cap\cup\ot\id)\\
        &&\hspace{-40mm}=(PS\ot\id)(\id\ot\cap\ot\id)\cap\cup(\id\ot\cup\ot\id)\\
        &&\hspace{-40mm} =(\id\ot P^{\tt T}S^{\tt T})(\id\ot\cap\cup\ot\id)(S^{-1}P\ot P^{\tt T}S^{\tt T})(\id\ot\cap\cup\ot\id).
\end{eqnarray*}
Hence,  $\mst_{r-1}\mse\mst_{r-1}^{-1}\mst_1^*\mse=\mst_1^*\mse\mst_{r-1}^{-1}\mst_1^*\mse$,
and it follows that $\mse\mst_{r-1}^{-1}\mst_1^*\mse\mst_1^* \mst_{r-1}^{-1}= \mst_{r-1}^{-1}\mst_1^*\mse\mst_{r-1}^{-1}\mst_1^*\mse $.
\end{proof}

\begin{proposition}  \label{rem:classical limit of BC(q)}
The walled Brauer-Clifford superalgebra $\mathsf{BC}_{r,s}$ is the {\it classical limit} of the
quantum walled Brauer-Clifford superalgebra $\mathsf{BC}_{r,s}(q)$. \end{proposition}

\begin{proof}
To see this, let $\mcR =\mathbb{C}[q,q^{-1}]_{(q-1)}$ be the localization of  $\mathbb{C}[q,q^{-1}]$ at the ideal generated by  $q-1$. Let $\BCloc$ be the $\mathcal{R}$-subalgebra of $\BC$ generated  by
$\mst_1, \ldots, \mst_{r-1}, \msc_1,\ldots, \msc_r, \mst_1^*, \ldots, \mst_{s-1}^*$, $\msc_1^*, \ldots, \msc_s^*, \mse$. Let $\V_{\mcR} = \mcR \ot_\C \V$ and set $\V_{\mcR}^{r,s} = \mcR \ot_{\C} \V^{r,s}$.

It follows from \cite[Thm.~5.1]{JK} that  there is a natural
epimorphism from the walled Brauer-Clifford superalgebra $\msB\msC_{r,s}$
onto $(\mcR/(q-1)\mcR) \otimes_{\mathcal{R}} \BCloc \cong
\BCloc/(q-1)\BCloc$ and hence a natural epimorphism $$\pi: \
\msB\msC_{r,s} \twoheadrightarrow \BCloc/(q-1)\BCloc.$$    We want
to argue that $\pi$ is an isomorphism.

Let $\rho_n^{r,s}: \ \msB\msC_{r,s} \lra \End_{\mathbb{C}}(\V^{r,s})^{\rm{op}}$ be
the representation given just before Theorem \ref{repmixtensor}. The action
of $\BC$ on $\V_q^{r,s}$ defined in Theorem \ref{BCqact} restricts
to a representation $\rho_{n,{\mcR}}^{r,s}: \BCloc  \lra
\End_{\mathcal{R}}(\V_{\mcR}^{r,s})$. Let $\ol{\rho}_{n,{\mcR}}^{r,s}$ be
the homomorphism
\[
    \BCloc/(q-1)\BCloc \lra \End_{\mcR}(\V_{\mcR}^{r,s})/(q-1)\End_{\mcR}(\V_{\mcR}^{r,s}).
\]
Since $\V_{\mcR}^{r,s}$ is a free $\mcR$-module, the algebra
$\End_{\mcR}(\V_{\mcR}^{r,s})$ is also free; thus, it is  possible
to identify
$\End_{\mcR}(\V_{\mcR}^{r,s})/(q-1)\End_{\mcR}(\V_{\mcR}^{r,s})$
with $\End_{\C}(\V^{r,s})$.

Let $\jmath$ be the anti-involution of $\msB\msC_{r,s}$ which fixes each generator.
Upon the previous identification, the composite $\ol{\rho}_{n,\mcR}^{r,s}
\circ \pi$ is equal to $\rho_n^{r,s} \circ \jmath$,  as can be checked from the action of
the generators on the mixed tensor space given in Theorem
\ref{BCqact}. (Setting $q=1$ in the formula  for $S$ in \eqref{Sfor2}
gives the identity map).  When $n\ge r+s$, the map $\rho_n^{r,s}$ is known
to be injective by \cite[Thm.~4.5]{JK}.  Therefore, if $n\ge
r+s$, then $\ol{\rho}_{n,{\mcR}}^{r,s} \circ \pi$ must also be injective,
hence so is $\pi$.  In conclusion, $\pi$ is an isomorphism.
\end{proof}

In the proof of Theorem 5.1 in \cite{JK}, a vector space basis of the walled Brauer-Clifford superalgebra
$\mathsf{BC}_{r,s}$ is constructed.
In this section, we obtain a basis of $\BC$ which specializes to the one in \cite{JK} when $q \mapsto 1$ (in a suitable sense).
 Our basis is inspired by the basis of the quantum walled Brauer algebra
 $\mathsf{H}_{r,s}^n(q)$ constructed in Section 2 of \cite{KM} (see Corollary \ref{subHC} below).

\begin{definition}{\rm \cite{KM}}\label{mondef}
A \emph{monomial $\msn$ in normal form in the generators
$\mst_1,\mst_2, \ldots, \mst_{r-1}$} is a product of the form $\msn =\msp_1 \msp_2 \cdots \msp_{r-1}$,
where $\msp_i = \mst_i^{-1} \mst_{i-1}^{-1} \cdots \mst_j^{-1}$ for some $j$ with $1\le j\le i+1$.
(If $j=i+1$, then $\msp_i=1$.) A \emph{monomial $\msn^*$ in normal form in the generators
$\mst_1^*,\mst_2^*, \ldots, \mst_{s-1}^*$} is a product of the form $\msn^* = \msp_1^* \msp_2^* \cdots \msp_{s-1}^*$, where $\msp_i^* =\mst_i^{*} \mst_{i-1}^{*} \cdots \mst_j^{*}$ for some $j$ with $1\le j\le i+1$. (If $j=i+1$, then $\msp_i^*=1$.)
\end{definition}

\begin{definition}  \label{BDef}
Suppose that $I = (i_1,\ldots, i_a)$ with $1\le i_1 < \cdots < i_a \le r$,
$J = (j_1, \ldots, j_a)$ with $0 \le j_k \le s-1$ for $k=1,\ldots, a$, and
if $k_1\neq k_2$, then $j_{k_1} \neq j_{k_2}$.
Let $\wt{I} \subseteq I$ and $\wt{J} \subseteq \{ 1,2,\ldots, r+s  \} \setminus J$.

A \emph{monomial $\msm$ in normal form in $\BC$} is one of the form
\begin{equation*}
\msm  = \msc_{\wt{I}}  \left( \prod_{k=1,\ldots,a}^{\longrightarrow}
\mst_{j_k}^*  \mst_{j_k-1}^* \cdots  \mst_1^*  \mst_{i_k}^{-1} \cdots  \mst_{r-2}^{-1}  \mst_{r-1}^{-1} \mse
\mst_{r-1}^{-1} \mst_{r-2}^{-1} \cdots \mst_{i_k}^{-1} \mst_1^* \cdots \mst_{j_k-1}^* \mst_{j_k}^* \right)
\msc_{\widetilde{J}} \msn \msn^*,
\end{equation*}
where
\begin{itemize}

\item[{\rm (1)}] $\msn$ is a monomial in normal form in the generators $\mst_1, \mst_2, \ldots,  \mst_{r-1}$;

\item[{\rm (2)}] $\msn^*$ is a monomial in normal form in the generators $ \mst_1^*, \mst_2^*, \ldots,  \mst_{s-1}^*$;

\item[{\rm (3)}] the product is arranged from $k = 1$ to $k = a$ from left to right;

\item[{\rm (4)}]$\msc_{\wt I}$ is the product of $\msc_i$ over $i\in\wt{I}$ in increasing order,
and $\msc_{\wt J}$ is defined similarly.

\item[{\rm (5)}] Moreover, it is required that if $\msn =  \msp_1 \msp_2 \cdots \msp_{r-1}$ and
$\msp_i =  \mst_i^{-1}  \mst_{i-1}^{-1} \cdots  \mst_j^{-1}$,
then $\widetilde{\sigma}^{-1}(i_1) < \cdots < \widetilde{\sigma}^{-1}(i_a)$
where $\widetilde{\sigma} = \sigma_1 \sigma_2 \cdots \sigma_{r-1}$
and $\sigma_i$ is the cycle $(i+1 \; i \; i-1 \; \cdots \; j+1 \; j)$.
\end{itemize}
\end{definition}

\begin{theorem}\label{basis}
The set $\mathcal{B}$ of monomials $\msm$ in normal form is a basis of $\BC$ over $\mathbb{C}(q)$.
\end{theorem}

\begin{proof}
As in \textbf{Step 1} of \cite[Thm.~5.1]{JK}, there
are relations analogous to those labeled (1)-(8), except that words
of strictly smaller length need to be added on one side of each
equality. By using induction on the length of words, we can verify
that the set $\mathcal {B}$ spans $\mathsf{BC}_{r,s}(q)$ over
$\C(q)$.

Since $\mathsf{BC}_{r,s}(\mathcal{R})$ is a finitely generated
torsion-free $\mathcal{R}$-module, it is free over $\mathcal{R}$.
Now by a standard argument in abstract algebra (cf. \cite[Chap.~4,
Thm.~5.11]{H}), it follows that $\mathcal{B}$ is
linearly independent over $\C(q)$.
\end{proof}

\begin{corollary} \label{cor:dim of BC}
The dimension of $\BC$ over $\C(q)$ is $(r+s)!\,2^{r+s}$.
\end{corollary}

\begin{proof}
This follows from Theorem \ref{basis}, \textbf{Step 2} in the proof
\cite[Thm.~5.1]{JK},  and  \cite[Lem.~1.7]{KM}.
\end{proof}

\begin{corollary}\label{subHC}
The subalgebra of $\BC$ generated by $\mst_1,\ldots, \mst_{r-1},\msc_1, \ldots, \msc_r$ (resp. by $\mst_1^*,\ldots, \mst_{s-1}^*$ and $\msc_1^*, \ldots, \msc_s^*$) is isomorphic to the finite Hecke-Clifford superalgebra $\Heckr$ (resp. to $\Hecks$).  The subalgebra generated by $\mst_1,\ldots, \mst_{r-1},\mst_1^*,\ldots, \mst_{s-1}^*$, and $\mse$ is isomorphic
to the quantum walled Brauer algebra $\mathsf{H}_{r,s}^0(q)$ in {\rm \cite{KM}}.
\end{corollary}

\begin{proof}
  The first assertion follows from the fact that the set $\{\msc_{\tilde I} \msn \}$, as $\tilde I$ ranges over
  the subsets of $I=\{ 1, \cdots, r\}$ and $\msn$ ranges over the
  monomials in normal form in $\mst_1, \mst_2, \ldots,  \mst_{r-1}$,
   is a basis of $\mathsf{HC}_r(q)$ over $\C(q)$.
For the second assertion, one can show that the set $\mathcal{B'}$ consisting of the elements
$$ \left( \prod_{k=1,\ldots,a}^{\longrightarrow}
\mst_{j_k}^*  \mst_{j_k-1}^* \cdots  \mst_1^*  \mst_{i_k}^{-1} \cdots  \mst_{r-2}^{-1}  \mst_{r-1}^{-1} \mse
\mst_{r-1}^{-1} \mst_{r-2}^{-1} \cdots \mst_{i_k}^{-1} \mst_1^* \cdots \mst_{j_k-1}^* \mst_{j_k}^* \right) \msn \msn^*,$$
 where $\msn$ is a monomial in normal form in $\mst_1, \mst_2, \ldots,  \mst_{r-1}$
 satisfying the condition (5) in Definition \ref{BDef},
and $\msn^*$ is a monomial in  normal form in $ \mst_1^*, \mst_2^*, \ldots,  \mst_{s-1}^*$,
spans $\mathsf{H}_{r,s}^0(q)$ over $\C(q)$.
Since $\dim_{\C(q)} \mathsf{H}_{r,s}^0(q)=(r+s)!$ and $|\mathcal{B}'| \le (r+s)!$, the set $\mathcal{B}'$ is
a basis of $\mathsf{H}_{r,s}^0(q)$ over $\C(q)$.
\end{proof}

 We will frequently deduce properties of
$\mathsf{BC}_{r,s}(q)$ from the corresponding properties of
$\mathsf{BC}_{r,s}$ using the following well-known facts about
specialization, which we prove here for convenience.
\begin{lemma}\label{localring}
Suppose $\msR$ is a Noetherian local integral domain whose maximal ideal is generated by a single element $x\in \msR$. Let $\psi:\msA \rightarrow \msB$ be a homomorphism of finitely generated $\msR$-modules, and consider the corresponding induced homomorphism
\[
    \ol{\psi}: \msA /x\msA \rightarrow \msB/x\msB,  \qquad    \ol{\psi}(a + x\msA) = \psi(a) + x\msB.
\]
\begin{itemize}
\item[{\rm(i)}]  If $\ol{\psi}$   is surjective, then $\psi$ is surjective.
\item[{\rm(ii)}]  If $\msB$ is torsion free and $\ol{\psi}$ is injective, then $\psi$ is injective, and its cokernel is also torsion free.
\end{itemize}
\end{lemma}
\begin{proof}  (i) \  Let $\msC$ be the cokernel of $\psi$. The right exact sequence $\msA \stackrel{\psi}{\rightarrow}\msB \twoheadrightarrow \msC$ induces a right exact sequence $\msA /x\msA \stackrel{\ol \psi}{\rightarrow}\msB/x\msB\twoheadrightarrow \msC/x\msC$. By assumption, $\ol \psi$ is surjective, so $\msC/x\msC=0$. Thus $\msC=0$ by Nakayama's lemma.

(ii) Let $\msK$ be the kernel of $\psi$. If $k\in \msK$,  then $\psi(k) + x\msA$
    is in the kernel of $\ol \psi$, which is zero by assumption. Thus $k\in x\msA$,
    so $k=xa$ for some $a\in \msA$. Thus $x\psi(a)=\psi(k)=0$, so $\psi(a)=0$
    since $\msB$ is torsion free.  Thus,  $a\in \msK$, so $\msK=x\msK$. Again by Nakayama's lemma, $\msK=0$, so $\psi$ is injective.
   Finally,  choose an $(\msR/x\msR)$-basis $\mcX$ of $\msA/x\msA$
and extend it to a basis $\mcX \sqcup \mcY$ of $\msB/x\msB$  (here
we are identifying $\ol{\psi}(\mcX)$ with $\mcX$ by injectivity). By
lifting these basis elements arbitrarily to $\msA$ and $\msB$, we
obtain a commutative diagram
    \[\xymatrix@C=20mm{
        \msR^\mcX\ar[r]\ar[d]&\msR^{\mcX}\oplus \msR^\mcY\ar[d]\\
        \msA\ar[r]&\msB\ar@{->>}[r]&\msB/\msA
    }\]
    where the top two modules are free over $\msR$, and the vertical
maps induce isomorphisms of $(\msR/x\msR)$-vector spaces. By what
we've shown so far, $\msR^\mcX\rightarrow \msA$ is surjective and
$\msR^\mcX\oplus \msR^\mcY\rightarrow \msB$ is an isomorphism.
Therefore $\msB/\msA\cong \msR^\mcY$ is free, and in particular,  is
torsion free.
\end{proof}

We now show that $\mathsf{BC}_{r,s}(q)$ gives  the centralizer  of the action of $\mfU_q(\mfq(n))$ on $\V_q^{r,s}$.
We deduce this from the
corresponding result in the classical case, which is proven in \cite{JK}.

\begin{theorem}\label{centraliser}
Let  $\rho_{n,q}^{r,s}:  \BC \rightarrow\End_{\mfU_q(\mfq(n))}(\V_q^{r,s})$ be the representation
of $\BC$ coming from Theorem \ref{BCqact}.  Then $\rho_{n,q}^{r,s}$  is surjective,  and when  $n\geq r+s$, it is an isomorphism.
\end{theorem}

\begin{proof}
Let $\pi:\mathsf{BC}_{r,s} \stackrel{\sim}{\rightarrow}(\mathcal{R}/(q-1)\mcR)  \otimes_{\mathcal{R}}\mathsf{BC}_{r,s}(\mathcal{R})$ be the isomorphism established in Proposition \ref{rem:classical limit of BC(q)}.
Consider
the following elements of $\mfU_q(\mfq(n))$  for $i,j\in
\mathtt{I}=\{ \pm i \, | \, i=1,\ldots, n \}$ with $i\leq j$:
\[
    \widetilde \msu_{ij}=(q-1)^{-1}\begin{cases}
        \msu_{ij}-1&\text{if }i=j,\\
        \msu_{ij}&\text{if }i\neq j.
    \end{cases}
\]
Let
\[
    \widetilde U=\sum_{i\leq j}\widetilde \msu_{ij}\ot E_{ij} \in\mfU_q(\mfq(n))\ot_{\C}\End_{\C}(\V).
\]
By (\ref{Sfor2}), the action of $\widetilde \msu_{ij}$ on $\V_q$ lies in $\End_{\mcR}(\V_{\mcR})\subseteq\End_{\C(q)}(\V_q)$. Moreover, under the surjection $\End_{\mcR}(\V_{\mcR})\twoheadrightarrow \End_{\C}(\V)$ given by evaluation at $q=1$, the action of $\widetilde \msu_{ij}$ maps to the action of the following element of $\mfq(n)$:
\[
    u_{ij}=\begin{cases}
       (-1)^{|i|} \msE_{ii}^0&\quad \text{if } \ i=j,\\
        (-1)^{|i|}2\msE_{(-1)^{|j|}j,(-1)^{|i|}i}^{1-\delta_{|i|,|j|}}&\quad\text{if }\ i<j.
    \end{cases}
\]
Similarly, by (\ref{S*for}), the action of $\widetilde \msu_{ij}$ on $\V_q^*$ lies in $\End_{\mcR}(\V_{\mcR}^*)$ and maps to the action of  $u_{ij}$ on $\V^*$ by evaluation at $q=1$.
Finally, since the coproduct on $\mfU_q(\mfq(n))$ sends
\[
    \Delta(\widetilde U)=\widetilde U^{13}+\widetilde U^{23}+(q-1)\widetilde U^{13}\widetilde U^{23},
\]
the corresponding statements extend to the action of $\widetilde \msu_{ij}$ on $\V_q^{r,s}$.

Now let $\End_{\widetilde U}(\V_{\mcR}^{r,s})$ denote the space of endomorphisms in $\End_{\mcR}(\V_{\mcR}^{r,s})$ which supercommute with the action of $\widetilde \msu_{ij}$ for all $i\leq j$. We will show that the $\mcR$-module homomorphism \[ \psi:\mathsf{BC}_{r,s}(\mcR) \lra \End_{\widetilde{U}}(\V_{\mcR}^{r,s}) \] is surjective, and an isomorphism if $n\ge r+s$.
Note that if $X\in\End_{\mcR}(\V_{\mcR}^{r,s})$ is such that $(q-1)X$ supercommutes with $\msu_{ij}$, then $X$ also supercommutes with $\msu_{ij}$. Therefore,  the induced homomorphism
\[
    (\mcR/(q-1)\mcR) \otimes_{\mathcal{R}}\End_{\widetilde U}(\V_{\mcR}^{r,s})\rightarrow\End_{\mathbb{C}}(\V^{r,s})
\]
is injective. Moreover since the elements $\{u_{ij}\mid i\leq j\}$ generate $\mfq(n)$, this map factors through $\End_{\mfq(n)}(\V^{r,s})$. We obtain the following diagram.
\[\xymatrix@C=20mm{
    \mathsf{BC}_{r,s} \ar[r]^-{\pi}\ar@{->>}[dr]&(\mcR/(q-1)\mcR) \ot_{\mathcal R}\BCloc \ar[r]^-{\id \ot \psi}&
        (\mcR/(q-1)\mcR)  \otimes_{\mcR}\End_{\widetilde U}(\V_{\mcR}^{r,s})\ar[dl]\ar@{^{(}->}[d]\\
    &\End_{\mfq(n)}(\V^{r,s})\ar@{^{(}->}[r]&\End_{\mathbb{C}}(\V^{r,s})
}\]
Now Theorem 3.5 of \cite{JK} shows that the homomorphism $\rho_n^{r,s}: \mathsf{BC}_{r,s}\rightarrow\End_{\mfq(n)}(\V^{r,s})^{\rm op}$ given by the $\mathsf{BC}_{r,s}$-module action is surjective, and also injective for $n\geq r+s$. It follows that $(\mcR/(q-1)\mcR) \otimes_{\mcR} \End_{\widetilde{U}}(\V_{\mcR}^{r,s}) \rightarrow \End_{\mfq(n)}(\V^{r,s})$ is an isomorphism for all $n$, so $(\mcR/(q-1)\mcR) \otimes_{\mcR} \psi$ is surjective for all $n$ and injective for $n\geq r+s$. Since $\End_{\widetilde U}(\V_{\mcR}^{r,s})$ is torsion free, we conclude by Lemma \ref{localring} that
\[
    \mathsf{BC}_{r,s}(\mcR)\rightarrow\End_{\widetilde U}(\V_{\mcR}^{r,s})
\]
is surjective for all $n$ and injective for $n\geq r+s$. Finally,  since the $\widetilde \msu_{ij}$ generate $\mfU_q(\mfq(n))$, we have
\[
    \C(q)\ot_{\mathcal R}\End_{\widetilde U}(\V^{r,s})=\End_{\mfU_q(\mfq(n))}(\V_q^{r,s}).
\]
Therefore,  tensoring by $\C(q)$, we obtain the desired result.
\end{proof} \smallskip

\begin{remark}\label{R:open}  The following question is left open: \  Does $\mfU_q(\mfq(n))$ surject onto $\End_{\BC}(\V_q^{r,s})$?
\end{remark}

\vskip3mm
\section{The $(r,s)$-bead tangle algebras $\BT_{r,s}(q)$} \label{sec:bead tangle algebra}

In this section, we introduce a diagrammatic realization of the quantum walled Brauer-Clifford superalgebra $\BC$
 given in Definition \ref{def:BCq}.

\begin{definition}\label{def:tangle}
  An \emph{$(r,s)$-bead tangle} is a portion of a planar knot diagram
  in a rectangle $R$ 
  with the following conditions: \begin{itemize}

\item[{\rm (1)}] The top and bottom boundaries of $R$ each have $r+s$ vertices in some standard position.

\item[{\rm (2)}] There is a vertical wall that separates  the first $r$ vertices  from the last $s$ vertices on the top and bottom boundaries.

\item[{\rm (3)}] Each vertex must be connected to exactly one other vertex by an arc.

\item[{\rm (4)}]  Each arc may (or may not) have finitely many numbered beads. The bead numbers in the tangle start with $1$ and are distinct consecutive positive integers.

\item[{\rm (5)}] A \emph{vertical arc} connects a vertex on the top boundary  to a vertex on the bottom boundary of $R$, and it cannot
       cross the wall.   A \emph{horizontal arc} connects two vertices on  the same boundary  of $R$, and it must cross the wall.

\item[{\rm (6)}] An $(r,s)$-bead tangle may have finitely many loops.  \end{itemize}
\end{definition}

The following is
an example of $(3,2)$-bead tangle.

$$
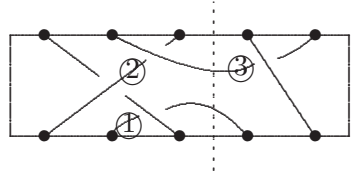
{\beginpicture
\setcoordinatesystem units <0.9cm,0.45cm>
\setplotarea x from 0 to 5.5, y from -1.5 to 4
\put{$\bullet$} at  1 3  \put{$\bullet$} at  1 0
\put{$\bullet$} at  2 3  \put{$\bullet$} at  2 0
\put{$\bullet$} at  3 3  \put{$\bullet$} at  3 0
\put{$\bullet$} at  4 3  \put{$\bullet$} at  4 0
\put{$\bullet$} at  5 3  \put{$\bullet$} at  5 0

\plot 0.5 3 5.5 3 /
\plot 5.5 3 5.5 0 /
\plot 5.5 0 0.5 0 /
\plot 0.5 0 0.5 3 /

\plot 2.5 2.25  1 0 /
\plot 2.8 2.7 3 3 /

\plot 2.2 1.2 3 0 /
\plot 1.8 1.8 1 3 /

\plot 4 3  5 0 /

\put{\textcircled{1}} at 2.25 0.3
\put{\textcircled{2}} at 2.3 1.9
\put{\textcircled{3}} at 3.9 2

\setdashes  <.4mm,1mm>
\plot 3.5 -1   3.5 4 /
\setsolid
\setquadratic

\plot 2 3   3.3 2 4.1 2.15 /
\plot 4.45 2.3  4.75 2.6 5 3 /

\plot 2.8 0.8  3.4 0.9 4 0 /
\plot 2 0  2.1 0.3   2.4 0.6 /

\endpicture}$$

We want to stress that an $(r,s)$-bead tangle is in the plane,
not in $3$-dimensional space.
We consider a bead as a point  on the arc.
Two $(r,s)$-bead tangles are \emph{regularly isotopic}
if they are related by a finite sequence of the Reidemeister moves
II, III together with isotopies fixing the boundaries of $R$.

Reidemeister move II :  \hskip1em
 
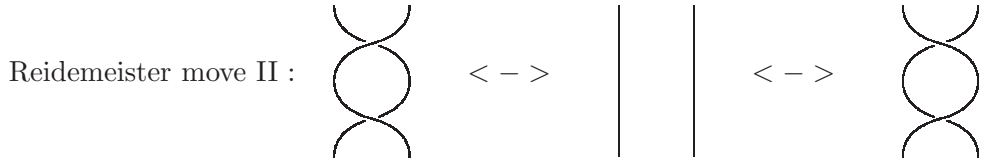
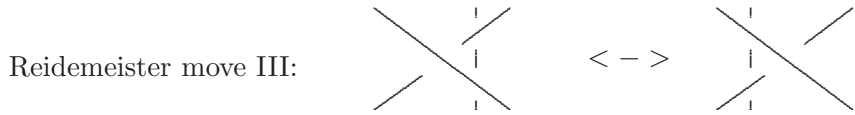
\xy
\vtwist~{(-5,10)}{(5,10)}{(-5,0)}{(5,0)};
\vcross~{(-5,0)}{(5,0)}{(-5,-10)}{(5,-10)};
\endxy  \hskip2em $<->$ \hskip2em \xy
(-5,10)*{}="T";
(5,10)*{}="B";
(-5,-10)*{}="T'";
(5,-10)*{}="B'";
"T";"T'" **\dir{-};
"B";"B'" **\dir{-};
\endxy \hskip2em $<->$ \hskip2em  \xy
\vcross~{(-5,10)}{(5,10)}{(-5,0)}{(5,0)};
\vtwist~{(-5,0)}{(5,0)}{(-5,-10)}{(5,-10)};
\endxy

\vskip1em

Reidemeister move III:
${\beginpicture
\setcoordinatesystem units <0.9cm,0.45cm>
\setplotarea x from 0 to 4, y from -1.5 to 4

\plot 1 2 3 -1 /
\plot  1 -1 1.7 0 /
\plot 2.3 .95 3 2 /
\plot 2.5 2 2.5 1.75 /
\plot 2.5 .75 2.5 .25 /
\plot 2.5 -0.75 2.5 -1 /
\endpicture}$
${\beginpicture
\setcoordinatesystem units <0.9cm,0.45cm>
\setplotarea x from 0 to 4, y from -1.5 to 4
\put{$ < - >$} at -0.3 .5

\plot 1 2 3 -1 /

\plot  1 -1 1.7 0 /
\plot 2.3 .95 3 2 /

\plot 1.5 2 1.5 1.75 /
\plot 1.5 .75 1.5 .25 /
\plot 1.5 -0.75 1.5 -1 /
\endpicture}$

We observe that there are isotopies fixing the boundaries of the rectangles  between the following tangles:

$
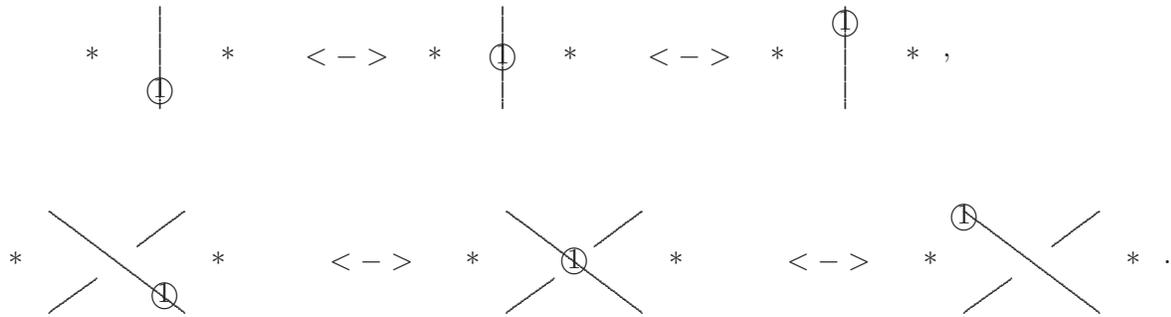
{\beginpicture
\setcoordinatesystem units <0.9cm,0.45cm>
\setplotarea x from 0 to 4, y from -1.5 to 4
\put{*} at 1 1.5
\put{*} at 3 1.5

\plot 2 3 2 0 /

\put{\textcircled{1}} at 2 0.5
\endpicture}$
${\beginpicture
\setcoordinatesystem units <0.9cm,0.45cm>
\setplotarea x from 0 to 4, y from -1.5 to 4
\put{$ < - > $} at -0.3 1.5
\put{*} at 1 1.5
\put{*} at 3 1.5

\plot 2 3 2 0 /

\put{\textcircled{1}} at 2 1.5
\endpicture}$
${\beginpicture
\setcoordinatesystem units <0.9cm,0.45cm>
\setplotarea x from 0 to 4, y from -1.5 to 4
\put{$ < - > $} at -0.3 1.5

\put{*} at 1 1.5
\put{*} at 3 1.5
\put{,} at 3.5 1.5

\plot 2 3 2 0 /

\put{\textcircled{1}} at 2 2.5
\endpicture}$

\hskip-5em${\beginpicture
\setcoordinatesystem units <0.9cm,0.45cm>
\setplotarea x from -1.5 to 5, y from -1.5 to 4
\put{*} at 0.5 1.5
\put{*} at 3.5 1.5

\plot 1 3 3 0 /
\plot 1 0 1.7 1 /
\plot 2.3 1.95  3 3 /

\put{\textcircled{1}} at 2.7 0.5
\endpicture}$
${\beginpicture
\setcoordinatesystem units <0.9cm,0.45cm>
\setplotarea x from -1.5 to 5, y from -1.5 to 4
\put{$ < - >$} at -1 1.5
\put{*} at 0.5 1.5
\put{*} at 3.5 1.5

\plot 1 3 3 0 /
\plot 1 0 1.7 1 /
\plot 2.3 1.95  3 3 /

\put{\textcircled{1}} at 2 1.5
\endpicture}$
${\beginpicture
\setcoordinatesystem units <0.9cm,0.45cm>
\setplotarea x from -1.5 to 5, y from -1.5 to 4
\put{$ < - >$} at -1 1.5

\put{*} at 0.5 1.5
\put{*} at 3.5 1.5
\put{.} at 4 1.5

\plot 1 3 3 0 /
\plot 1 0 1.7 1 /
\plot 2.3 1.95  3 3 /

\put{\textcircled{1}} at 1 2.8
\endpicture}$

Therefore, moving a bead along a non-crossing arc or an over-crossing arc  gives tangles that are regularly isotopic.
We want to emphasize that the following are \emph{not} regularly isotopic:

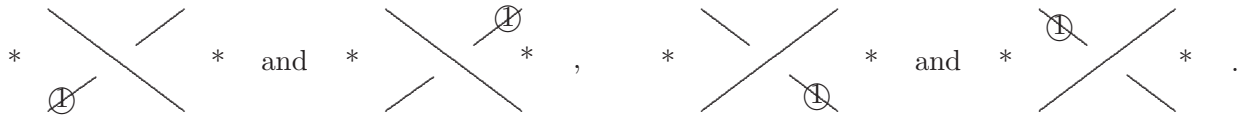
\begin{center}
\hskip-2.5em
${\beginpicture
\setcoordinatesystem units <0.9cm,0.45cm>
\setplotarea x from -0.5 to 4, y from -1.5 to 4

\put{*} at 0.5 1.5
\put{*} at 3.5 1.5

\plot 1 3 3 0 /
\plot 1 0 1.7 1 /
\plot 2.3 1.95  3 3 /

\put{\textcircled{1}} at 1.2 0.3
\endpicture}$
${\beginpicture
\setcoordinatesystem units <0.9cm,0.45cm>
\setplotarea x from -0.5 to 4, y from -1.5 to 4
\put{and} at -0.5 1.5

\put{*} at 0.5 1.5
\put{* \quad , } at 3.5 1.5

\plot 1 3 3 0 /
\plot 1 0 1.7 1 /
\plot 2.3 1.95  3 3 /

\put{\textcircled{1}} at 2.8 2.7
\endpicture}$
${\beginpicture
\setcoordinatesystem units <0.9cm,0.45cm>
\setplotarea x from -0.5 to 4, y from -1.5 to 4

\put{*} at 0.5 1.5
\put{*} at 3.5 1.5

\plot 1 0 3 3 /
\plot 1 3 1.7 1.95 /
\plot 2.3 1.05  3 0 /

\put{\textcircled{1}} at 2.7 0.4
\endpicture}$
${\beginpicture
\setcoordinatesystem units <0.9cm,0.45cm>
\setplotarea x from -0.5 to 4, y from -1.5 to 4
\put{and} at -0.5 1.5

\put{*} at 0.5 1.5
\put{* \quad .} at 3.5 1.5

\plot 1 0 3 3 /
\plot 1 3 1.7 1.95 /
\plot 2.3 1.05  3 0 /

\put{\textcircled{1}} at 1.3 2.45
\endpicture}$
\end{center}

We identify an $(r,s)$-bead tangle with its regular isotopy class, and
denote by $\widetilde{\BT}_{r,s}$ the set of $(r,s)$-bead tangles (up to regular isotopy).
The $(r,s)$-bead tangle in which there are even (resp. odd) number of
beads is regarded as \emph{even} (resp. \emph{odd}).

Now we define a multiplication on $\wt{\BT}_{r,s}$.
For $(r,s)$-bead tangles $d_1,d_2$,
we place $d_1$ under $d_2$ and identify the top row of $d_1$ with the bottom row of $d_2$.
We add the largest bead number in $d_1$ to each bead number in $d_2$, as we did for $(r,s)$-bead
diagrams, and then concatenate the tangles.
For example, if

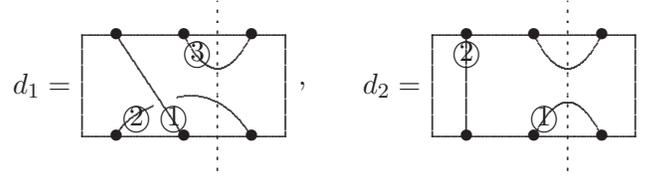
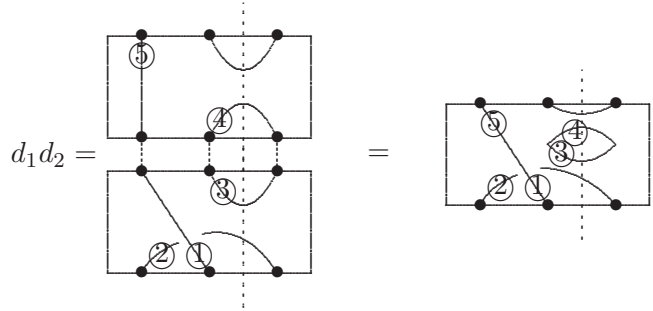
\begin{center}
${\beginpicture
\setcoordinatesystem units <0.9cm,0.45cm>
\setplotarea x from 0 to 4.5, y from -1.5 to 4
\put{$d_1=$} at -0.1 1.5
\put{$\bullet$} at  1 3  \put{$\bullet$} at  1 0
\put{$\bullet$} at  2 3  \put{$\bullet$} at  2 0
\put{$\bullet$} at  3 3  \put{$\bullet$} at  3 0
\put{,} at 3.75 1.5

\plot 0.5 3 3.5 3 /
\plot 3.5 3 3.5 0 /
\plot 3.5 0 0.5 0 /
\plot 0.5 0 0.5 3 /

\plot 1 3  2 0 /

\put{\textcircled{1}} at 1.85 0.5
\put{\textcircled{2}} at 1.3 0.5
\put{\textcircled{3}} at 2.2 2.4

\setdashes  <.4mm,1mm>
\plot 2.5 -1   2.5 4 /
\setsolid

\setquadratic
\plot 2 3  2.5 2 3 3 /


\plot 1 0  1.2 0.5  1.55 0.9 /
\plot 1.9 1.15 2.5 1 3 0 /
\endpicture}$
${\beginpicture
\setcoordinatesystem units <0.9cm,0.45cm>
\setplotarea x from 0 to 4.5, y from -1.5 to 4
\put{$d_2=$} at -0.1 1.5
\put{$\bullet$} at  1 3  \put{$\bullet$} at  1 0
\put{$\bullet$} at  2 3  \put{$\bullet$} at  2 0
\put{$\bullet$} at  3 3  \put{$\bullet$} at  3 0
\put{,} at 3.75 1.5

\plot 0.5 3 3.5 3 /
\plot 3.5 3 3.5 0 /
\plot 3.5 0 0.5 0 /
\plot 0.5 0 0.5 3 /

\plot 1 0 1 3 /


\put{\textcircled{1}} at 2.15 0.5
\put{\textcircled{2}} at 1 2.4

\setdashes  <.4mm,1mm>
\plot 2.5 -1   2.5 4 /
\setsolid

\setquadratic
\plot 2 3  2.5 2 3 3 /
\plot 2 0 2.5 1  3 0 /
\endpicture}$
\end{center}

then,
\begin{center}

${\beginpicture
\setcoordinatesystem units <0.9cm,0.45cm>
\setplotarea x from 0 to 4.5, y from -1.5 to 4
\put{$d_1 d_2=$} at -0.3 3.5
\put{$\bullet$} at  1 3  \put{$\bullet$} at  1 0
\put{$\bullet$} at  2 3  \put{$\bullet$} at  2 0
\put{$\bullet$} at  3 3  \put{$\bullet$} at  3 0

\plot 0.5 3 3.5 3 /
\plot 3.5 3 3.5 0 /
\plot 3.5 0 0.5 0 /
\plot 0.5 0 0.5 3 /

\plot 1 3  2 0 /

\put{\textcircled{1}} at 1.85 0.5
\put{\textcircled{2}} at 1.3 0.5
\put{\textcircled{3}} at 2.2 2.4


\setdashes  <.4mm,1mm>
\plot 2.5 -1   2.5 4 /
\setsolid

\setdashes  <.4mm,0.3mm>
\plot 1 3   1 4 /
\plot 2 3   2 4 /
\plot 3 3   3 4 /
\setsolid


\put{$\bullet$} at  1 7  \put{$\bullet$} at  1 4
\put{$\bullet$} at  2 7  \put{$\bullet$} at  2 4
\put{$\bullet$} at  3 7  \put{$\bullet$} at  3 4

\plot 0.5 7 3.5 7 /
\plot 3.5 7 3.5 4 /
\plot 3.5 4 0.5 4 /
\plot 0.5 4 0.5 7 /

\plot 1 4 1 7 /


\put{\textcircled{4}} at 2.15 4.5
\put{\textcircled{5}} at 1 6.4


\setdashes  <.4mm,1mm>
\plot 2.5 3   2.5 8 /
\setsolid

\put{$=$} at 4.5 3.5
\put{$\bullet$} at  6 5  \put{$\bullet$} at  6 2
\put{$\bullet$} at  7 5  \put{$\bullet$} at  7 2
\put{$\bullet$} at  8 5  \put{$\bullet$} at  8 2

\plot 5.5 5 8.5 5 /
\plot 8.5 5 8.5 2 /
\plot 8.5 2 5.5 2 /
\plot 5.5 2 5.5 5 /

\plot 6 5  7 2 /

\put{\textcircled{1}} at 6.85 2.5
\put{\textcircled{2}} at 6.3 2.5
\put{\textcircled{3}} at 7.2 3.5
\put{\textcircled{4}} at 7.4 4.15
\put{\textcircled{5}} at 6.2 4.4

\setdashes  <.4mm,1mm>
\plot 7.5 1   7.5 6 /
\setsolid

\setquadratic
\plot 2 7  2.5 6 3 7 /
\plot 2 4 2.5 5  3 4 /

\setquadratic
\plot 2 3  2.5 2 3 3 /

\plot 1 0  1.2 0.5  1.55 0.9 /
\plot 1.9 1.15 2.5 1 3 0 /

\setquadratic
\plot 7 5  7.5 4.7 8 5 /

\plot 6 2  6.2 2.5  6.55 2.9 /
\plot 6.9 3.1 7.5 2.8 8 2 /

\plot 7 3.8 7.5 3.3 8 3.8 /
\plot 7 3.8 7.5 4.3 8 3.8 /
\endpicture}$
\end{center}

We observe $\wt{\BT}_{r,s}$ is closed under this product,  and
it is $\Z_2$-graded. The shape of the arcs are the
same in $(d_1d_2)d_3$ and $d_1(d_2d_3)$. Moreover, the locations of the
beads and the bead numbers are also the same in $(d_1d_2)d_3$ and
$d_1(d_2d_3)$. It follows that the multiplication on $\wt{\BT}_{r,s}$
is associative. Hence $\wt{\BT}_{r,s}$ is a monoid with  identity element
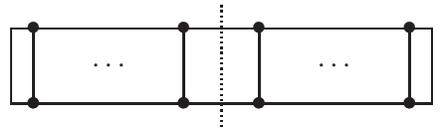
\[
\xy
(0,0)*{\bullet};
(0,-10)*{\bullet};
(0,0.5)*{};(0,-10)*{\bullet}**\dir{-};
(20,0)*{\bullet}; (20,-10)*{\bullet};
(20,0.5)*{}; (20,-10)*{\bullet}**\dir{-};
(30,0)*{\bullet}; (30,-10)*{\bullet};
(30,0.5)*{}; (30,-10)*{\bullet}**\dir{-};
(50,0)*{\bullet}; (50,-10)*{\bullet};
(50,0.5)*{}; (50,-10)*{\bullet}**\dir{-};
(25,3)*{}; (25,-13)*{}**\dir{.};
(10,-5)*{\cdots};(40,-5)*{\cdots};
(-3,0)*{};(53,0)*{}**\dir{-};
(-3,0)*{};(-3,-10)*{}**\dir{-};
(-3,-10)*{};(53,-10)*{}**\dir{-};
(53,-10)*{};(53,0)*{}**\dir{-};
(55,-7)*{.};
\endxy
\]

 For $1\leq i \leq r-1, \ 1 \leq j
\leq s-1$, $1 \leq k \leq r$, $1 \leq l \leq s$, we define the
following $(r,s)$-bead tangles:

$
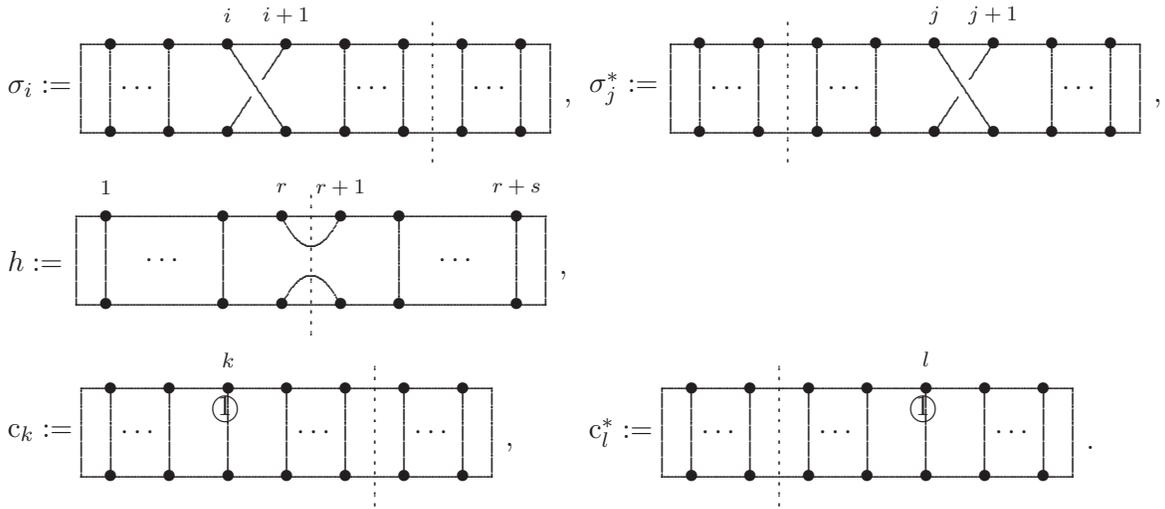
{\beginpicture
\setcoordinatesystem units <0.78cm,0.39cm>
\setplotarea x from 0 to 9, y from -1 to 4
\put{$\sigma_i:= $} at -0.2 1.5

\plot 0.5 0 0.5 3 /
\plot 0.5 3 8.5 3 /
\plot 8.5 3 8.5 0 /
\plot 8.5 0 0.5 0 /

\put{$\bullet$} at  1 0  \put{$\bullet$} at  1 3
\put{$\bullet$} at  2 0  \put{$\bullet$} at  2 3
\put{$\bullet$} at  3 0  \put{$\bullet$} at  3 3
\put{$\bullet$} at  4 0  \put{$\bullet$} at  4 3
\put{$\bullet$} at  5 0  \put{$\bullet$} at  5 3
\put{$\bullet$} at  6 0  \put{$\bullet$} at  6 3
\put{$\bullet$} at  7 0  \put{$\bullet$} at  7 3
\put{$\bullet$} at  8 0  \put{$\bullet$} at  8 3
\put{$\cdots$} at 1.5 1.5
\put{$\cdots$} at 5.5 1.5
\put{$\cdots$} at 7.5 1.5
\put{{\scriptsize$i$}} at 3 4
\put{{\scriptsize$i+1$}} at 4 4
\put{,} at 8.8 1
\plot 1 3 1 0 /
\plot 2 3 2 0 /
\plot 3 3 4 0 /
\plot 3 0 3.4 1.2 /
\plot  3.6 1.8 4 3 /
\plot 5 3 5 0 /
\plot 6 3 6 0 /
\plot 7 3 7 0 /
\plot 8 3 8 0 /
\setdashes  <.4mm,1mm>
\plot 6.5 -1   6.5 4 /
\setsolid
\endpicture}$
${\beginpicture \setcoordinatesystem units <0.78cm,0.39cm>
\setplotarea x from 0 to 9, y from -1 to 4
\put{$\sigma_j^*:= $} at -0.3 1.5

\plot 0.5 0 0.5 3 /
\plot 0.5 3 8.5 3 /
\plot 8.5 3 8.5 0 /
\plot 8.5 0 0.5 0 /

\put{$\bullet$} at  1 0
\put{$\bullet$} at  1 3
\put{$\bullet$} at 2 0
 \put{$\bullet$} at  2 3
 \put{$\bullet$} at  3 0
  \put{$\bullet$} at 3 3
\put{$\bullet$} at  4 0
\put{$\bullet$} at 4 3
\put{$\bullet$} at  5 0
\put{$\bullet$} at  5 3
\put{$\bullet$} at 6 0
\put{$\bullet$} at  6 3
 \put{$\bullet$} at  7 0
 \put{$\bullet$} at 7 3
\put{$\bullet$} at  8 0
 \put{$\bullet$} at 8 3

\put{$\cdots$} at 1.5 1.5
\put{$\cdots$} at 3.5 1.5
\put{$\cdots$} at 7.5 1.5
\put{,} at 8.8 1
\put{{\scriptsize$j$}} at 5 4
\put{{\scriptsize$j+1$}} at 6 4

\plot 1 3 1 0 /
\plot 2 3 2 0 /
\plot 3 3 3 0 /
\plot 4 3 4 0 /
\plot 5 3 6 0 /
\plot 7 3 7 0 /
\plot 8 3 8 0 /

\plot 5 0 5.4  1.2 /
\plot 5.6 1.8  6 3 /

\setdashes  <.4mm,1mm>
\plot 2.5 -1   2.5 4 /
\setsolid
\endpicture}$

${\beginpicture
\setcoordinatesystem units <0.78cm,0.39cm>
\put{$h := $} at -0.2 1.5
\plot 0.5 0 0.5 3 /
\plot 0.5 3 8.5 3 /
\plot 8.5 3 8.5 0 /
\plot 8.5 0 0.5 0 /
\put{$\bullet$} at  1 0  \put{$\bullet$} at  1 3
\put{$\bullet$} at  3 0  \put{$\bullet$} at  3 3
\put{$\bullet$} at  4 0  \put{$\bullet$} at  4 3
\put{$\bullet$} at  5 0  \put{$\bullet$} at  5 3
\put{$\bullet$} at  6 0  \put{$\bullet$} at  6 3
\put{$\bullet$} at  8 0  \put{$\bullet$} at  8 3

\put{$\cdots$} at 2 1.5
\put{$\cdots$} at 7 1.5
\put{,} at 8.8 1
\put{{\scriptsize $1$}} at 1 4
\put{{\scriptsize $r$}} at 4 4
\put{{\scriptsize $r+1$}} at 5 4
\put{{\scriptsize $r+s$}} at 8 4
\plot 1 3 1 0 /
\plot 3 3 3 0 /
\plot 8 3 8 0 /
\plot 6 3 6 0 /
\setdashes  <.4mm,1mm>
\plot 4.5 -1   4.5 4 /
\setsolid
\setquadratic
\plot 4 3 4.5 2 5 3 /
\plot 4 0 4.5 1 5 0  /
\endpicture}$

\vskip3mm
${\beginpicture
\setcoordinatesystem units <0.78cm,0.39cm>
\setplotarea x from 0 to 8, y from -1 to 4
\put{$\mathrm{c}_{k}:=$} at -0.2 1.5
\plot 0.5 0 0.5 3 /
\plot 0.5 3 7.5 3 /
\plot 7.5 3 7.5 0 /
\plot 7.5 0 0.5 0 /

\put{$\bullet$} at  1 0  \put{$\bullet$} at  1 3
\put{$\bullet$} at  2 0  \put{$\bullet$} at  2 3
\put{$\bullet$} at  3 0  \put{$\bullet$} at  3 3
\put{$\bullet$} at  4 0  \put{$\bullet$} at  4 3
\put{$\bullet$} at  5 0  \put{$\bullet$} at  5 3
\put{$\bullet$} at  6 0  \put{$\bullet$} at  6 3
\put{$\bullet$} at  7 0  \put{$\bullet$} at  7 3

\put{$\cdots$} at 1.5 1.5
\put{$\cdots$} at 4.5 1.5
\put{$\cdots$} at 6.5 1.5
\put{{\scriptsize $k$}} at 3 4
\put{,} at 7.8 1
\plot 1 3 1 0 /
\plot 2 3 2 0  /
\plot 3 3 3 0 /
\plot 4 3 4 0 /
\plot 5 3 5 0 /
\plot 6 3 6 0 /
\plot 7 3 7 0 /
\put{\textcircled{1}} at 2.95 2.3
\setdashes  <.4mm,1mm>
\plot 5.5 -1   5.5 4 /
\endpicture}$  \hskip2em
${\beginpicture \setcoordinatesystem units <0.78cm,0.39cm>
\setplotarea x from 0 to 7.5, y from -1 to 4
\put{$\mathrm{c}_l^*:= $} at -0.2 1.5

\plot 0.5 0 0.5 3 /
\plot 0.5 3 7.5 3 /
\plot 7.5 3 7.5 0 /
\plot 7.5 0 0.5 0 /

\put{$\bullet$} at  1 0 \put{$\bullet$} at  1 3
\put{$\bullet$} at  2 0  \put{$\bullet$} at  2 3
\put{$\bullet$} at  3 0 \put{$\bullet$} at  3 3
\put{$\bullet$} at  4 0   \put{$\bullet$} at 4 3
\put{$\bullet$} at  5 0  \put{$\bullet$} at  5 3
\put{$\bullet$} at  6 0   \put{$\bullet$} at  6 3
\put{$\bullet$} at  7 0   \put{$\bullet$} at  7 3

\put{$\cdots$} at 1.5 1.5
\put{$\cdots$} at 3.5 1.5
\put{$\cdots$} at 6.5 1.5
\put{.} at 7.8 1
\put{{\scriptsize $l$}} at 5 4

\plot 1 3 1 0 /
\plot 2 3 2 0  /
\plot 3 3 3 0 /
\plot 4 3 4 0 /
\plot 5 3 5 0 /
\plot 6 3 6 0 /
\plot 7 3 7 0 /
\put{\textcircled{1}} at 4.95 2.3
\setdashes  <.4mm,1mm>
\plot 2.5 -1   2.5 4 /
\endpicture}$

From Reidemeister move II, we obtain the following elements in $\wt{\BT}_{r,s}$:

${\beginpicture
\setcoordinatesystem units <0.75cm,0.39cm>
\setplotarea x from 0 to 9, y from -1 to 4
\put{$\sigma_i^{-1}= $} at -0.7 1.5

\plot 0.5 0 0.5 3 /
\plot 0.5 3 8.5 3 /
\plot 8.5 3 8.5 0 /
\plot 8.5 0 0.5 0 /

\put{$\bullet$} at  1 0  \put{$\bullet$} at  1 3
\put{$\bullet$} at  2 0  \put{$\bullet$} at  2 3
\put{$\bullet$} at  3 0  \put{$\bullet$} at  3 3
\put{$\bullet$} at  4 0  \put{$\bullet$} at  4 3
\put{$\bullet$} at  5 0  \put{$\bullet$} at  5 3
\put{$\bullet$} at  6 0  \put{$\bullet$} at  6 3
\put{$\bullet$} at  7 0  \put{$\bullet$} at  7 3
\put{$\bullet$} at  8 0  \put{$\bullet$} at  8 3
\put{$\cdots$} at 1.5 1.5
\put{$\cdots$} at 5.5 1.5
\put{$\cdots$} at 7.5 1.5
\put{{\scriptsize$i$}} at 3 4
\put{{\scriptsize$i+1$}} at 4 4
\put{,} at 8.8 1

\plot 1 3 1 0 /
\plot 2 3 2 0 /
\plot 4 3 3 0 /
\plot 3 3 3.4 1.8 /
\plot  3.6 1.2 4 0 /
\plot 5 3 5 0 /
\plot 6 3 6 0 /
\plot 7 3 7 0 /
\plot 8 3 8 0 /
\setdashes  <.4mm,1mm>
\plot 6.5 -1   6.5 4 /
\setsolid
\endpicture}$
${\beginpicture \setcoordinatesystem units <0.75cm,0.39cm>
\setplotarea x from 0 to 9, y from -1 to 4
\put{$(\sigma_j^*)^{-1}= $} at -0.8 1.5
\plot 0.5 0 0.5 3 /
\plot 0.5 3 8.5 3 /
\plot 8.5 3 8.5 0 /
\plot 8.5 0 0.5 0 /

\put{$\bullet$} at  1 0 \put{$\bullet$} at  1 3
\put{$\bullet$} at 2 0   \put{$\bullet$} at  2 3
 \put{$\bullet$} at  3 0  \put{$\bullet$} at 3 3
\put{$\bullet$} at  4 0   \put{$\bullet$} at 4 3
\put{$\bullet$} at  5 0   \put{$\bullet$} at  5 3
\put{$\bullet$} at 6 0    \put{$\bullet$} at  6 3
 \put{$\bullet$} at  7 0   \put{$\bullet$} at 7 3
\put{$\bullet$} at  8 0     \put{$\bullet$} at 8 3

\put{$\cdots$} at 1.5 1.5
\put{$\cdots$} at 3.5 1.5
\put{$\cdots$} at 7.5 1.5
\put{.} at 8.8 1
\put{{\scriptsize$j$}} at 5 4
\put{{\scriptsize$j+1$}} at 6 4

\plot 1 3 1 0 /
\plot 2 3 2 0 /
\plot 3 3 3 0 /
\plot 4 3 4 0 /
\plot 6 3 5 0 /
\plot 7 3 7 0 /
\plot 8 3 8 0 /

\plot 5 3 5.4  1.8 /
\plot 5.6 1.2  6 0 /

\setdashes  <.4mm,1mm>
\plot 2.5 -1   2.5 4 /
\setsolid
\endpicture}$

Now we consider the submonoid $\BT'_{r,s}$ of $\wt{\BT}_{r,s}$
generated by $\sigma_i^{\pm 1}, (\sigma_j^*)^{\pm 1}, h, \mathrm{c}_k,
\mathrm{c}_l^*$. We denote by $\BT'_{r,s}(q)$ the monoid algebra of
$\BT'_{r,s}$ over $\C(q)$ and define the algebra  $\BT_{r,s}(q)$ to be
the quotient of $\BT'_{r,s}(q)$ by the following relations (for allowable $i,j$):
\begin{equation} \label{def:I_in_BT}
\begin{aligned}
  & \sigma_i^{-1}=\sigma_i-(q-q^{-1}), \quad  && (\sigma^*_j)^{-1}=\sigma^*_j-(q-q^{-1}), \\
  & h \sigma_{r-1}h=h , \; h^2=0, \quad && h \sigma_1^* h =h, \; h \mathrm{c}_r h=0, \\
  &\mathrm{c}_i^2=-1, \; \mathrm{c}_i \mathrm{c}_j=-\mathrm{c}_j \mathrm{c}_i \  (i \neq j), \;  \mathrm{c}_i \mathrm{c}^*_j =-\mathrm{c}^*_j \mathrm{c}_i  ,  \quad && (\mathrm{c}_i^*)^2=1, \; \mathrm{c}_i^* \mathrm{c}_j^* = - \mathrm{c}_i^* \mathrm{c}_j^* \ (i \neq j).
\end{aligned}
\end{equation}

For simplicity, we identify  the coset of a diagram in $\BT_{r,s}(q)$ with the diagram itself. Note that
we get extra terms when a bead moves along an under-crossing arc.
That is, we have

${\beginpicture
\setcoordinatesystem units <0.9cm,0.45cm>
\setplotarea x from -0.5 to 4, y from -1.5 to 4

\plot 1 3 2 0 /
\plot 1 0 1.4 1.2 /
\plot 1.6 1.8  2 3 /

\put{\textcircled{1}} at 1.15 0.4
\endpicture}$
${\beginpicture
\setcoordinatesystem units <0.9cm,0.45cm>
\setplotarea x from -0.5 to 4, y from -1.5 to 4
\put{$ < - >$} at -1 1.5

\plot 2 3 1 0 /
\plot 1 3 1.4 1.8 /
\plot  1.6 1.2 2 0 /

\put{$+ \ (q-q^{-1})$} at 3.5 1.5

\plot 5 0 5 3 /
\plot 6 0 6 3 /

\put{\textcircled{1}} at 1.15 0.4
\put{\textcircled{1}} at 4.95 0.3
\endpicture}$

\hskip11em ${\beginpicture
\setcoordinatesystem units <0.9cm,0.45cm>
\setplotarea x from -0.5 to 4, y from -1.5 to 4
\put{$ < - >$} at -1 1.5
\plot 2 3 1 0 /
\plot 1 3 1.4 1.8 /
\plot  1.6 1.2 2 0 /

\put{$+ \ (q-q^{-1})$} at 3.5 1.5

\plot 5 0 5 3 /
\plot 6 0 6 3 /

\put{\textcircled{1}} at 1.95 2.6
\put{\textcircled{1}} at 4.95 0.3
\endpicture}$

\hskip11em ${\beginpicture
\setcoordinatesystem units <0.9cm,0.45cm>
\setplotarea x from -0.5 to 4, y from -1.5 to 4
\put{$ < - >$} at -1 1.5

\plot 1 3 2 0 /
\plot 1 0 1.4 1.2 /
\plot 1.6 1.8  2 3 /
\put{$+ \ (q-q^{-1})\Bigg ( \ $} at 3.5 1.5

\plot 5 0 5 3 /
\plot 6 0 6 3 /

\put{\textcircled{1}} at 1.95 2.6
\put{\textcircled{1}} at 4.95 0.3

\put{$-$} at 6.5 1.5
\plot 7 0 7 3 /
\plot 8 0 8 3 /
\put{\textcircled{1}} at 7.95 0.3
\put{$\Bigg )$} at 8.5 1.5
\endpicture}$,

\noindent
which is equivalent to
$ c_i \sigma_i=\sigma_i c_{i+1} +(q-q^{-1})
(c_i -c_{i+1}).$
Similarly, we obtain
$ c_{i+1} \sigma_{i}^{-1}=\sigma_i^{-1} c_i +(q-q^{-1}) (c_i-c_{i+1}).$
We call $\BT_{r,s}(q)$ the \emph{$(r,s)$-bead tangle algebra} or
simply the \emph{bead tangle algebra}.

In \cite{K}, Kauffman introduced the algebra of tangles
and showed that it is isomorphic to the {\it Birman-Murakami-Wenzl
algebra}. To show that $\BT_{r,s}(q)$ is isomorphic to $\BC$, we will
follow the outline of the argument given in \cite[Thm.~4.4]{K}.

Let $F'_{r,s}$ be the monoid generated by
$t_1^{\pm 1}, t_2^{\pm 1}, \ldots, t_{r-1}^{\pm 1}, (t_1^*)^{\pm 1}, (t_2^*)^{\pm 1}, \ldots, (t_{s-1}^*)^{\pm 1},$
 $e $, $c_1, c_2, \ldots, c_r$ and $c_1^*, c_2^*, \ldots, c_s^*$
with the following defining relations (for $i,j$ in the allowable  range):
\allowdisplaybreaks
\begin{align}
&t_i t_i^{-1}=t_i^{-1}t_i=1, \ \ \ && t^*_j(t_j^*)^{-1}=(t_j^*)^{-1}t_j^*=1, \label{rel:Frs-1}\\
& t_it_{i+1}t_i = t_{i+1} t_i t_{i+1}, \ \ \
&& t_i^*t^*_{i+1}t^*_i = t^*_{i+1} t^*_i t^*_{i+1},\\
& t_i t_j = t_jt_i  \quad (|i-j|>1), \ \ \  &&t^*_i t^*_j = t^*_jt^*_i  \quad (|i-j|>1), \\
& t_i t_j^*=t_j^*t_i, \ \ \ && \\
&e t_j=t_j e \quad (j \neq r-1), && e t_j^*=t_j^* e \quad (j \neq 1 ), \\
& e t_{r-1}^{-1}t_1^*e= e t_{r-1}^{-1}t_1^*e t_1^* t_{r-1}^{-1} = t_{r-1}^{-1} t_1^*e t_{r-1}^{-1}t_1^*e, && \label{rel:Frs-new}\\
& e t_{r-1}(t_1^*)^{-1}e= e t_{r-1} (t_1^*)^{-1} e (t_1^*)^{-1} t_{r-1} && \\
                 \nonumber &  \hskip5.5em   = t_{r-1}(t_1^*)^{-1}e t_{r-1} (t_1^*)^{-1} e, && \\
& e t_{r-1}^{-1}t_1^*e= e t_{r-1} (t_1^*)^{-1} e, && \label{rel:Frs-6} \\
&t_ic_i=c_{i+1}t_i, \ \ \ &&t^*_ic^*_i=c^*_{i+1}t^*_i, \label{rel:Frs-7} \\
&t_ic_j=c_{j}t_i \quad (j \neq i, i+1), \ \ \ &&t^*_ic^*_j=c^*_{j}t^*_i \quad (j \neq i, i+1),\\
& t_ic_j^*=c_{j}^*t_i, \ \ \ && t^*_i c_j=c_{j}t^*_i, \\
& c_r e=c_1^* e, && e c_r =e c_1^*, \\
&c_j e=e c_j \quad ( j \neq r), && c_j^* e =e c_j^* \quad (j \neq 1). \label{rel:Frs-11}
\end{align}

We define a monoid homomorphism $\varphi_{r,s}: F'_{r,s} \rightarrow
\BT'_{r,s}$ by
$$\varphi_{r,s}(t_i^{\pm 1})=\sigma_i^{\pm 1}, \ \ \varphi_{r,s}(( t_j^*)^{\pm 1})=(\sigma_j^*)^{\pm 1}, \ \
\varphi_{r,s}(e)=h, \ \ \varphi_{r,s}(c_i)=\mathrm{c}_i, \  \text{ and }\
\varphi_{r,s}(c_j^*)=\mathrm{c}_j^*.$$   By direct computations, one
can check that $\sigma_i^{\pm 1}, (\sigma_j^*)^{\pm 1}, h, \mathrm{c}_k,
\mathrm{c}_l^*$ satisfy the corresponding defining relations
\eqref{rel:Frs-1} - \eqref{rel:Frs-11} in $\BT'_{r,s}$. Moreover,
$\varphi_{r,s}$ is surjective.

\begin{theorem} \label{th:F'_and_BT'}
 The monoid $F'_{r,s}$ is isomorphic to $\BT'_{r,s}$ as monoids.
\end{theorem}

\begin{proof}
  It is suffices to show that $\varphi_{r,s}$ is injective.
  Assume that $\varphi_{r,s}(d) =\varphi_{r,s}(d') \in \BT'_{r,s}$.
  This means that $\varphi_{r,s}(d')$ can be obtained from
  $\varphi_{r,s}(d)$  by a finite sequence of Reidemeister moves II, III and vice versa.

If $\varphi_{r,s}(d')$ can be obtained from $\varphi_{r,s}(d)$
without moving beads, then modifying the proof of \cite[Thm.~
4.4]{K}, we can show that $d'$ can be obtained from $d$ using relations \eqref{rel:Frs-1} - \eqref{rel:Frs-6}.

We consider the various cases in which we need to move the beads.
\smallskip

{\bf Case 1:} A bead moves along a non-crossing arc. We have the
following five cases.

$
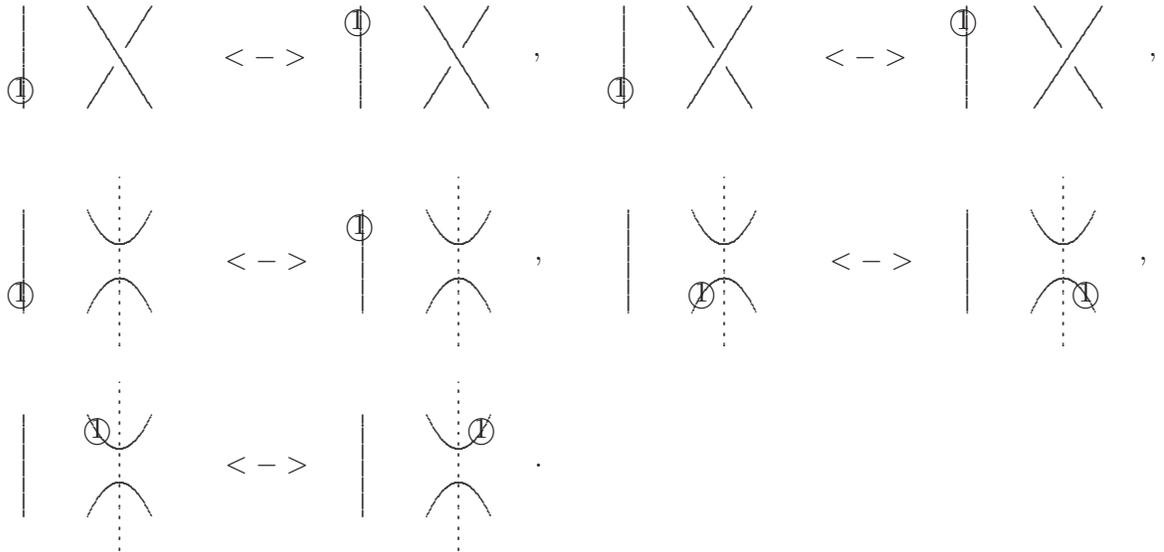
{\beginpicture
\setcoordinatesystem units <0.85cm,0.45cm>
\setplotarea x from 0 to 4, y from -1.5 to 4


\plot 1 3  1 0 /
\plot 2 3 3 0 /
\plot 2 0 2.4 1.2 /
\plot 2.6 1.8  3 3 /

\put{\textcircled{1}} at 0.95 0.5

\endpicture}$
 ${\beginpicture
\setcoordinatesystem units <0.85cm,0.45cm>
\setplotarea x from 0 to 4, y from -1.5 to 4
\put{$<->$} at -0.5 1.5
\put{,} at 3.75 1.5


\plot 1 3  1 0 /
\plot 2 3 3 0 /
\plot 2 0 2.4 1.2 /
\plot 2.6 1.8  3 3 /

\put{\textcircled{1}} at 0.95 2.5

\endpicture}$
${\beginpicture
\setcoordinatesystem units <0.85cm,0.45cm>
\setplotarea x from 0 to 4, y from -1.5 to 4


\plot 1 3  1 0 /
\plot 3 3 2 0 /
\plot 2 3 2.4 1.8 /
\plot  2.6 1.2 3 0 /

\put{\textcircled{1}} at 0.95 0.5

\endpicture}$
 ${\beginpicture
\setcoordinatesystem units <0.9cm,0.45cm>
\setplotarea x from 0 to 4, y from -1.5 to 4
\put{$<->$} at -0.5 1.5
\put{,} at 3.75 1.5


\plot 1 3  1 0 /
\plot 2 0 3 3 /
\plot 2 3 2.4 1.8 /
\plot  2.6 1.2 3 0 /

\put{\textcircled{1}} at 0.95 2.5

\endpicture}$

${\beginpicture
\setcoordinatesystem units <0.85cm,0.45cm>
\setplotarea x from 0 to 4, y from -1.5 to 4


\plot 1 3  1 0 /

\put{\textcircled{1}} at 0.95 0.5

\setdashes  <.4mm,1mm>
\plot 2.5 -1   2.5 4 /
\setsolid

\setquadratic
\plot 2 3  2.5 2 3 3 /
\plot 2 0 2.5 1 3 0 /

\endpicture}$
 ${\beginpicture
\setcoordinatesystem units <0.85cm,0.45cm>
\setplotarea x from 0 to 4, y from -1.5 to 4
\put{$<->$} at -0.5 1.5
\put{,} at 3.75 1.5


\plot 1 0 1 3 /

\setdashes  <.4mm,1mm>
\plot 2.5 -1   2.5 4 /
\setsolid

\setquadratic
\plot 2 3  2.5 2 3 3 /
\plot 2 0 2.5 1 3 0 /

\put{\textcircled{1}} at 0.95 2.5

\endpicture}$
${\beginpicture
\setcoordinatesystem units <0.85cm,0.45cm>
\setplotarea x from 0 to 4, y from -1.5 to 4


\plot 1 3  1 0 /

\put{\textcircled{1}} at 2.15 0.5

\setdashes  <.4mm,1mm>
\plot 2.5 -1   2.5 4 /
\setsolid

\setquadratic
\plot 2 3  2.5 2 3 3 /
\plot 2 0 2.5 1 3 0 /

\endpicture}$
 ${\beginpicture
\setcoordinatesystem units <0.85cm,0.45cm>
\setplotarea x from 0 to 4, y from -1.5 to 4
\put{$<->$} at -0.5 1.5
\put{,} at 3.75 1.5


\plot 1 0 1 3 /

\setdashes  <.4mm,1mm>
\plot 2.5 -1   2.5 4 /
\setsolid

\setquadratic
\plot 2 3  2.5 2 3 3 /
\plot 2 0 2.5 1 3 0 /

\put{\textcircled{1}} at 2.85 0.5

\endpicture}$

${\beginpicture
\setcoordinatesystem units <0.85cm,0.45cm>
\setplotarea x from 0 to 4, y from -1.5 to 4


\plot 1 3  1 0 /

\put{\textcircled{1}} at 2.15 2.5

\setdashes  <.4mm,1mm>
\plot 2.5 -1   2.5 4 /
\setsolid

\setquadratic
\plot 2 3  2.5 2 3 3 /
\plot 2 0 2.5 1 3 0 /

\endpicture}$
 ${\beginpicture
\setcoordinatesystem units <0.85cm,0.45cm>
\setplotarea x from 0 to 4, y from -1.5 to 4
\put{$<->$} at -0.5 1.5
\put{.} at 3.75 1.5


\plot 1 0 1 3 /

\setdashes  <.4mm,1mm>
\plot 2.5 -1   2.5 4 /
\setsolid

\setquadratic
\plot 2 3  2.5 2 3 3 /
\plot 2 0 2.5 1 3 0 /

\put{\textcircled{1}} at 2.85 2.5
\endpicture}$

Observe that the above are equivalent to the following
relations for $j \neq i, i+1$, \  $l\neq r$, \  $m \neq 1$,  and allowable values of $k$:
\begin{align*}
 & \mathrm{c}_j \sigma_i=\sigma_i \mathrm{c}_j , \quad   &&\mathrm{c}_j \sigma^*_k=\sigma^*_k \mathrm{c}_j, \quad
 &&\mathrm{c}_j \sigma_i^{-1}=\sigma_i^{-1} \mathrm{c}_j, \quad
 && \mathrm{c}_j (\sigma^*_k)^{-1}=(\sigma^*_k)^{-1} \mathrm{c}_j,\\
&  \mathrm{c}^*_j \sigma^*_i=\sigma^*_i \mathrm{c}^*_j,  \quad && \mathrm{c}^*_j \sigma_k=\sigma_k \mathrm{c}^*_j,
  \quad  && \mathrm{c}^*_j (\sigma^*_i)^{-1}=(\sigma^*_i)^{-1} \mathrm{c}^*_j, \quad
  && \mathrm{c}^*_j \sigma_k^{-1}=\sigma_k^{-1} \mathrm{c}^*_j,\\
&\mathrm{c}_l h=h\mathrm{c}_l, \quad && \mathrm{c}_m^* h =h \mathrm{c}_m^*,
\quad && \mathrm{c}_r h = \mathrm{c}_1^* h, \quad && h \mathrm{c}_r = h \mathrm{c}_1^*.
\end{align*}
Hence $d$ and $d'$ are related in $F'_{r,s}$.

{\bf Case 2:} A bead moves along an over-crossing arc.  In the following two cases,

$
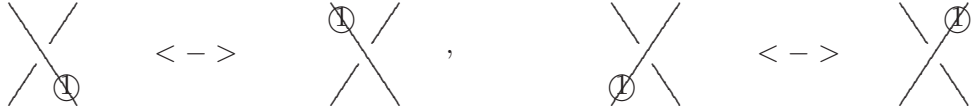
{\beginpicture
\setcoordinatesystem units <0.9cm,0.45cm>
\setplotarea x from 0 to 4, y from -1.5 to 4


\plot 2 3 3 0 /
\plot 2 0 2.4 1.2 /
\plot 2.6 1.8  3 3 /

\put{\textcircled{1}} at 2.85 0.5

\endpicture}$
 ${\beginpicture
\setcoordinatesystem units <0.9cm,0.45cm>
\setplotarea x from 0 to 4, y from -1.5 to 4
\put{$<->$} at 0 1.5
\put{,} at 3.75 1.5


\plot 2 3 3 0 /
\plot 2 0 2.4 1.2 /
\plot 2.6 1.8  3 3 /

\put{\textcircled{1}} at 2.15 2.5

\endpicture}$
${\beginpicture
\setcoordinatesystem units <0.9cm,0.45cm>
\setplotarea x from 0 to 4, y from -1.5 to 4


\plot 3 3 2 0 /
\plot 2 3 2.4 1.8 /
\plot  2.6 1.2 3 0 /

\put{\textcircled{1}} at 2.15 0.5

\endpicture}$
 ${\beginpicture
\setcoordinatesystem units <0.9cm,0.45cm>
\setplotarea x from 0 to 4, y from -1.5 to 4
\put{$<->$} at 0.5 1.5


\plot 3 3 2 0 /
\plot 2 3 2.4 1.8 /
\plot  2.6 1.2 3 0 /

\put{\textcircled{1}} at 2.85 2.5

\endpicture}$

\noindent the corresponding relations are equivalent to
\begin{align*}
\mathrm{c}_{i+1} \sigma_i =\sigma_i \mathrm{c}_i, \quad
\mathrm{c}^*_{i+1} \sigma^*_i =\sigma^*_i \mathrm{c}^*_i, \quad
\mathrm{c}_i \sigma_i^{-1}=\sigma_i^{-1}\mathrm{c}_{i+1}, \quad
\mathrm{c}^*_i
(\sigma^*_i)^{-1}=(\sigma^*_i)^{-1}\mathrm{c}^*_{i+1},
\end{align*}
which implies that  $d$ and $d'$ are related.

The remaining cases appear as a mixture of {\bf Case 1} and {\bf
Case 2}. For instance, a crossing of a horizontal and
a vertical arc is a combination of the following moves.

$
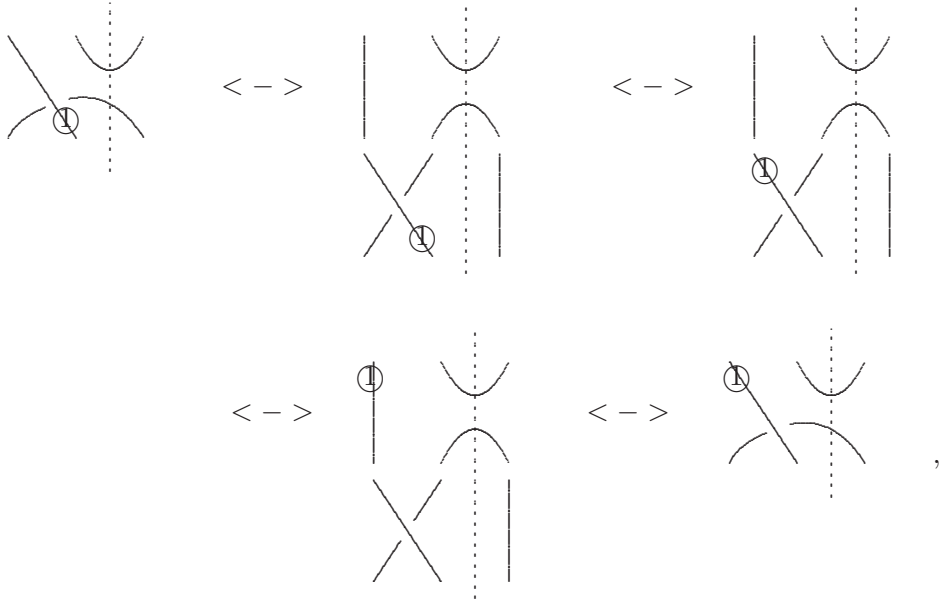
{\beginpicture
\setcoordinatesystem units <0.9cm,0.45cm>
\setplotarea x from 0 to 4, y from -1.5 to 4


\plot 1 3  2 0 /

\put{\textcircled{1}} at 1.85 0.5

\setdashes  <.4mm,1mm>
\plot 2.5 -1   2.5 4 /
\setsolid

\setquadratic
\plot 2 3  2.5 2 3 3 /


\plot 1 0  1.2 0.5  1.55 0.9 /
\plot 1.9 1.15 2.5 1 3 0 /
\endpicture}$
${\beginpicture
\setcoordinatesystem units <0.9cm,0.45cm>
\setplotarea x from 0 to 4.5, y from -1.5 to 4
\put{$<->$} at -0.5 1.5


\plot 1 0  1 3 /

\plot 1 -0.5 2 -3.5 /
\plot 1 -3.5 1.4 -2.3 /
\plot 1.6 -1.7  2 -0.5 /

\plot 3 -0.5 3 -3.5 /

\put{\textcircled{1}} at 1.85 -3

\setdashes  <.4mm,1mm>
\plot 2.5 -4  2.5 4 /
\setsolid

\setquadratic
\plot 2 3  2.5 2 3 3 /

\plot 2 0 2.5 1  3 0 /

\endpicture}$
${\beginpicture
\setcoordinatesystem units <0.9cm,0.45cm>
\setplotarea x from 0 to 4, y from -1.5 to 4
\put{$<->$} at -0.5 1.5


\plot 1 0  1 3 /

\plot 1 -0.5 2 -3.5 /
\plot 1 -3.5 1.4 -2.3 /
\plot 1.6 -1.7  2 -0.5 /

\plot 3 -0.5 3 -3.5 /

\put{\textcircled{1}} at 1.15 -1

\setdashes  <.4mm,1mm>
\plot 2.5 -4  2.5 4 /
\setsolid

\setquadratic
\plot 2 3  2.5 2 3 3 /

\plot 2 0 2.5 1  3 0 /

\endpicture}$

\vskip2em \hskip10em
${\beginpicture
\setcoordinatesystem units <0.9cm,0.45cm>
\setplotarea x from 0 to 4, y from -4.5 to 4
\put{$<->$} at -0.5 1.5


\plot 1 0  1 3 /

\plot 1 -0.5 2 -3.5 /
\plot 1 -3.5 1.4 -2.3 /
\plot 1.6 -1.7  2 -0.5 /

\plot 3 -0.5 3 -3.5 /

\put{\textcircled{1}} at 0.95 2.5

\setdashes  <.4mm,1mm>
\plot 2.5 -4  2.5 4 /
\setsolid

\setquadratic
\plot 2 3  2.5 2 3 3 /
\plot 2 0 2.5 1  3 0 /

\endpicture}$
${\beginpicture
\setcoordinatesystem units <0.9cm,0.45cm>
\setplotarea x from 0 to 4, y from -4.5 to 4
\put{$<->$} at -0.5 1.5


\plot 1 3  2 0 /

\put{\textcircled{1}} at 1.1 2.5

\setdashes  <.4mm,1mm>
\plot 2.5 -1   2.5 4 /
\setsolid

\setquadratic
\plot 2 3  2.5 2 3 3 /


\plot 1 0  1.2 0.5  1.55 0.9 /
\plot 1.9 1.15 2.5 1 3 0 /
\endpicture},$

\noindent  which can be written as $\mathrm{c}_2
\sigma_1 h = \sigma_1 \mathrm{c}_1 h =\sigma_1 h \mathrm{c}_1$.

In conclusion, when $\varphi_{r,s}(d')$ is obtained from
$\varphi_{r,s}(d)$ by moving beads, $d$ can be transformed to $d'$
by the corresponding relations in \eqref{rel:Frs-7} -
\eqref{rel:Frs-11}.
\end{proof}

Let $F'_{r,s}(q) = \C(q)F'_{r,s}$ be the associated monoid algebra,
and let $R$ be the two-sided ideal of   $F'_{r,s}(q)$ corresponding to the following
relations:
\begin{equation} \label{def:I_in_F}
\begin{aligned}
  & t_i^{-1}=t_i-(q-q^{-1}), \quad  && (t^*_j)^{-1}=t^*_j-(q-q^{-1}), \\
  & e t_{r-1}e=e , \; e^2=0, \quad && e t_1^* e =e, \; e c_r e=0,  \\
  &c_i^2=-1, \; c_i c_j=-c_j c_i  \  (i \neq j), \; c_i c^*_j =-c^*_j c_i. \quad && (c_i^*)^2=1, \; c_i^* c_j^* = - c_i^* c_j^*   \ (i \neq j).
\end{aligned}
\end{equation}
We consider $(t_i)^{\pm 1}, (t_j^*)^{\pm 1}, h$ as the \emph{even}
generators and $c_i, c_j^*$ as the \emph{odd} generators.

We denote  by $F_{r,s}(q)$ the quotient superalgebra
$F'_{r,s}(q)/R$.  For ease of notation, we also use $t_i,t_j^*, h, c_k$ and
$c_l^*$ for the generators of $F_{r,s}(q)$. We note that the
relations in \eqref{def:I_in_BT} correspond to the relations
in \eqref{def:I_in_F} via the map $\varphi_{r,s}$. By the definitions of
$F_{r,s}(q)$ and $\BT_{r,s}(q)$ and Theorem \ref{th:F'_and_BT'}, we
obtain the following corollary.

\begin{corollary}
The superalgebra $F_{r,s}(q)$ is isomorphic to $\BT_{r,s}(q)$ as associative superalgebras.
\end{corollary}

One can check that  relations \eqref{rel:Frs-1} -
\eqref{rel:Frs-11} and \eqref{def:I_in_F}
include the corresponding relations \eqref{def:BC(q)} if we map $\mst_i, \mst_j, \mse, \msc_k$
and $\msc_l^*$ to $t_i,t_j, h, c_k$ and $c_l^*$, respectively.
Using the relations in \eqref{def:I_in_F} and Remark \ref{rem:original relations in BC(q)},
we obtain that the relations corresponding to \eqref{rel:Frs-new} - \eqref{rel:Frs-6} are also satisfied in $\BC$.
It follows that $F_{r,s}(q)$ is isomorphic to $\BC$ as associative
superalgebras. Therefore, we obtain the following main result of this
section.

\begin{theorem} \label{th:BC and BT}
  The quantum walled Brauer-Clifford superalgebra $\BC$ is isomorphic to the
 $(r,s)$-bead tangle algebra  
   $\BT_{r,s}(q)$
  as  associative superalgebras.
\end{theorem}

\begin{corollary}
    The dimension of $\BT_{r,s}(q)$ over $\C(q)$ is $(r+s)!\,2^{r+s}$.
\end{corollary}

\begin{remark} Since $\mathsf{BC}_{r,s}$ is the classical limit of
$\mathsf{BC}_{r,s}(q)$, by Theorem \ref{th:BC and BD} and  Theorem
\ref{th:BC and BT}, we conclude that $\BD_{r,s}$ is the classical
limit of $\BT_{r,s}(q)$.
\end{remark}

\vskip3mm
\section{The $q$-Schur superalgebra of type $\msQ$ and its dual}

There are two equivalent ways to define the $q$-Schur algebra $\msS_q(n;\ell)$ associated to $\mfU_q(\mfgl(n))$:  either as the image of  $\mfU_q(\mfgl(n))$ in $\End_{\C(q)}((\C(q)^n)^{\otimes \ell})$ or as $\End_{\msH_{\ell}(q)}((\C(q)^n)^{\otimes \ell})$, where $\msH_{\ell}(q)$ is the Hecke algebra (the subalgebra of $\Heckl$ generated by the $\mst_i$, $i=1,\dots, \ell-1$).   Analogous definitions can be considered in our quantum super context, but we are not able to prove that they are equivalent.    Therefore, in order to develop a viable theory, we have settled on the following definition.

\begin{definition}\label{def:qsuperschur}
The  \emph{$q$-Schur superalgebra of type $\msQ$}, denoted $\msS_q(n;r,s)$, is $\End_{\BC}(\V_q^{r,s})$.
\end{definition}

Even when $s=0$, this superalgebra had not been studied until the recent paper \cite{DW}. In this case, it follows from \cite[Thm.~5.3]{Ol} that the next result holds, but we don't know if it is true for arbitrary  $s\ge 1$.

\begin{proposition}
$\msS_q(n;r,0)$ is equal to the image of $\mfU_q(\mfq(n))$ in $\End_{\C(q)}(\V_q^{r,0})$.
\end{proposition}

There is a third point of view on $q$-Schur algebras adopted for instance in \cite{Do}, which is as duals of certain homogeneous subspaces of quantum matrix algebras. Super analogues of quantum $\msG\msL_n$ were first introduced in \cite{Ma} and more general superalgebras were studied in \cite{HuZh},  where bases were constructed using quantum minors and indexed by standard bitableaux. In this section, we obtain similar results for a quantum matrix superalgebra of type $\msQ$.

Let
\[
    \delta_{i<j}=\begin{cases}1&\text{if }i<j,\\0&\text{otherwise}\end{cases}
\]
and $\delta_{i\pm j}=\delta_{ij}+\delta_{i,-j}$. Also recall that $\eps=q-q^{-1}$.
\begin{definition}
We denote by $\msA_q(n)$ the associative unital algebra over $\C(q)$ generated by $x_{ab}$ and $\bar x_{ab}$ for $1\leq a,b\leq n$, subject to the following relations for any $1\leq a,b,c,d\leq n$ with $a\leq c$:
\begin{eqnarray*}
q^{\delta_{ac}}x_{ab}x_{cd}&=&q^{\delta_{bd}}x_{cd}x_{ab}+\eps\delta_{b<d}\,x_{cb}x_{ad}+\eps\bar x_{cb}\bar x_{ad}, \\
    q^{\delta_{ac}}x_{ab}\bar x_{cd}&=&q^{-\delta_{bd}}\bar x_{cd}x_{ab}-\eps\delta_{d<b}\,\bar x_{cb}x_{ad},   \notag \\
    q^{\delta_{ac}}\bar x_{ab}x_{cd}&=&q^{\delta_{bd}}x_{cd}\bar x_{ab}+\eps\bar x_{cb}x_{ad}+\eps\delta_{b<d}\,x_{cb}\bar x_{ad},    \\
    q^{\delta_{ac}}\bar x_{ab}\bar x_{cd}&=&-q^{-\delta_{bd}}\bar x_{cd}\bar x_{ab}+\eps\delta_{d<b}\,\bar x_{cb}\bar x_{ad}.
\end{eqnarray*}
We define a $\Z$-grading on $\msA_q(n)$ by declaring each generator to have degree 1. We call $\msA_q(n)$  \emph{the quantum matrix superalgebra of type $\msQ$}.
\end{definition}

\begin{remark}
The quotient of $\msA_q(n)$ by the two-sided ideal generated by the odd elements is isomorphic to the quantum matrix algebra as presented for instance in Section 1.3 of  \cite{BrDK} with $v=q^{-1}$.
\end{remark}

The algebra $\msA_q(n)$ can be viewed as a $q$-deformation of the algebra of polynomial functions on the space $\mcM_n(\msQ)$ of $(2n\times 2n)$-matrices of type $\msQ$ inside $\mcM_{n|n}(\C)$.  A $q$-deformation of the algebra of polynomial functions on $\mcM_{n|n}(\C)$ was first given in \cite{Ma}.

\begin{lemma}
The algebra $\msA_q(n)$ is isomorphic to the unital  associative algebra over $\C(q)$ generated by elements $x_{ij}$ with $i,j\in {\tt I} =\{ \pm 1, \ldots, \pm n \}$, which satisfy the relations $x_{ij} = x_{-i,-j}$ and \begin{equation}
S^{23} X^{12} X^{13} = X^{13} X^{12} S^{23}
\end{equation}
 where $S^{23}$ is the same matrix used in Definition \ref{Uqq}.
\end{lemma}

\begin{proof}
This follows from relations \eqref{Rij} and \eqref{Rijkl} below and from the proof of Theorem \ref{Aqdual}.
\end{proof}

\begin{corollary}
The algebra $\msA_q(n)$ is a bialgebra with coproduct $\Delta$ given by \[ \Delta(x_{ij}) = \sum_{\stackrel{k=-n}{k\neq 0}}^n (-1)^{(|i|+|k|)(|j|+|k|)} x_{ik} \otimes x_{kj}. \]
\end{corollary}

\begin{theorem}\label{Aqdual}
Let $\msA_q(n,r)$ denote the degree $r$ component of $\msA_q(n)$. There is a vector space isomorphism $\msA_q(n,r)\stackrel{\sim}{\rightarrow}\End_{\Heckr}(\V_q^{\ot r})^*$.   Explicitly, let $\{E_{ij}^{^\vee} \}$ denote the basis of
$\End_{\C(q)}(\V_q)^*$ dual to the natural basis $\End_{\C(q)}(\V_q)$. Define a map $\msA_q(n,1)\rightarrow\End_{\C(q)}(\V_q)^*$ by $x_{ab}\rightarrow E_{ab}^{^\vee},\;\; \bar x_{ab}\rightarrow E_{a,-b}^{^\vee}$. This extends to the map $\msA_q(n,r)\rightarrow\End_{\Heckr}(\V_q^{\ot r})^*$ via the (super) identification
\[
    \left(\End_{\C(q)}(\V_q)^*\right)^{\ot r}\cong\End_{\C(q)}(\V_q^{\ot r})^*
\]
\end{theorem}
\begin{proof}
Let $\msF_q(n)$ denote the free algebra generated by $\varepsilon_{ij}$ for $i,j\in {\tt I}=\{\pm1,\pm2,\ldots,\pm n\}$, and let $\msF_q(n,r)$ denote the degree $r$ component of $\msF_q(n)$ where each generator has degree 1. Sending $\varepsilon_{ij}$ to $E_{ij}^{^\vee}$,
we obtain a vector space isomorphism $\msF_q(n,r)\cong\End_{\C(q)}(\V_q^{\ot r})^*$.  An element of $\End_{\C(q)}(\V_q^{\ot r})$ will lie in the subspace $\End_{\Heckr}(\V_q^{\ot r})$ if its coefficients satisfy certain relations. We can
obtain $\End_{\Heckr}(\V_q^{\ot r})^*$ by quotienting $\msF_q(n)$ by the same relations.

The generator $\msc_k$ of $\Heckr$ acts on the $k$th tensor factor via the map $\Iphi$ \eqref{eqn:Iphi},  and the generator $\mst_k$ of $\Heckr$ acts on the $k$th and $(k+1)$st factors via the map $PS$  \eqref{eqn:PS}.
We compute the supercommutator of each of these maps with an arbitrary endomorphism.
\begin{eqnarray*}
    \left[\Iphi,\sum_{ij}a_{ij}E_{ij}\right] &=&\sum_{ij}(-1)^{|i|}(a_{-i,-j}-a_{ij})E_{i,-j}.
\end{eqnarray*}
\begin{eqnarray}
    \left[PS,\sum_{ijkl}a_{ijkl}E_{ij}\ot E_{kl}\right]
        &=&\sum_{ijkl}a_{ijkl}\left[(-1)^{|i|+(|i|+|k|)(|i|+|j|)}q^{\delta_{i\pm k}(-1)^{|k|}}E_{kj}\ot E_{il}\right. \notag \\
        &&\;{}+\eps\delta_{i<k}\,E_{ij}\ot E_{kl}+\eps\delta_{-i<k}\,(-1)^{|i|+|k|+|i|+|j|}E_{-i,j}\ot E_{-k,l} \notag \\
        &&\;{}-(-1)^{|l|+(|j|+|l|)(|k|+|l|)}q^{\delta_{j\pm l}(-1)^{|j|}}E_{il}\ot E_{kj} \notag \\
        &&\;{}-\left.\eps\delta_{j<l}\,E_{ij}\ot E_{kl}-\eps\delta_{j<-l}\,(-1)^{|j|+|l|+|k|+|l|}E_{i,-j}\ot E_{k,-l}\right] \notag \\
        &=&\sum_{ijkl}\left[(-1)^{|i|(|j|+|k|)+|j||k|}q^{\delta_{i\pm k}(-1)^{|k|}}a_{ijkl}\right. \label{aijkl} \\
        &&\;{}+\eps\delta_{k<i}\,a_{kjil}-\eps\delta_{k<-i}\,(-1)^{|i|+|j|}a_{-k,j,-i,l} \notag \\
        &&\;{}-(-1)^{|j|(|l|+|i|)+|l||i|}q^{\delta_{j\pm l}(-1)^{|l|}}a_{klij} \notag \\
        &&\;{}-\left.\eps\delta_{j<l}\,a_{kjil}+\eps\delta_{-j<l}\,(-1)^{|i|+|j|}a_{k,-j,i,-l}\right]E_{kj}\ot E_{il}. \notag
\end{eqnarray}
Therefore,  $\End_{\Heckr}(\V_q^{\ot r})^*$ is the degree $r$ component of $\msA_q(n)'=\msF_q(n)/\langle R(ij),R(ijkl)\rangle$ by \cite[Lemma 2.3]{DDS2}, where we have factored out by the ideal generated by the elements
\begin{eqnarray}
    R(ij)&=&\varepsilon_{-i,-j}-\varepsilon_{ij}, \label{Rij}\\
    R(ijkl)&=& q^{\delta_{i\pm k}(-1)^{|k|}} (-1)^{(|i|+|j|)|l|} \varepsilon_{ij}\varepsilon_{kl}  \notag \\
    &&\;{} - q^{\delta_{j\pm l}(-1)^{|l|}}(-1)^{(|i|+|j|)|k|}\varepsilon_{kl}\varepsilon_{ij} \notag \\
        &&\;{}+\eps (-1)^{|k||i|+|k||l|+|j||l|} \left(\delta_{k<i}\,-\delta_{j<l}\right)  \varepsilon_{kj}\varepsilon_{il} \label{Rijkl} \\
        &&\;{}+\eps (-1)^{|k||i|+|k||l|+|j||l|+|k|+|l|} \left(\delta_{k<-i}\,\varepsilon_{-k,j}\varepsilon_{-i,l}-\delta_{-j<l}\,\varepsilon_{k,-j}\varepsilon_{i,-l}\right). \notag
\end{eqnarray}
The element  $R(ijkl)$ was obtained from the right-hand side of \eqref{aijkl} by multiplying it by $(-1)^{|i||j|}$ and by replacing each $a_{ijkl}$ by $(-1)^{(|i|+|j|)(|k|+|l|)} \varepsilon_{ij}\varepsilon_{kl} $. It remains to show that $\msA_q(n)' \cong \msA_q(n)$. Let $\msF_q(n)'$ be the free algebra generated by $x_{ab}$ and $\bar x_{ab}$ for $1\leq a,b\leq n$. For convenience, we define elements $x_{ij}$ and $\bar x_{ij}$ in $\msF_q(n)'$ for
all $i,j\in {\tt I}$ such that
\[
    x_{-i,j}=\bar x_{ij}=x_{i,-j}.
\]
Clearly there is an isomorphism $\msF_q(n)/\langle R(ij)\rangle\cong \msF_q(n)'$ sending $\varepsilon_{ij}$ to $x_{ij}$. Let $R(ijkl)'$ be the image of $R(ijkl)$ in $\msF_q(n)'$. Note that
\begin{eqnarray*}
    R(-i,-j,k,l)'
        &=&(-1)^{|l|(|i|+|j|)}q^{\delta_{i\pm k}(-1)^{|k|}}x_{ij}x_{kl}-(-1)^{|k|(|i|+|j|)}q^{\delta_{j\pm l}(-1)^{|l|}}x_{kl}x_{ij}\\
        &&\;{}+(-1)^{|k||i|+|k||l|+|j||l|+|k|+|l|} \eps\left(\delta_{k<-i}\,-\,\delta_{-j<l}\right)\bar x_{kj}\bar x_{il}\\
        &&\;{}+ (-1)^{|k||i|+|k||l|+|j||l|} \eps\left(\delta_{k<i}\,-\,\delta_{j<l}\right)x_{kj}x_{il}\\
        &=&R(ijkl)'.
\end{eqnarray*}
Similarly, using $q^{\delta_{i\pm k}(-1)^{|k|}}-q^{-\delta_{i\pm k}(-1)^{|k|}}=\delta_{i\pm k}(-1)^{|k|}\eps$,
we have
\begin{align*}
\allowdisplaybreaks
    (-1)^{|i|+|j|} & R(i,j,-k,-l)' = (-1)^{|l|(|i|+|j|)}q^{-\delta_{i\pm k}(-1)^{|k|}}x_{ij}x_{kl}-(-1)^{|k|(|i|+|j|)}q^{-\delta_{j\pm l}(-1)^{|l|}}x_{kl}x_{ij}\\
        &\;{} - (-1)^{|k||i|+|k||l|+|j||l|+|k|+|l|} \eps\left(1-\delta_{i<-k}\,-\,\delta_{i,-k}-1+\delta_{-l<j}\,+\,\delta_{-l,j}\right)\bar x_{kj}\bar x_{il}\\
        &\;{} - (-1)^{|k||i|+|k||l|+|j||l|}  \eps\left(1-\delta_{k<i}\,-\,\delta_{ik}-1+\delta_{j<l}\,+\,\delta_{jl}\right)x_{kj}x_{il}\\
        &=(-1)^{|l|(|i|+|j|)}q^{\delta_{i\pm k}(-1)^{|k|}}x_{ij}x_{kl}-\delta_{i\pm k}(-1)^{|k|+|l|(|i|+|j|)}\eps x_{ij}x_{kl}\\
        &\;{}-(-1)^{|k|(|i|+|j|)}q^{\delta_{j\pm l}(-1)^{|l|}}x_{kl}x_{ij}+\delta_{j\pm l}(-1)^{|k|(|i|+|j|)+|l|}\eps x_{kl}x_{ij}\\
        &\;{}+ (-1)^{|k||i|+|k||l|+|j||l|+|k|+|l|} \eps\left(\delta_{i<-k}\,-\,\delta_{-l<j} +\delta_{i,-k}-\delta_{-l,j} \right)\bar x_{kj}\bar x_{il}\\
        &\;{}+ (-1)^{|k||i|+|k||l|+|j||l|}  \eps\left(\delta_{k<i}\,-\,\delta_{j<l} + \delta_{i,k}-\delta_{j,l} \right)x_{kj}x_{il}\\
        &=R(ijkl)'.
\end{align*}
Observe also that $q^{\delta_{i\pm k}((-1)^{|k|}+(-1)^{|i|})}-1=\delta_{ik}(-1)^{|k|}q^{(-1)^{|k|}}\eps$. Therefore
\begin{eqnarray*}
\allowdisplaybreaks
 (-1)^{|j||l|}q^{\delta_{i\pm k}(-1)^{|i|}}R(ijkl)' +(-1)^{|i||k|}q^{\delta_{j\pm l}(-1)^{|l|}}R(klij)'\hspace{-70mm}\\
        &=&(-1)^{|i||l|}\left(q^{\delta_{i\pm k}((-1)^{|k|}+(-1)^{|i|})}-q^{\delta_{j\pm l}((-1)^{|j|}+(-1)^{|l|})}\right)x_{ij}x_{kl}\hspace{-45mm}\\
        &&\;{}+(-1)^{|k|(|i|+|l|)}q^{\delta_{i\pm k}(-1)^{|i|}}\eps\left(\delta_{k<i}\,-\,\delta_{j<l}\right)x_{kj}x_{il}\hspace{-35mm}\\
        &&\;{}+(-1)^{|k||-i|+|l||-k|}q^{\delta_{i\pm k}(-1)^{|i|}}\eps\left(\delta_{k<-i}\,-\,\delta_{-j<l}\right)\bar x_{kj}\bar x_{il}\hspace{-35mm}\\
        &&\;{}+(-1)^{(|i|+|l|)|j|}q^{\delta_{j\pm l}(-1)^{|l|}}\eps\left(\delta_{i<k}\,-\,\delta_{l<j}\right)x_{il}x_{kj}\hspace{-35mm}\\
        &&\;{}+(-1)^{|i||-j|+|j||-l|}q^{\delta_{j\pm l}(-1)^{|l|}}\eps\left(\delta_{i<-k}\,-\,\delta_{-l<j}\right)\bar x_{il}\bar x_{kj}\hspace{-25mm}\\
        &=&(-1)^{|k|(|i|+|l|)}q^{\delta_{i\pm k}(-1)^{|i|}}\eps\left(\delta_{k<i}\,+\,\delta_{ik}-\delta_{j<l}\right)x_{kj}x_{il}\hspace{-35mm}\\
        &&\;{}+(-1)^{(|i|+|l|)|j|}q^{\delta_{j\pm l}(-1)^{|l|}}\eps\left(\delta_{i<k}\,-\,\delta_{l<j}\,-\delta_{lj}\right)x_{il}x_{kj}\hspace{-35mm}\\
        &&\;{}+\eps\left(\delta_{i<-k}\,-\,\delta_{-j<l}\right)\ \times \hspace{-35mm}\\
        &&\;\;\;\left[(-1)^{|k||-i|+|l||-k|}q^{\delta_{i\pm k}(-1)^{|i|}}\bar x_{kj}\bar x_{il}
            +(-1)^{|i||-j|+|j||-l|}q^{\delta_{j\pm l}(-1)^{|l|}}\bar x_{il}\bar x_{kj}\right]\hspace{-35mm}\\
        &=&\eps\left(1-\delta_{i<k}-\delta_{j<l}\right) \ \times \hspace{-.35mm} \\
        &&\;\;\;  \left[(-1)^{|k|(|i|+|l|)}q^{\delta_{i\pm k}(-1)^{|i|}}x_{kj}x_{il}
            -(-1)^{(|i|+|l|)|j|}q^{\delta_{j\pm l}(-1)^{|l|}}x_{il}x_{kj}\right]\hspace{-35mm}\\
        &&{}+\eps\left(\delta_{i<-k}\,-\,\delta_{-j<l}\right) \ \times \hspace{-35mm}\\
        &&\;\; \left[(-1)^{|k||-i|+|l||-k|}q^{\delta_{i\pm k}(-1)^{|i|}}\bar x_{kj}\bar x_{il}
            +(-1)^{|i||-j|+|j||-l|}q^{\delta_{j\pm l}(-1)^{|l|}}\bar x_{il}\bar x_{kj}\right].
\end{eqnarray*}
On the other hand, note that if $\delta_1,\delta_2\in\{0,1\}$ then $(\delta_1-\delta_2)(1-\delta_1-\delta_2)=\delta_1-\delta_2-\delta_1^2+\delta_2^2=0$, so

\begin{align*}
    \left(1-\delta_{i<k} -\delta_{j<l}\right)R(kjil)' &= \left(1-\delta_{i<k}\,-\,\delta_{j<l}\right) \ \times \\
        &\quad \left[(-1)^{|l|(|k|+|j|)}q^{\delta_{i\pm k}(-1)^{|i|}}x_{kj}x_{il}-(-1)^{|i|(|k|+|j|)}q^{\delta_{j\pm l}(-1)^{|l|}}x_{il}x_{kj}\right]\\
        &\;{}+(-1)^{|i||k|+|i||l|+|j||l|+|i|+|l|}\eps\left(1-\delta_{i<k}\,-\,\delta_{j<l}\right)\left(\delta_{i<-k}\,-\,\delta_{-j<l}\right)\bar x_{ij}\bar x_{kl}.
\end{align*}
Similarly
\begin{eqnarray*}
    R(k,-j,i,-l)'&=&(-1)^{|-l|(|k|+|-j|)}q^{\delta_{i\pm k}(-1)^{|i|}}\bar x_{kj}\bar x_{il}-(-1)^{|i|(|k|+|-j|)}q^{\delta_{j\pm l}(-1)^{|l|}}\bar x_{il}\bar x_{kj}\\
        &&\;{}+(-1)^{|i||k|+|i||l|+|j||l| +|i|+|j|+|l|}\eps\left(1-\delta_{i<k}\,-\,\delta_{j<l}\right)\bar x_{ij}\bar x_{kl}\\
        &&\;{}-(-1)^{|i||k|+|i||l|+|j||l|+|j|}\eps\left(1-\delta_{i<-k}\,-\,\delta_{-j<l}\right)x_{ij}x_{kl},
\end{eqnarray*}
so
\begin{align*}
    -(-1)^{|j||l|+|j|+|i||k|}&\left(\delta_{i<-k}\,-\,\delta_{-j<l}\right)R(k,-j,i,-l)' =\left(\delta_{i<-k}\,-\,\delta_{-j<l}\right) \ \times\hspace{-60mm} \\
        &\;\;\;  \left[(-1)^{|k||-i|+|l||-k|}q^{\delta_{i\pm k}(-1)^{|i|}}\bar x_{kj}\bar x_{il}+(-1)^{|i||-j|+|j||-l|}q^{\delta_{j\pm l}(-1)^{|l|}}\bar x_{il}\bar x_{kj}\right]\\
        &\;{}+(-1)^{|-i||-l|}\eps\left(1-\delta_{i<k}\,-\,\delta_{j<l}\right)\left(\delta_{i<-k}\,-\,\delta_{-j<l}\right)\bar x_{ij}\bar x_{kl}.
\end{align*}
Thus,
\begin{eqnarray*}
    (-1)^{|j||l|}q^{\delta_{i\pm k}(-1)^{|i|}}R(ijkl)'+(-1)^{|i||k|}q^{\delta_{j\pm l}(-1)^{|l|}}R(klij)'\hspace{-80mm}\\
        &=&\eps(-1)^{|k||i|+|l||j|}\left(1-\delta_{i<k}\,-\,\delta_{j<l}\right)R(kjil)' \\
        &&-\eps(-1)^{|j|+|j||l|+|k||i|}\left(\delta_{i<-k}\,-\,\delta_{-j<l}\right)R(k,-j,i,-l)'.
\end{eqnarray*}
These dependencies amongst the $R(ijkl)'$ imply that
\[
    {\rm span}_{\C(q)}\{R(ijkl)'\mid i,j,k,l\in {\tt I}\}
        ={\rm span}_{\C(q)}\{R(ijkl)'\mid i,j,k,l\in {\tt I}\text{ with }0<i\leq k\}.
\]
Note that if $0<i\leq k$, then $R(ijkl)'$ simplifies to
\begin{eqnarray*}
    R(ijkl)'&=&q^{\delta_{i\pm k}}(-1)^{|j||l|}x_{ij}x_{kl}-q^{\delta_{j\pm l}(-1)^{|l|}}x_{kl}x_{ij}\\
        &&\;{}-(-1)^{|j||l|}\eps\delta_{j<l}\,x_{kj}x_{il} - (-1)^{|j||l|+|l|}\eps\delta_{-j<l}\,\bar x_{kj}\bar x_{il}.
\end{eqnarray*}
By considering the four possibilities for $(|j|,|l|)$, we obtain the four relations in the definition of $\msA_q(n)$. Thus
\[
    \msA_q(n)'=\msF_q(n)'/\langle R(ijkl)'\mid i,j,k,l\in {\tt I}\text{ with }0<i\leq k\rangle=\msA_q(n).
\]
\end{proof}

It is also possible to prove an analogue of Theorem \ref{Aqdual} when $s\neq 0$ using the coalgebra $\msA_q(n;r,s)$ that we define immediately below. Set $x_{ab}^*= x_{ab}$ and $\bar x_{ab}^*= \sqrt{-1}\, \bar x_{ab}$. The relations in the following definition are super analogues of those in Lemma 4.1 of \cite{DDS2}.

\begin{definition}\label{wtA} Abbreviate $x_{ab} \ot 1$ and $1 \ot x_{ab}^*$ in $\msA_q(n) \otimes_{\C(q)} \msA_{q^{-1}}(n)$ by $x_{ab}$ and $x_{ab}^*$, respectively.
Then $\wt{\msA}_q(n)$ is defined to be the quotient of $\msA_q(n) \otimes_{\C(q)} \msA_{q^{-1}}(n)$ by the two-sided ideal generated by the following: \\
\begin{eqnarray} \label{wtA1}
&&\sum_{e=1}^n q^{2e} \left( x_{eb} x_{ed}^* - \bar
{x}_{eb} \bar{x}_{ed}^* \right)  \; \text{for $b\neq d$}, \ \ \ \sum_{e=1}^n q^{2e} \left(x_{eb} \bar
{x}_{ed}^* -  \bar{x}_{eb} x_{ed}^* \right),\\
\label{wtA2}
&&\sum_{e=1}^n  \left( x_{ae} x_{ce}^* +  \bar{x}_{ae} \bar{x}_{ce}^* \right) \; \text{for $a \neq c$}, \ \ \  \sum_{e=1}^n  \left( x_{ae} \bar
{x}_{ce}^* - \bar{x}_{ae} x_{ce}^* \right)  ,\\
\label{wtA3}
&&\sum_{e=1}^n q^{2e-2b} \left(  x_{eb} x_{eb}^* -
 \bar{x}_{eb} \bar{x}_{eb}^* \right)
- \sum_{e=1}^n \left( x_{ae}x_{ae}^*
+ \bar{x}_{ae} \bar{x}_{ae}^*\right) .
\end{eqnarray}
\end{definition}


\begin{definition}
$\msA_q(n;r,s)$ is defined as the subspace of $\wt{\msA}_q(n)$ spanned by monomials in the generators of bidegree $(r,s)$; that is,  of degree $r$ in the generators $x_{ab},\bar{x}_{ab}$ and of degree $s$ in the generators $x_{ab}^*,\bar{x}_{ab}^*$.
\end{definition}

\begin{theorem}
There is a vector space isomorphism $\msA_q(n;r,s) \stackrel{\sim}{\rightarrow} \End_{\BC}(\V_q^{r,s})^*$.
\end{theorem}

\begin{proof}
Most of the necessary computations are already contained in the proof of Theorem \ref{Aqdual}. Recall that $\cap\cup = q^{-(2n+1)} \sum_{i,j\in {\tt I}} (-1)^{|i||j|} q^{2j(1-2|j|)} E_{ij} \otimes E_{ij}$. We only need to explain where the new relations \eqref{wtA1}-\eqref{wtA3} in Definition \ref{wtA} come from,  and for this we have to compute the following commutator: \begin{eqnarray*}
\left[ \cap\cup , \sum_{ijkl} a_{ijkl} E_{ij} \otimes E_{kl} \right] & = & q^{-(2n+1)} \sum_{ijl} \left( \sum_p a_{pjpl} q^{2p(1-2|p|)}  (-1)^{|p|+|p||j|} \right) (-1)^{|i||j|} E_{ij} \otimes E_{il} \\
& & - q^{-(2n+1)} \sum_{ikl} \left( \sum_p a_{ipkp} q^{2l(1-2|l|)} (-1)^{|p|+|k||p|} \right) (-1)^{|k||l|} E_{il} \otimes E_{kl}  \\
\end{eqnarray*}
This leads to the relation \[ \delta_{ik} \sum_p (-1)^{|p||l|+|j||l|} q^{2p(1-2|p|)} x_{pj} x_{pl}^*  =\delta_{jl}\sum_p (-1)^{|i||k|+|i||p|}  q^{2l(1-2|l|)} x_{ip} x_{kp}^* . \] The relations \eqref{wtA1}-\eqref{wtA3} can be deduced by considering the cases $i\neq k$ and $j=l$; \  $i=k$ and $j\neq l$; \  $i=k$ and $j=l$;  and also the various possibilities for the signs of $i,j,k,l$. 
Note that $x_{ab}^*=x_{-a,-b}^*$ and
$x_{a,-b}^*=-x_{-a,b}^*=\bar{x}_{ab}^*$
for $1 \le a,b \le n$, because of the relation $[\Phi^T, \sum_{ij} a_{ij}E_{ij}]=\sum_{ij}\big( a_{ij}-(-1)^{|i|+|j|}a_{-i,-j} \big)E_{-i,j}$.
\end{proof}

\begin{remark} The algebra $\msA_q(n;r,s)$ could possibly be used  to prove the open problem of showing  the surjectivity of  the map $\mfU_q(\mfq(n)) \rightarrow \End_{\BC}(\V_q^{r,s})$ (see Remark \ref{R:open}). The following line of reasoning was applied in \cite{DDS2} to establish a similar surjectivity result for
 $\mfU_q(\mfgl(n))$ over a quite general base ring.  First,  it might be possible to obtain a homomorphism $\tau: \msS_q(n;r',0) \lra \msS_q(n;r,s)$ for some $r'$ (possibly $r' =r+s$)  via some embedding of the mixed tensor space into $\C(n|n)^{\otimes r'}$. The surjectivity of the map $\mfU_q(\mfq(n)) \rightarrow \End_{\BC}(\V_q^{r,s})$ then would follow from the surjectivity of $\tau$, which is equivalent to the injectivity of $\tau^*: \msS_q(n;r,s)^* \lra \msS_q(n;r',0)^*$. As suggested by \cite{DDS2}, the injectivity of $\tau^*$ could perhaps be shown  by constructing bases of $\msA_q(n;r,s)$ and $\msA_q(n;r',0)$ using super analogues of bideterminants. For the quantum general linear supergroup, this was accomplished  in \cite{HuZh}. \end{remark}

\begin{remark} The algebra  $\msA_q(n)$ is a bialgebra, so it can be enlarged to a Hopf algebra, the so-called \emph{Hopf envelope} of $\msA_q(n)$.  This is explained in \cite{Ma} in the context of the quantum general linear supergroup attached to $\mfgl({m|n})$. Moreover,  it is proved in \cite{Ma} that the Hopf envelope of $\msA_q(m|n)$, the quantized algebra of functions on the space of super matrices of size $(m|n)\times (m|n)$, is isomorphic to the localization of $\msA_q(m|n)$ with respect to the quantum Berezinian, a super analogue of the quantum determinant. This localization is the quantized algebra of functions $\C_q[\mathsf{GL}_{m|n}]$. This raises the following question: is the Hopf envelope of $\msA_q(n)$ isomorphic to the localization of $\msA_q(n)$ with respect to an appropriate super version of the quantum determinant? Such a localization could be thought of as a quantized algebra of functions for the supergroup of type $\msQ_n$.  \end{remark}


\end{document}